\documentclass[microtype]{gtpart}
\usepackage[utf8]{inputenc}
\usepackage{color}
\usepackage{amstext}
\usepackage{amsthm}
\usepackage{amssymb}
\usepackage{graphicx}
\PassOptionsToPackage{normalem}{ulem}
\usepackage{ulem}

\makeatletter

\pdfpageheight\paperheight
\pdfpagewidth\paperwidth

\numberwithin{equation}{section}
\numberwithin{figure}{section}
\theoremstyle{plain}
\newtheorem{thm}{\protect\theoremname}[section]
\theoremstyle{plain}
\newtheorem{cor}[thm]{\protect\corollaryname}
\theoremstyle{plain}
\newtheorem{lem}[thm]{\protect\lemmaname}
\theoremstyle{plain}
\newtheorem{prop}[thm]{\protect\propositionname}
\theoremstyle{definition}
\newtheorem{defn}[thm]{\protect\definitionname}
\theoremstyle{remark}
\newtheorem{rem}[thm]{\protect\remarkname}

\usepackage{graphicx}

\usepackage{enumitem}

\gtart{}

\makeatother

\providecommand{\corollaryname}{Corollary}
\providecommand{\definitionname}{Definition}
\providecommand{\lemmaname}{Lemma}
\providecommand{\propositionname}{Proposition}
\providecommand{\remarkname}{Remark}
\providecommand{\theoremname}{Theorem}

\begin{document}
\global\long\def\F{\mathrm{\mathbf{F}} }%
\global\long\def\Aut{\mathrm{Aut}}%
\global\long\def\C{\mathbf{C}}%
\global\long\def\H{\mathcal{H}}%
\global\long\def\U{\mathsf{U}}%
\global\long\def\ext{\mathrm{ext}}%
\global\long\def\hull{\mathrm{hull}}%
\global\long\def\triv{\mathrm{triv}}%
\global\long\def\Hom{\mathrm{Hom}}%

\global\long\def\trace{\mathrm{tr}}%
\global\long\def\End{\mathrm{End}}%

\global\long\def\L{\mathcal{L}}%
\global\long\def\W{\mathcal{W}}%
\global\long\def\E{\mathbb{E}}%
\global\long\def\SL{\mathrm{SL}}%
\global\long\def\R{\mathbf{R}}%
\global\long\def\Pairs{\mathrm{PowerPairs}}%
\global\long\def\Z{\mathbf{Z}}%
\global\long\def\rs{\to}%
\global\long\def\A{\mathcal{A}}%
\global\long\def\a{\mathbf{a}}%

\global\long\def\b{\mathbf{b}}%
\global\long\def\df{\mathrm{def}}%
\global\long\def\eqdf{\stackrel{\df}{=}}%
\global\long\def\ZZ{\overline{Z}}%
\global\long\def\Tr{\mathrm{Tr}}%
\global\long\def\N{\mathbf{N}}%
\global\long\def\std{\mathrm{std}}%
\global\long\def\HS{\mathrm{H.S.}}%
\global\long\def\spec{\mathrm{spec}}%
\global\long\def\Ind{\mathrm{Ind}}%
\global\long\def\half{\frac{1}{2}}%
\global\long\def\Re{\mathrm{Re}}%
\global\long\def\Im{\mathrm{Im}}%
\global\long\def\Rect{\mathrm{Rect}}%
\global\long\def\Crit{\mathrm{Crit}}%
\global\long\def\Stab{\mathrm{Stab}}%
\global\long\def\SL{\mathrm{SL}}%
\global\long\def\Tab{\mathrm{Tab}}%
\global\long\def\Cont{\mathrm{Cont}}%
\global\long\def\I{\mathcal{I}}%
\global\long\def\J{\mathcal{J}}%
\global\long\def\short{\mathrm{short}}%
\global\long\def\Id{\mathrm{Id}}%
\global\long\def\B{\mathcal{B}}%
\global\long\def\ax{\mathrm{ax}}%
\global\long\def\cox{\mathrm{cox}}%
\global\long\def\row{\mathrm{row}}%
\global\long\def\col{\mathrm{col}}%
\global\long\def\X{\mathbb{X}}%

\global\long\def\V{\mathcal{V}}%
\global\long\def\P{\mathbb{P}}%
\global\long\def\Fill{\mathsf{Fill}}%
\global\long\def\fix{\mathsf{fix}}%
\global\long\def\reg{\mathrm{reg}}%
\global\long\def\edge{E}%
\global\long\def\id{\mathrm{id}}%
\global\long\def\emb{\mathrm{emb}}%

\global\long\def\Hom{\mathrm{Hom}}%
 
\global\long\def\F{\mathrm{\mathbf{F}} }%
  
\global\long\def\pr{\mathrm{Prob} }%
 
\global\long\def\tr{\mathrm{Tr} }%
\global\long\def\gs{\mathsf{GS}}%
 
\global\long\def\Xcov{{\scriptscriptstyle \overset{\twoheadrightarrow}{\mathcal{\gs}}}}%
 
\global\long\def\covers{\leq_{\Xcov}}%
\global\long\def\core{\mathrm{Core}}%
\global\long\def\pcore{\mathrm{PCore}}%
\global\long\def\im{\mathrm{imm}}%
\global\long\def\br{\mathsf{BR}}%
\global\long\def\ebs{\mathsf{EBS}}%
\global\long\def\ev{\mathrm{ev}}%
\global\long\def\CC{\mathcal{C}}%
\global\long\def\sides{\mathrm{Sides}}%
\global\long\def\tp{\mathrm{top}}%
\global\long\def\lf{\mathrm{left}}%
\global\long\def\MCG{\mathrm{MCG}}%

\global\long\def\defect{\mathrm{Defect}}%
\global\long\def\M{\mathcal{M}}%
\global\long\def\O{\mathcal{O}}%
\global\long\def\EE{\mathcal{E}}%
\global\long\def\sign{\mathrm{sign}}%

\global\long\def\v{\mathfrak{v}}%
\global\long\def\e{\mathfrak{e}}%
\global\long\def\f{\mathfrak{f}}%
\global\long\def\D{\mathfrak{D}}%
\global\long\def\d{\mathfrak{d}}%
\global\long\def\he{\mathfrak{he}}%
\global\long\def\OVB{\mathsf{OvB}}%
\global\long\def\G{\mathcal{G}}%

\global\long\def\SU{\mathsf{SU}}%
\global\long\def\p{\mathfrak{p}}%
\global\long\def\NN{\mathcal{N}}%
\global\long\def\hN{\hat{\mathcal{N}}}%
\global\long\def\Wg{\mathrm{Wg}}%
\global\long\def\Res{\mathrm{Res}}%
\global\long\def\ABG{\mathrm{ABG}}%
\global\long\def\zd{\mathrm{zd}}%
\global\long\def\SSTab{\mathcal{SST}}%
\global\long\def\FF{\mathcal{F}}%
\global\long\def\q{\mathfrak{q}}%
\global\long\def\PSU{\mathsf{PSU}}%
\global\long\def\irr{\mathrm{irr}}%
\global\long\def\su{\mathfrak{su}}%
\global\long\def\Ad{\mathrm{Ad}}%
\global\long\def\m{\mathbf{m}}%
\global\long\def\x{\mathbf{x}}%
\global\long\def\LR{\mathsf{LR}}%
\global\long\def\Q{\mathbf{Q}}%
\global\long\def\Z{\mathbf{Z}}%
\global\long\def\T{\mathcal{T}}%
\global\long\def\tkld{\dot{\T}_{n}^{k,\ell}}%
\global\long\def\tkl{\T_{n}^{k,\ell}}%
\global\long\def\match{\mathsf{MATCH}}%
\global\long\def\surfaces{\mathsf{surfaces}}%
\global\long\def\tkldc{\check{\dot{\mathcal{T}}}_{n}^{k,\ell}}%
\global\long\def\tklc{\check{\mathcal{T}}_{n}^{k,\ell}}%
\global\long\def\II{\mathfrak{I}}%
\global\long\def\j{\mathfrak{j}}%

\global\long\def\u{\mathbf{u}}%
\global\long\def\v{\mathbf{v}}%
\global\long\def\r{\mathbf{r}}%
\global\long\def\s{\mathbf{s}}%
\global\long\def\UU{\mathbf{U}}%
\global\long\def\V{\mathbf{V}}%
\global\long\def\RR{\mathbf{\mathbf{R}}}%
\global\long\def\SS{\mathbf{S}}%

\title[Random Unitary Representations of Surface Groups II]{Random Unitary Representations
 of Surface Groups II: \\ The large $n$ limit}
\author{Michael Magee}
\givenname{Michael}
\surname{Magee}
\address{Department of Mathematical Sciences, Durham University, Lower Mountjoy, DH1 3LE Durham, United
Kingdom}
\email{michael.r.magee@durham.ac.uk}
\urladdr{https://www.mmagee.net/}

\begin{abstract}
Let $\Sigma_{g}$ be a closed surface of genus $g\geq2$ and $\Gamma_{g}$
denote the fundamental group of $\Sigma_{g}$. We establish a generalization
of Voiculescu's theorem on the asymptotic $*$-freeness of Haar unitary
matrices from free groups to $\Gamma_{g}$. We prove that for a random
representation of $\Gamma_{g}$ into $\SU(n)$, with law given by
the volume form arising from the Atiyah-Bott-Goldman symplectic form
on moduli space, the expected value of the trace of a fixed non-identity
element of $\Gamma_{g}$ is bounded as $n\to\infty$. The proof involves
an interplay between Dehn's work on the word problem in $\Gamma_{g}$
and classical invariant theory.
\end{abstract}

\maketitle
\begin{center}
\emph{To Oona}
\par\end{center}

\keyword{}
\subject{primary}{msc2010}{14H60, 20C30, 20C35, 22D10, 32G15, 46L54, 57M20, 70S15}
\subject{secondary}{msc2010}{}


\arxivreference{2101.03224}

\tableofcontents{}

\section{Introduction}

In a foundational series of papers \cite{Voiculescu1985,Voiculescu86,Voiculescu87,Voiculescu90,Voiculescu1991},
Voiculescu developed a robust theory of non-commuting random variables
that became known as \emph{free probability}. One of the initial landmarks
of this theory is the following result. Let $\F_{r}$ denote the non-commutative
free group of rank $r$. Let $\U(n)$ denote the group of $n\times n$
complex unitary matrices. For any $w\in\F_{r}$ we obtain a \emph{word
map $w:\U(n)^{r}\to\U(n)$ }by substituting matrices for generators
of $\F_{r}$. Let $\mu_{\U(n)^{r}}^{\mathrm{Haar}}$ denote the probability
Haar measure on $\U(n)^{r}$ and $\tr:\U(n)\to\C$ the standard trace.
Any integral over a compact group will be done with respect to the
probability Haar measure, denoted by $d\mu$.

A simplified version of Voiculescu's result \cite[Thm. 3.8]{Voiculescu1991}
can be formulated as follows\footnote{Voiculescu's result in \cite[Thm. 3.8]{Voiculescu1991} is more general
than what we state here, also involving a deterministic sequence of
unitary matrices.}:
\begin{thm}[Voiculescu]
\label{thm:Voiculescu}For any non-identity $w\in\F_{r}$, as $n\to\infty$
\begin{align}
\int_{\U(n)^{r}}\tr(w(x))d\mu(x) & =o_{w}(n).\label{eq:haar-word-integral}
\end{align}
\end{thm}

We describe the interpretation of Theorem \ref{thm:Voiculescu} as
convergence of non-commutative random variables momentarily. Before
this, we explain the main result of the current paper.

Another way to think about the integral (\ref{eq:haar-word-integral}),
that invites generalization, is to identify $\U(n)^{r}$ with $\Hom(\F_{r},\U(n))$
and Haar measure as a natural probability measure on this \emph{representation
variety}. Now it is natural to ask whether there are other infinite
discrete groups rather than $\F_{r}$ such that $\Hom(\F_{r},\U(n))$
has a natural measure, and whether similar phenomena as in Theorem
\ref{thm:Voiculescu} may hold. \emph{The main point of this paper
is to establish the analog of Theorem \ref{thm:Voiculescu} when $\F_{r}$
is replaced by the fundamental group of a compact surface of genus
at least 2. }

We now explain this generalization of Theorem \ref{thm:Voiculescu};
for technical reasons it superficially looks slightly different as
follows.
\begin{enumerate}
\item The integral (\ref{eq:haar-word-integral}) is equal to $0$ if $w\notin[\F_{r},\F_{r}]$,
the commutator subgroup of $\F_{r}$ \cite[Claim 3.1]{MP}, and if
$w\in[\F_{r},\F_{r}]$, the value of (\ref{eq:haar-word-integral})
is for $n\geq n_{0}(w)$ the same as the corresponding integral over
$\SU(n)^{r}\leq\U(n)^{r}$, where $\SU(n)$ is the subgroup of determinant
one matrices \cite[Prop. 3.1]{MageeRURSG1}. So in all cases of interest
we can replace $\U(n)$ by $\SU(n)$ in (\ref{eq:haar-word-integral}).
\item Since $\tr\circ w$ is invariant under the diagonal conjugation action
of $\SU(n)$ on $\Hom(\F_{r},\SU(n))\cong\SU(n)^{r}$, the integral\\
 $\int_{\SU(n)^{r}}\tr(w(x))d\mu(x)$ can be written as an integral
over\\
 $\Hom(\F_{r},\SU(n))/\PSU(n)$. Here $\PSU(n)$ is $\SU(n)$ modulo
its center.
\end{enumerate}
For $g\geq2$ let $\Sigma_{g}$ denote a closed topological surface
of genus $g$. We let $\Gamma_{g}$ denote the fundamental group of
$\Sigma_{g}$ with explicit presentation
\[
\Gamma_{g}=\langle a_{1},b_{1},\ldots,a_{g},b_{g}\,|\,[a_{1},b_{1}]\cdots[a_{g},b_{g}]\,\rangle.
\]
The most natural measure on $\Hom(\Gamma_{g},\SU(n))/\PSU(n)$ to
replace the measure induced by Haar measure on $\Hom(\F_{r},\SU(n))/\PSU(n)$
is called the Atiyah--Bott--Goldman measure. The definition of this
measure involves removing singular parts of $\Hom(\Gamma_{g},\SU(n))/\PSU(n)$.
Indeed, let  $\Hom(\Gamma_{g},\SU(n))^{\irr}$ denote the collection of homomorphisms
that are irreducible as linear representations. Then 
\[
\M_{g,n}\eqdf\Hom(\Gamma_{g},\SU(n))^{\irr}/\PSU(n)
\]
is a smooth manifold \cite{GOLDMANSYMPLECTIC}. Moreover there is
a symplectic form $\omega_{g,n}$ on $\M_{g,n}$ called the Atiyah--Bott--Goldman
form after \cite{AB,GOLDMANSYMPLECTIC}. This symplectic form gives,
in the usual way, a volume form on $\M_{g,n}$ denoted by $\mathrm{Vol}_{\M_{g,n}}$.
For many more details see Goldman \cite{GOLDMANSYMPLECTIC} or the
prequel paper \cite[\S\S 2.7]{MageeRURSG1}.

For any $\gamma\in\Gamma$, we obtain a function $\tr_{\gamma}:\Hom(\Gamma_{g},\SU(n))\to\C$
defined by
\[
\tr_{\gamma}(\phi)\eqdf\tr(\phi(\gamma)).
\]
This function descends to a function $\tr_{\gamma}:\M_{g,n}\to\C$.
We are interested in the expected value
\[
\E_{g,n}[\tr_{\gamma}]\eqdf\frac{\int_{\M_{g,n}}\tr_{\gamma}\,d\mathrm{Vol}_{\M_{g,n}}}{\int_{\M_{g,n}}d\mathrm{Vol}_{\M_{g,n}}}.
\]
The main theorem of this paper is the following.
\begin{thm}
\label{thm:main-convergence}Let $g\geq2$. If $\gamma\in\Gamma_{g}$
is not the identity, then $\E_{g,n}[\tr_{\gamma}]=O_{\gamma}(1)$
as $n\to\infty$. 
\end{thm}

The non-commutative probabilistic consequences of Theorem \ref{thm:main-convergence}
will be discussed in the next section.

\subsection{Non-commutative probability}

We follow the book \cite{VDN}. A \emph{non-commutative probability
space} is a pair $(\B,\tau)$ where $\B$ is a complex unital algebra
and $\tau$ is a linear functional on $\B$ such that $\tau(1)=1$.
Let $\C\langle x_{1},\ldots,x_{r}\rangle$ denote the free non-commutative
unital algebra in indeterminates $x_{1},\ldots,x_{r}$. A \emph{random
variable }in $(\B,\tau)$ is an element of \emph{$\B$.} If $(X_{1},\ldots,X_{r})\in\B^{r}$
are random variables in $(\B,\tau)$, their \emph{joint distribution}
is defined to be the linear functional 
\[
\tilde{\tau}:\C\langle x_{1},\ldots,x_{r}\rangle\to\C
\]
defined by $\tilde{\tau}(z)\eqdf\tau(\Phi(z))$ where $\Phi:\C\langle x_{1},\ldots,x_{r}\rangle\to\B$
is the linear map defined by $\Phi(x_{i})=X_{i}$. For a linear functional
$\tilde{\tau}_{\infty}:\C\langle x_{1},\ldots,x_{r}\rangle\to\C$
with $\tilde{\tau}_{\infty}(1)=1$, we say that a sequence of random
variables $(X_{1}^{(n)},\ldots,X_{r}^{(n)})\in(\B_{n},\tau_{n})$
converge in distribution as $n\to\infty$ to $\tilde{\tau}_{\infty}$
if $\tilde{\tau}_{n}$ convergences pointwise to $\tilde{\tau}_{\infty}$
on $\C\langle x_{1},\ldots,x_{r}\rangle$.

A very concrete example of this phenomenon is as follows. The function
\[
\tau_{n}:\F_{r}\to\C,\quad\tau_{n}(w)\eqdf\frac{1}{n}\int_{\U(n)^{r}}\tr(w(x))d\mu(x)
\]
extends to a linear functional $\tau_{n}$ on the algebra $\C[\F_{r}]$
with $\tau_{n}(\id)=1$. From this point of view, Theorem \ref{thm:Voiculescu}
implies the following statement.
\begin{thm}[Voiculescu]
\label{thm:voicNCP}Let $r\geq0$ and $X_{1},\ldots,X_{r}$ denote
fixed generators of $\F_{r}$, and $\bar{X}_{1},\ldots.\bar{X}_{r}$
denote their inverses, i.e. $\bar{X}_{i}=X_{i}^{-1}$. The random
variables $X_{1},\ldots,X_{r},\bar{X}_{1},\ldots.\bar{X}_{r}$ in
the non-commutative probability spaces $(\C[\F_{r}],\tau_{n})$ converge
as $n\to\infty$ to a limiting distribution
\[
\tilde{\tau}_{\infty}:\C\langle x_{1},\ldots,x_{r},\bar{x}_{1},\ldots,\bar{x}_{r}\rangle\to\C
\]
that is completely determined by (\ref{eq:haar-word-integral}). Indeed,
if $w$ is any monomial in $x_{1},\ldots,x_{r},\bar{x}_{1},\ldots,\bar{x}_{r}$,
then $\tilde{\tau}_{\infty}(w)=1$ if and only if after identifying
$\bar{x}_{i}$ with $x_{i}^{-1}$, $w$ reduces to the identity in
$\F_{r}=\langle x_{1},\ldots,x_{r}\rangle$, and $\tilde{\tau}_{\infty}(w)=0$
otherwise. 
\end{thm}

In the language of \cite{Voiculescu1991}, in the limiting non-commutative
probability space \\$(\C\langle x_{1},\ldots,x_{r},\bar{x}_{1},\ldots,\bar{x}_{r}\rangle,\tilde{\tau}_{\infty})$,
the subalgebras
\[
\A_{1}\eqdf\C\langle x_{1},\bar{x}_{1}\rangle,\ldots,\A_{r}\eqdf\C\langle x_{r},\bar{x}_{r}\rangle
\]
are a \emph{free family of subalgebras}: if $a_{j}\in\A_{i_{j}}$
for $j\in[q]$, $i_{1}\neq i_{2}\neq\cdots\neq i_{q}$, and $\tilde{\tau}_{\infty}(a_{j})=0$
for $j\in[q]$ then
\[
\tilde{\tau}_{\infty}(a_{1}a_{2}\cdots a_{q})=0.
\]
Accordingly, it is said \cite[Thm. 3.8]{Voiculescu1991} that if $\{u_{j}(n):1\leq j\leq r\}$
are independent Haar-random elements of $\U(n)$, the family $\{\{u_{j}(n),u_{j}^{*}(n)\}:1\leq j\leq r\}$
of sets of random variables are \emph{asymptotically free}. 

Because $\Gamma_{g}$ is not free, asymptotic freeness does not correctly
capture the asymptotic behavior of the expected values $\E_{g,n}[\tr_{\gamma}]$,
however, an analog of Theorem \ref{thm:voicNCP} is implied by Theorem
\ref{thm:main-convergence}. For $\gamma\in\Gamma_{g}$ let 
\[
\tau_{g,n}(\gamma)\eqdf\frac{1}{n}\E_{g,n}[\tr_{\gamma}].
\]

\begin{cor}
\label{cor:main-cor}Let $g\geq2$, $a_{1},b_{1},\ldots,a_{g},b_{g}$
denote the previously fixed generators of $\Gamma_{g}$, and $\bar{a}_{1},\bar{b}_{1},\ldots,\bar{a}_{g},\bar{b}_{g}$
denote their inverses. The random variables $a_{1},b_{1},\ldots,a_{g},b_{g}$,
$\bar{a}_{1},\bar{b}_{1},\ldots,\bar{a}_{g},\bar{b}_{g}$ in the non-commutative
probability spaces $(\C[\Gamma_{g}],\tau_{g,n})$ converge in distribution
as $n\to\infty$ to a limiting distribution
\[
\tilde{\tau}_{g,\infty}:\C\langle x_{1},\ldots,x_{g},y_{1},\ldots,y_{g},\bar{x}_{1},\ldots,\bar{x}_{g},\bar{y}_{1},\ldots,\bar{y}_{g}\rangle\to\C,
\]
where $x_{i}$ (resp. $y_{i},\bar{x}_{i},\bar{y}_{i}$) corresponds
to $a_{i}$ (resp. $b_{i},\bar{a}_{i},\bar{b}_{i})$. This can be
described explicitly as follows. If $w$ is any monomial in $x_{1},\ldots,x_{g},y_{1},\ldots,y_{g},$
$\bar{x}_{1},\ldots,\bar{x}_{g},\bar{y}_{1},\ldots,\bar{y}_{g}$,
then $\tilde{\tau}_{g,\infty}(w)=1$ if and only if $w$ maps to the
identity under the map
\[
\C\langle x_{1},\ldots,x_{g},y_{1},\ldots,y_{g},\bar{x}_{1},\ldots,\bar{x}_{g},\bar{y}_{1},\ldots,\bar{y}_{g}\rangle\to\C[\Gamma_{g}]
\]
obtained by identifying $x_{i},y_{i},\bar{x}_{i},\bar{y}_{i}$ with
the corresponding elements of $\Gamma_{g}$. If $w$ does not map
to the identity under this map, then $\tilde{\tau}_{g,\infty}(w)=0$.
\end{cor}

Notice that the estimate given in Theorem \ref{thm:main-convergence}
is stronger than needed to establish Corollary \ref{cor:main-cor}.

\subsection{Related works and further questions}

The most closely related existing result to Theorem \ref{thm:main-convergence}
is a theorem of the author and Puder \cite[Thm. 1.2]{MPasympcover}
that establishes Theorem \ref{thm:main-convergence} when the family
of groups $\SU(n)$ is replaced by the family of symmetric groups
$S_{n}$, and $\tr$ is replaced by the character $\fix$ given by
the number of fixed points of a permutation. In this case, the result
is phrased in terms of integrating over $\Hom(\Gamma_{g},S_{n})$
with respect to the uniform probability measure. The corresponding
result for $\Hom(\F_{r},S_{n})$ was proved much longer ago by Nica
in \cite{nica1994number}.

The problem of integrating geometric functions like $\tr_{\gamma}$
over $\M_{g,n}$ is also connected to the work of Mirzakhani since,
as Goldman explains in \cite[\S 2]{GOLDMANSYMPLECTIC}, the Atiyah--Bott--Goldman
symplectic form generalizes the Weil--Petersson symplectic form on
the Teichmüller space of genus $g$ Riemann surfaces. In \cite{Mirzakhani},
Mirzakhani developed a method for integrating geometric functions
on moduli spaces of Riemann surfaces with respect to the Weil--Petersson
volume form. Although there is certainly a similarity between \emph{(ibid.)
}and the current work, here the emphasis is on $n\to\infty$, whereas
\emph{(ibid.)} caters to the regime $g\to\infty$; the target group
playing the role of $\SU(n)$ is always $\mathsf{\mathsf{PSL}}(2,\R)$.

We now take the opportunity to mention some questions that Theorem
\ref{thm:main-convergence} leads to. In the paper \cite{Voiculescu1991},
Voiculescu is able to boost Theorem \ref{thm:Voiculescu} from a convergence
in distribution result to a result on convergence in probability,
that is, for any $\epsilon>0$, and fixed $w\in\F_{r}$, the Haar
measure of the set 
\[
\{\,\phi\in\Hom(\F_{r},\U(n))\,:\,|\tr(\phi(w))|\leq\epsilon n\,\}
\]
tends to one as $n\to\infty$ \cite[Thm. 3.9]{Voiculescu1991}. To
do this, Voiculescu uses that the family of measure spaces $(\Hom(\F_{r},\U(n)),\mu)$
form a\emph{ Levy family} in the sense of Gromov and Milman \cite{GromovMilman}.
This latter fact relies on an estimate for the first non-zero eigenvalue
of the Laplacian on $\Hom(\F_{r},\U(n))$. It is interesting to ask
whether a similar phenomenon holds for the family of measure spaces
$(\M_{g,n},\mu_{g,n}^{\mathrm{ABG}})$ where $\mu_{g,n}^{\mathrm{ABG}}$
is the probability measure corresponding to $\mathrm{Vol}_{\M_{g,n}}$.
The fact that $\M_{g,n}$ is non-compact seems to be a significant
complication in answering this question using isoperimetric inequalities.

On the other hand, as pointed out to us by a referee, the results
of this paper can very likely be extended to give bounds on the variance
\[
\E_{g,n}[|\tr_{\gamma}|^{2}]
\]
that can be used to improve Theorem \ref{thm:main-convergence} to
the result that for $\gamma\neq\id$, the normalized traces $\frac{\tr_{\gamma}}{n}$
converge in probability to zero as $n\to\infty$. To avoid adding
complications to this paper, this will be pursued elsewhere.

In the prequel to this paper \cite{MageeRURSG1} it was proved that
for any fixed $\gamma\in\Gamma_{g}$, there is an infinite sequence
of rational numbers $a_{-1}(\gamma),a_{0}(\gamma),a_{1}(\gamma),\ldots\in\Q$
such that for any $M\in\N$, 
\begin{equation}
\E_{g,n}[\tr_{\gamma}]=a_{-1}(\gamma)n+a_{0}(\gamma)+\frac{a_{1}(\gamma)}{n}+\cdots+\frac{a_{M-1}(\gamma)}{n^{M-1}}+O_{\gamma,M}\left(\frac{1}{n^{M}}\right)\label{eq:Laurent}
\end{equation}
as $n\to\infty$. Theorem \ref{thm:main-convergence} implies that
$a_{-1}(\gamma)=0$ if $\gamma\neq\id$. It is also interesting to
understand the other coefficients of this series. This has been accomplished
when $\Gamma_{g}$ is replaced by $\F_{r}$ by the author and Puder
in \cite{MPunitary} where in fact it is proved that
\[
\E_{\F_{r},n}[\tr_{w}]\eqdf\int_{\U(n)^{r}}\tr(w(x))d\mu(x)
\]
is given by a \emph{rational }function of $n$ and in particular can
be expanded as in (\ref{eq:Laurent}). The corresponding coefficients
of the Laurent series of $\E_{\F_{r},n}[\tr_{w}]$ are explained in
terms of Euler characteristics of subgroups of mapping class groups.
One corollary is that as $n\to\infty$
\begin{equation}
\E_{\F_{r},n}[\tr_{w}]=O\left(\frac{1}{n^{2\mathrm{cl}(w)-1}}\right)\label{eq:cl}
\end{equation}
where $\mathrm{cl}(w)$ is the the \emph{commutator length of $w$}:
the minimal number of commutators that $w$ can be written as a product
of, or $\infty$ if $w\notin[\F_{r},\F_{r}]$. We guess that an estimate
like (\ref{eq:cl}) should hold for $\E_{g,n}[\tr_{\gamma}]$ where
commutator length in $\F_{r}$ is replaced by commutator length in
$\Gamma_{g}$.

Another strengthening of Theorem \ref{thm:Voiculescu} is the \emph{strong
asymptotic freeness} of Haar unitaries. This states that for any complex
linear combination
\[
\sum_{w}a_{w}w\in\C[\F_{r}],
\]
almost surely w.r.t. Haar random $\phi\in\Hom(\F_{r},\U(n))$ as $n\to\infty$,
we have
\[
\Big\|\sum_{w}a_{w}\phi(w)\Big\|\to\Big\|\sum_{w}a_{w}w\Big\|_{\mathrm{Op}(\ell^{2}(\F_{r}))}
\]
where the left hand side is the operator norm on $\C^{n}$ with standard
Hermitian inner product and the norm on the right hand side is the
operator norm in the regular representation of $\F_{r}$. This result
was proved by Collins and Male in \cite{Collins2014}. It is probably
very hard to extend this result to $\Gamma_{g}$; the proof of Collins
and Male relies on seminal work of Haagerup and Thorbjørnsen \cite{Haagerup2005}
in a way that does not obviously extend to $\Gamma_{g}$.

We finally mention that the expected values $\E_{g,n}[\tr_{\gamma}]$
arise as a limiting form of expected values of Wilson loops in 2D
Yang-Mills theory, when the coupling constant is set to zero. This
will not be discussed in detail here, we refer the reader instead
to the introduction of \cite{MageeRURSG1}. Here we just mention recent
works of Lemoine \cite{Lemoine}and Dahlqvist---Lemoine \cite{https://doi.org/10.48550/arxiv.2201.05882}
that make progress on related problems in the Yang-Mills setting.

\subsection{Overview of paper}

Here we explain the structure of the paper.

In $\S\S$\ref{subsec:Representation-theory-of-sym-groups}-$\S\S$\ref{subsec:Witten-zeta-functions}
we give some general background to the paper not depending on the
prequel \cite{MageeRURSG1}. In $\S\S\ref{subsec:Results-of-the}$
we import results that we proved in the prequel and that are needed
here.

At the beginning of $\S\ref{sec:The-contribution-from-single-small-dim}$
we state the key result (Theorem \ref{thm:main-single-small-dim})
of the remainder of the paper. To motivate things, $\S\S$\ref{subsec:Setup-and-motivation}
contains a discussion of why the most straightforward approach does
not work, and also a discussion of what will follow instead. In the
remainder of $\S\ref{sec:The-contribution-from-single-small-dim}$
we explain how to augment the Weingarten calculus to arrive to a formula
for the key quantity $\J_{n}(w,\mu,\nu)$ (defined in Proposition
\ref{prop:passage-to-unitary}) in combinatorial terms that are `good'
for the next part of the argument.

Indeed in $\S\S$\ref{subsec:Construction-of-surfaces} we explain
how each combinatorial datum we encountered in our formula for $\J_{n}(w,\mu,\nu)$
can be used to build a decorated surface. In Corollary \ref{cor:goodmatchings}
we obtain a bound on $\J_{n}(w,\mu,\nu)$ in terms of the Euler characteristics
of some of the surfaces that previously arose. We may restrict to
certain surfaces of simplified form by performing two surgery arguments
explained in $\S\S$\ref{subsec:Two-simplifying-surgeries}. Given
that now we have reduced estimating $\J_{n}(w,\mu,\nu)$ to estimating
Euler characteristics of certain surfaces, in $\S\S$\ref{subsec:A-topological-result}
we formulate a topological result (Proposition \ref{prop:chi-bound})
which suffices to prove Theorem \ref{thm:main-single-small-dim}.
Proposition \ref{prop:chi-bound} is proved in $\S\S$\ref{subsec:Proof-of-chi-bound}
using arguments related to Dehn's algorithm and the work of Birman---Series.
The necessary additional background for this proof is given in $\S\S$\ref{subsec:Work-of-Dehn}.

In $\S$\ref{sec:Proof-of-main} we show how Theorem \ref{thm:main-single-small-dim}
in conjunction with the results of the prequel \cite{MageeRURSG1}
prove Theorem \ref{thm:main-convergence}.

\subsection{Notation}

We write $\N$ for the natural numbers $\{1,2,3,\ldots\}$ and $\N_{0}\eqdf\N\cup\{0\}$.
We write $[n]\eqdf\{1,\ldots,n\}$ for $n\in\N$ and $[k,\ell]\eqdf\{k,k+1,\ldots,\ell\}$
for $k,\ell\in\N$. If $A$ and $B$ are two sets we write $A\backslash B$
for the elements of $A$ not in $B$. If $H$ is a group and $h_{1},h_{2}\in H$
we write $[h_{1},h_{2}]\eqdf h_{1}h_{2}h_{1}^{-1}h_{2}^{-1}$. We
let $\id$ denote the identity element of a group. We let $[H,H]$
be the subgroup of $H$ generated by elements of the form $[h_{1},h_{2}]$;
this is called the commutator subgroup of $H$. If $V$ is a complex
vector space, for $q\in\N_{0}$ we let
\[
V^{\otimes q}\eqdf\underbrace{V\otimes V\otimes\cdots\otimes V}_{q}.
\]

We use Vinogradov notation as follows. If $f$ and $h$ are functions
of $n\in\N$, we write $f\ll h$ to mean that there are constants
$n_{0}\geq0$ and $C_{0}\geq0$ such that for $n\geq n_{0}$, $|f(n)|\leq C_{0}h(n)$.
We write $f=O(h)$ to mean $f\ll h$. We write $f\asymp h$ to mean
both $f\ll h$ and $h\ll f$. If in any of these statements the implied
constants depend on additional parameters we add these parameters
as subscript to $\ll,O,$ or $\asymp$. Throughout the paper we view
the genus $g$ as fixed and so any implied constant may depend on
$g$. 

In this paper, $\Tr$ denotes the standard (unnormalized) trace on
square complex matrices.

\subsubsection*{Acknowledgments}

We thank Benoît Collins, Antoine Dahlqvist, Doron Puder, Sanjaye Ramgoolam, Calum Shearer,
and Henry Wilton for valuable discussions about this work. This project
has received funding from the European Research Council (ERC) under
the European Union’s Horizon 2020 research and innovation programme
(grant agreement No 949143).

\global\long\def\i{\mathfrak{i}}%

\section{Background}

\subsection{Representation theory of symmetric groups\label{subsec:Representation-theory-of-sym-groups}}

Let $S_{k}$ denote the symmetric group of permutations of $[k]\eqdf\{1,\ldots,k\}$,
and $\C[S_{k}]$ denote its group algebra. The group $S_{0}$ is by
definition the group with one element.

If we refer to $S_{\ell}\leq S_{k}$ with $\ell\leq k$ we always
view $S_{\ell}$ as the subgroup of permutations that fix every element
of $[\ell+1,k]\eqdf\{\ell+1,\ldots,k\}$. We write $S'_{r}\leq S_{k}$
for the subgroup of permutations that fix every element of $[k-r]$.
As a consequence we obtain fixed inclusions $\C[S_{\ell}]\subset\C[S_{k}]$
for $\ell$ and $k$ as above. When we write $S_{\ell}\times S_{k-\ell}\leq S_{k}$,
the first factor is $S_{\ell}$ and the second factor is $S'_{k-\ell}$.

A \emph{Young diagram }$\lambda$ is a left-aligned contiguous collection
of identical square boxes in the plane, such that the number of boxes
in each row is non-increasing from top to bottom. We write $\lambda_{i}$
for the number of boxes in the $i$\textsuperscript{th} row of $\lambda$
and say $\lambda\vdash k$ if $\lambda$ has $k$ boxes. We write
$\ell(\lambda)$ for the number of rows of $\lambda$. For each $\lambda\vdash k$
there is a \emph{Young subgroup
\[
S_{\lambda}\eqdf S_{\lambda_{1}}\times S_{\lambda_{2}}\times\cdots\times S_{\lambda_{\ell(\lambda)}}\leq S_{k}
\]
}where the factors are subgroups in the obvious way, according to
the increasing order of $[k]$.

The equivalence classes of irreducible representations of $S_{k}$
are in one-to-one correspondence with Young diagrams $\lambda\vdash k$.
Given $\lambda$, the construction of the corresponding irreducible
representation $V^{\lambda}$ can be done for example using Young
symmetrizers as in \cite[Lec. 4]{FH}. We write $\chi_{\lambda}$
for the character of $S_{k}$ associated to $V^{\lambda}$ and $d_{\lambda}\eqdf\chi_{\lambda}(\id)=\dim V^{\lambda}$.
Given $\lambda\vdash k$, the element
\[
\p_{\lambda}\eqdf\frac{d_{\lambda}}{k!}\sum_{\sigma\in S_{k}}\chi_{\lambda}(\sigma)\sigma\in\C[S_{k}]
\]
 is a central idempotent in $\C[S_{k}]$. 

If $G$ is a compact group, $(\rho,W)$ is an irreducible representation
of $G$, and $(\pi,V)$ is any finite-dimensional representation of
$G$, the $(\rho,W)$-isotypic subspace of $(\pi,V)$ is the invariant
subspace of $V$ spanned by all irreducible direct summands of $(\pi,V)$
that are isomorphic to $(\rho,W)$. When $\rho$ and $\pi$ can be
inferred from $W$ and $V$ we call this simply the $W$\emph{-isotypic
subspace} \emph{of} $V$. If $H\leq G$ is a subgroup, and $(\rho,W)$
is an irreducible representation of $H$, then the $W$-isotypic subspace
of $V$ for $H$ is the $W$-isotypic subspace of the restriction
of $(\pi,V)$ to $H$.

If $(\pi,V)$ is any finite-dimensional unitary representation of
$S_{k}$, and $\lambda\vdash k$, then $V$ is also a module for $\C[S_{k}]$
by linear extension of $\pi$ and $\pi(\p_{\lambda})$ is the orthogonal
projection onto the $V^{\lambda}$-isotypic subspace of $V$.

For any compact group $G$ we write $(\triv_{G},\C)$ for the trivial
representation of $G$. The following lemma can be deduced for example
by combining Young's rule \cite[Cor. 4.39]{FH} with Frobenius reciprocity.
\begin{lem}
\label{lem:Youngs-rule-application}Let $k\in\N_{0}$, and $\lambda\vdash k$.
The space of vectors in $V^{\lambda}$ fixed by $S_{\lambda}$ is
one-dimensional.
\end{lem}

\subsection{Representation theory of $\protect\U(n)$ and $\protect\SU(n)$\label{subsec:Representation-theory-of-U(n)-SU(n)}}

Every irreducible representation of $\U(n)$ restricts to an irreducible
representation of $\SU(n)$, and all equivalence classes of irreducible
representations of $\SU(n)$ arise in this way. The equivalence classes
of irreducible representations of $\U(n)$ are parameterized by dominant
weights, that can be thought of as non-increasing sequences 
\[
\Lambda=(\Lambda_{1},\ldots,\Lambda_{n})\in\Z^{n},
\]
also known as \emph{signatures. }We write $W^{\Lambda}$ for the irreducible
representation of $\U(n)$ corresponding to the signature $\Lambda$.
Two irreducible representations of $\U(n)$ restrict to the same one
of $\SU(n)$ if and only if their signatures differ by a constant
vector. Let $\mathbb{T}(n)$ denote the maximal torus of $\U(n)$
consisting of diagonal matrices. Any matrix of $\mathbb{T}(n)$ has
the form $\mathrm{diag}(\exp(i\theta_{1}),\ldots,\exp(i\theta_{n}))$
where all $\theta_{j}\in\R$. Associated to the signature $\Lambda$
is the character $\xi_{\Lambda}$ of $\mathbb{T}(n)$ given by 
\[
\xi_{\Lambda}\left(\mathrm{diag}(\exp(i\theta_{1}),\ldots,\exp(i\theta_{n}))\right)\eqdf\exp\left(i\left(\sum_{j=1}^{n}\Lambda_{j}\theta_{j}\right)\right).
\]
The highest weight theory says among other things that the $\xi_{\Lambda}$-isotypic
subspace of $W^{\Lambda}$ for $\mathbb{T}(n)$ is one-dimensional.
Any vector in this subspace is called a \emph{highest weight vector
}of $W^{\Lambda}$.

Given $k,\ell\in\N_{0}$ and fixed Young diagrams $\mu\vdash k$,
$\nu\vdash\ell$, we define a family of representations of $\U(n)$
as follows. For $n\geq\ell(\mu)+\ell(\nu)$ define
\[
\Lambda_{\mu,\nu}(n)\eqdf(\mu_{1},\mu_{2},\ldots,\mu_{\ell(\mu)},\underbrace{0,\ldots,0}_{n-\ell(\mu)-\ell(\nu)},-\nu_{\ell(\nu)},-\nu_{\ell(\nu)-1},\ldots,-\nu_{1}).
\]
We let $(\rho_{n}^{\mu,\nu},W_{n}^{\mu,\nu})$ denote the irreducible
representation of $\U(n)$ corresponding to $\Lambda_{\mu,\nu}(n)$
when $n\geq\ell(\mu)+\ell(\nu)$. We let $D_{\mu,\nu}(n)\eqdf\dim W_{n}^{\mu,\nu}$
and $s_{\mu,\nu}(g)\eqdf\tr(\rho_{n}^{\mu,\nu}(g))$ for $g\in\U(n)$.
If $\mu\vdash k$ and $\nu\vdash\ell$ then as $n\to\infty$ 
\begin{equation}
D_{\mu,\nu}(n)\asymp n^{k+\ell}\label{eq:dimension-asymp}
\end{equation}
by \cite[Cor. 2.3]{MageeRURSG1} (alternatively \cite[Lem. 3.5]{EnomotoIzumi}).

We now present a version of Schur-Weyl duality for mixed tensors due
to Koike \cite{Koike}. The very definition of $\U(n)$ makes $\C^{n}$
into a unitary representation of $\U(n)$ for the standard Hermitian
inner product. We let $\{e_{1},\ldots,e_{n}\}$ denote the standard
basis of $\C^{n}$. If $(\rho,W)$ is any finite dimensional representation
of $\U(n)$ we write $(\rho^{\vee},W^{\vee})$ for the dual representation
where $W^{\vee}$ is the space of complex linear functionals on $W$.
The vector space $(\C^{n})^{\vee}$ has a dual basis $\{\check{e}_{1},\ldots,\check{e}_{n}\}$
given by $\check{e}_{j}(v)\eqdf\langle v,e_{j}\rangle$. Throughout
the paper we frequently use certain canonical isomorphisms e.g.
\[
\left((\C^{n})^{\otimes p}\right)^{\vee}\cong\left(\left(\C^{n}\right)^{\vee}\right){}^{\otimes p},\quad\End(W)\cong W\otimes W^{\vee}
\]
to change points of view on representations; if we use non-canonical
isomorphisms we point them out.

Let $\T_{_{n}}^{k,\ell}\eqdf(\C^{n})^{\otimes k}\otimes\left(\left(\C^{n}\right)^{\vee}\right){}^{\otimes\ell}$,
with the convention that $(\C^{n})^{\otimes0}\eqdf\C$. With the natural
inner product induced by that on $\C^{n}$, this is a unitary representation
of $\U(n)$ under the diagonal action and also a unitary representation
of $S_{k}\times S_{\ell}$ where $S_{k}$ acts by permuting the indices
of $(\C^{n})^{\otimes k}$ and $S_{\ell}$ acts by permuting the indices
of $\left(\left(\C^{n}\right)^{\vee}\right){}^{\otimes\ell}$. We
write $\pi_{n}^{k,\ell}:\U(n)\to\End[\tkl]$ and $\rho_{n}^{k,\ell}:\C[S_{k}\times S_{\ell}]\to\End[\tkl]$
for these representations. The actions of $\U(n)$ and $S_{k}\times S_{\ell}$
on $\tkl$ commute. We use the notation, for $I=(i_{1},\ldots,i_{k})\in[n]^{k}$
and $J=(j_{1},\ldots,j_{\ell})\in[n]^{\ell}$
\begin{align*}
e_{I} & \eqdf e_{i_{1}}\otimes\cdots\otimes e_{i_{k}}\in(\C^{n})^{\otimes k},\quad\check{e}_{J}\eqdf\check{e}_{j_{1}}\otimes\cdots\otimes\check{e}_{j_{\ell}}\in\left(\left(\C^{n}\right)^{\vee}\right){}^{\otimes\ell},\\
e_{I}^{J} & \eqdf e_{I}\otimes\check{e}_{J}\in\T_{n}^{k,\ell}.
\end{align*}
We write $I\sqcup J$ for the concatenation $(i_{1},\ldots,i_{k},j_{1},\ldots,j_{\ell})$.

For $k,\ell\geq1$ let $\tkld$ denote the intersection of the kernels
of the mixed contractions $c_{pq}:\tkl\to\mathcal{T}^{k-1,\ell-1}$,
$p\in[k]$, $q\in[\ell]$ given by
\begin{align}
 & c_{pq}(e_{i_{1}}\otimes\cdots\otimes e_{i_{k}}\otimes\check{e}_{j_{1}}\otimes\cdots\otimes\check{e}_{j_{\ell}})\nonumber \\
 & \eqdf\delta_{i_{p}j_{q}}e_{i_{1}}\otimes\cdots\otimes e_{i_{p-1}}\otimes e_{i_{p+1}}\otimes\cdots\otimes e_{i_{k}}\label{eq:contraction-def}\\
 & \thinspace\thinspace\thinspace\thinspace\thinspace\thinspace\thinspace\thinspace\otimes\check{e}_{j_{1}}\otimes\cdots\otimes\check{e}_{j_{q-1}}\otimes\check{e}_{j_{q+1}}\otimes\cdots\otimes\check{e}_{j_{\ell}},\nonumber 
\end{align}
where $\delta_{i_{p}j_{q}}$ is the Kronecker delta. If $k=1$ or
$\ell=1$ then the definition is extended in the natural way, interpreting
an empty tensor of $e_{i}$ or $\check{e}_{i}$ as $1$. If either
$k=0$ or $\ell=0$ then $\tkld=\tkl$ by convention. The space $\tkld$
is an invariant subspace under $\U(n)\times S_{k}\times S_{\ell}$
and hence a unitary subrepresentation of $\tkl$. On $\tkld$ there
is an analog of Schur--Weyl duality due to Koike.
\begin{thm}
\label{thm:schur-weyl-mixed}\cite[Thm. 1.1]{Koike}~There is an
isomorphism of unitary representations of $\U(n)\times S_{k}\times S_{\ell}$
\begin{equation}
\tkld\cong\bigoplus_{\substack{\mu\vdash k,\nu\vdash\ell\\
\ell(\mu)+\ell(\nu)\leq n
}
}W_{n}^{\mu,\nu}\otimes V^{\mu}\otimes V^{\nu}.\label{eq:schurweyl-mixed}
\end{equation}
\end{thm}

Next we explain how to construct $\U(n)$-subrepresentations of $\tkld$
isomorphic to $W_{n}^{\mu,\nu}$. Suppose that $\xi\in\tkld$ is a
non-zero vector such that under the isomorphism (\ref{eq:schurweyl-mixed}),
\begin{equation}
\xi\cong w\otimes v\label{eq:pure-tensor-form}
\end{equation}
 for $w\in W_{n}^{\mu,\nu}$ and $v\in V^{\mu}\otimes V^{\nu}$. Then
$\U(n)\cdot\xi$ linearly spans a $\U(n)$-subrepresentation of $\tkld$
isomorphic to $W_{n}^{\mu,\nu}$. The following argument to construct
such a vector $\xi$, given $\mu\vdash k,\nu\vdash\ell$, appears
implicitly in \cite{Koike} and is elaborated in \cite{Benkart}.
For $n\geq\ell(\mu)+\ell(\nu)$ let 
\begin{align}
\tilde{\theta}_{\mu,\nu}^{n}\eqdf & e_{1}^{\otimes\mu_{1}}\otimes\cdots\otimes e_{\ell(\mu)}^{\otimes\mu_{\ell(\mu)}}\otimes\left(\check{e}_{n}\right)^{\otimes\nu_{1}}\otimes\cdots\otimes\left(\check{e}_{n-\ell(\nu)+1}\right)^{\otimes\nu_{\ell(\nu)}}.\label{eq:tilde-theta-def}
\end{align}
This vector is in the $\xi_{\mu,\nu}$-isotypic subspace of $\tkld$
for the maximal torus $\mathbb{T}(n)$ of $\U(n)$ where $\xi_{\mu,\nu}$
is the character of $\mathbb{T}(n)$ corresponding to the highest
weight in $W_{n}^{\mu,\nu}$.

Let $\p_{\mu}\in\C[S_{k}]$, $\p_{\nu}\in\C[S{}_{\ell}]$ be the projections
defined in $\S\S$\ref{subsec:Representation-theory-of-sym-groups}.
Let $\rho_{n}^{k}:S_{k}\to\End(\tkl)$ denote the representation of
$S_{k}$ described above and $\hat{\rho}_{n}^{k}:S{}_{\ell}\to\End(\tkl)$
that of $S_{\ell}$. Clearly these two representations commute. Now
let 
\begin{equation}
\theta_{\mu,\nu}^{n}\eqdf\rho_{n}^{k}(\p_{\mu})\hat{\rho}_{n}^{\ell}(\p_{\nu})\tilde{\theta}_{\mu,\nu}^{n}\in\tkld.\label{eq:theta-def}
\end{equation}
Now this is in the same isotypic subspace for $\mathbb{T}(n)$ as
before since $S_{k}\times S_{\ell}$ commutes with $\U(n)$. Moreover
it is in the subspace of $\tkld$ corresponding to $W_{n}^{\mu,\nu}\otimes V^{\mu}\otimes V^{\nu}$
under the isomorphism (\ref{eq:schurweyl-mixed}). The intersection
of the two subspaces of $\tkld$ just discussed corresponds via (\ref{eq:schurweyl-mixed})
to $\C w\otimes V^{\mu}\otimes V^{\nu}$ where $w$ is a highest weight
vector in $W_{n}^{\mu,\nu}$ and hence $\theta_{\mu,\nu}^{n}$ takes
the form of (\ref{eq:pure-tensor-form}) as we desired.

Of course we also want to know $\theta_{\mu,\nu}^{n}\neq0$. 
\begin{lem}
\label{lem:theta-norm}Suppose that $k,\ell\in\N_{0}$, $\mu\vdash k,\nu\vdash\ell$,
and $\theta_{\mu,\nu}^{n}$ is as in (\ref{eq:theta-def}) for $n\geq\ell(\mu)+\ell(\nu)$.
We have 
\[
\|\theta_{\mu,\nu}^{n}\|^{2}=\frac{d_{\mu}d_{\nu}}{[S_{k}:S_{\mu}][S_{\ell}:S_{\nu}]}.
\]
\end{lem}

\begin{proof}
Recall the definition of Young subgroups $S_{\mu},S_{\nu}$ from $\S\S$\ref{subsec:Representation-theory-of-sym-groups}.
Letting $\tilde{\theta}=\tilde{\theta}_{\mu,\nu}^{n}$ (as in (\ref{eq:tilde-theta-def}))
and $\theta=\theta_{\mu,\nu}^{n}$ we have 
\begin{align*}
\theta & =\rho_{n}^{k}(\p_{\mu})\hat{\rho}_{n}^{\ell}(\p_{\nu})\tilde{\theta}=\frac{d_{\mu}d_{\nu}}{k!\ell!}\sum_{\sigma=(\sigma_{1},\sigma_{2})\in S_{k}\times S_{\ell}}\chi_{\mu}(\sigma_{1})\chi_{\nu}(\sigma_{2})\rho_{n}^{k}(\sigma_{1})\hat{\rho}_{n}^{\ell}(\sigma_{2})\tilde{\theta}\\
 & =\frac{d_{\mu}d_{\nu}}{k!\ell!}\sum_{\substack{[\sigma_{1}]\in S_{k}/S_{\mu}\\{}
[\sigma_{2}]\in S_{\ell}/S_{\nu}
}
}\left(\sum_{\tau_{1}\in S_{\mu}}\chi_{\mu}(\sigma_{1}\tau_{1})\right)\left(\sum_{\tau_{2}\in S_{\nu}}\chi_{\nu}(\sigma_{2}\tau_{2})\right)\rho_{n}^{k}(\sigma_{1})\hat{\rho}_{n}^{\ell}(\sigma_{2})\tilde{\theta}.
\end{align*}
The second equality used that $\tilde{\theta}$ is invariant under
$S_{\mu}\times S_{\nu}$. 

By Lemma \ref{lem:Youngs-rule-application}, there is a one dimensional
subspace of invariant vectors for $S_{\mu}$ in $V^{\mu}$. If $v_{\mu}\in V^{\mu}$
is a unit vector in this space then
\begin{equation}
\sum_{\tau_{1}\in S_{\mu}}\chi_{\mu}(\sigma_{1}\tau_{1})=|S_{\mu}|\langle\sigma_{1}v_{\mu},v_{\mu}\rangle.\label{eq:application-of-Young's-rule}
\end{equation}
 Since the vectors $\rho_{n}^{k}(\sigma_{1})\hat{\rho}_{n}^{\ell}(\sigma_{2})\tilde{\theta}$
for $[\sigma_{1}]\in S_{k}/S_{\mu}$ and $[\sigma_{2}]\in S_{\ell}/S_{\nu}$
are orthogonal unit vectors, this gives
\begin{align*}
\|\theta\|^{2} & =\left(\frac{d_{\mu}d_{\nu}}{k!\ell!}\right)^{2}\sum_{\substack{[\sigma_{1}]\in S_{k}/S_{\mu}\\{}
[\sigma_{2}]\in S_{\ell}/S_{\nu}
}
}\left(\sum_{\tau_{1}\in S_{\mu}}\chi_{\mu}(\sigma_{1}\tau_{1})\right)^{2}\left(\sum_{\tau_{2}\in S_{\nu}}\chi_{\nu}(\sigma_{2}\tau_{2})\right)^{2}\\
 & \stackrel{\eqref{eq:application-of-Young's-rule}}{=}\left(\frac{d_{\mu}d_{\nu}}{k!\ell!}\right)^{2}|S_{\mu}|^{2}|S_{\nu}|^{2}\sum_{\substack{[\sigma_{1}]\in S_{k}/S_{\mu}\\{}
[\sigma_{2}]\in S_{\ell}/S_{\nu}
}
}\left|\langle\sigma_{1}v_{\mu},v_{\mu}\rangle\right|^{2}\left|\langle\sigma_{2}v_{\nu},v_{\nu}\rangle\right|^{2}\\
 & =\left(\frac{d_{\mu}d_{\nu}}{k!\ell!}\right)^{2}|S_{\mu}||S_{\nu}|\sum_{\substack{\sigma_{1}\in S_{k}\\
\sigma_{2}\in S_{\ell}
}
}\left|\langle\sigma_{1}v_{\mu},v_{\mu}\rangle\right|^{2}\left|\langle\sigma_{2}v_{\nu},v_{\nu}\rangle\right|^{2}=\frac{d_{\mu}d_{\nu}}{[S_{k}:S_{\mu}][S_{\ell}:S_{\nu}]}.
\end{align*}
The last inequality used the orthogonality relations for matrix coefficients.
\end{proof}
Recall that we write $\pi_{n}^{k,\ell}:\U(n)\to\End(\tkl)$ for the
diagonal representation of $\U(n)$ on $\tkl$. Lemma \ref{lem:theta-norm}
implies that $\theta_{\mu,\nu}^{n}$ is a non-zero vector. By the
remarks following (\ref{eq:theta-def}) it is of the pure tensor form
$w\otimes v$ under the Schur-Weyl isomorphism (\ref{eq:schurweyl-mixed}),
with $w\in W_{n}^{\mu,\nu}$, and hence we obtain the following corollary.
\begin{cor}
\label{cor:subspace-construct}Suppose $n\geq\ell(\mu)+\ell(\nu)$.
The subspace 
\[
W_{n}(\theta_{\mu,\nu}^{n})\eqdf\mathrm{span}\{\pi_{n}^{k,\ell}(u)\theta_{\mu,\nu}^{n}\,:u\in\U(n)\,\}\subset\tkld
\]
 is, under $\pi_{n}^{k,\ell}$, a $\U(n)$-subrepresentation of $\tkld$
isomorphic to $W_{n}^{\mu,\nu}.$
\end{cor}

\subsection{The Weingarten calculus\label{subsec:The-Weingarten-calculus}}

The Weingarten calculus is a method based on Schur--Weyl duality
that allows one to calculate integrals of products of matrix coefficients
in the defining representation of $\U(n)$ in terms of sums over permutations.
It was discovered initially by Weingarten \cite{weingarten1978asymptotic},
and developed further in works of Xu, Collins, and Collins--Śniady
in \cite{xu1997random,collins2003moments,CS}. 

We present two formulations of the Weingarten calculus. Given $k\in\N$,
$n\in\N$, the \emph{Weingarten function }with parameters $n,k$ is
the following element\footnote{Although not relevant here, classically the Weingarten function arises
as the multiplicative inverse of $\sum_{\sigma\in S_{k}}n^{\text{\#}\mathrm{cycles}(\sigma)}\sigma$
in $\C[S_{k}],$whenever $n\geq k$.} of $\C[S_{k}]$ \cite[eq. (9)]{CS}
\begin{equation}
\Wg_{n,k}\eqdf\frac{1}{(k!)^{2}}\sum_{\substack{\lambda\vdash k\\
\ell(\lambda)\leq n
}
}\frac{d_{\lambda}^{2}}{D_{\lambda}(n)}\sum_{\sigma\in S_{k}}\chi_{\lambda}(\sigma)\sigma.\label{eq:Wg-def}
\end{equation}
We write $\Wg_{n,k}(\sigma)$ for the coefficient of $\sigma$ in
(\ref{eq:Wg-def}). The following theorem was proved by Collins and
Śniady \cite[Cor. 2.4]{CS}.
\begin{thm}
\label{thm:Wg-vanilla}For $k\in\N$ and for $i_{1},i'_{1},j_{k},j'_{k},\ldots,i_{k},i'_{k},j_{k},j'_{k}\in[n]$
\begin{align}
 & \int_{u\in\U(n)}u_{i_{1}j_{1}}\cdots u_{i_{k}j_{k}}\overline{u_{i'_{1}j'_{1}}}\cdots\overline{u_{i'_{k}j'_{k}}}d\mu(u)\label{eq:wg-int}\\
= & \sum_{\sigma,\tau\in S_{k}}\delta_{i_{1}i'_{\sigma(1)}}\cdots\delta_{i_{k}i'_{\sigma(k)}}\delta_{j_{1}j'_{\tau(1)}}\cdots\delta_{j_{k}j'_{\tau(k)}}\Wg_{n,k}(\tau\sigma^{-1}),\nonumber 
\end{align}
where $\delta_{pq}$ is the Kronecker delta function.
\end{thm}

It is sometimes more flexible to reformulate Theorem \ref{thm:Wg-vanilla}
in terms of projections. Here $u\in\U(n)$ acts on $A\in\End((\C^{n})^{\otimes k})$
by $A\mapsto\pi_{n}^{k}(u)A\pi_{n}^{k}(u^{-1})$, $\pi_{n}^{k}:\U(n)\to\End((\C^{n})^{\otimes k})$
the diagonal action. Write $P_{n,k}$ for the orthogonal projection
in $\End((\C^{n})^{\otimes k})$ onto the $\U(n)$-invariant vectors.
The following proposition is due to Collins and Śniady \cite[Prop. 2.3]{CS}.
\begin{prop}[Collins--Śniady]
\label{prop:Weingarten}Let $n,k\in\N$. Suppose $A\in\End((\C^{n})^{\otimes k})$.
Then
\[
P_{n,k}[A]=\rho_{n}^{k}\left(\Phi[A]\cdot\Wg_{n,k}\right)
\]
where 
\[
\Phi[A]\eqdf\sum_{\sigma\in S_{k}}\tr(A\rho_{n}^{k}(\sigma^{-1}))\sigma.
\]
\end{prop}

Later we will need the following bound for the Weingarten function
due to Collins and Śniady \cite[Prop. 2.6]{CS}. For a permutation
$\sigma$ let $|\sigma|$ denote the minimum number of transpositions
that $\sigma$ can be written as a product of.
\begin{prop}
\label{prop:wg-bound}For any fixed $\sigma\in S_{k}$, $\Wg_{n,k}(\sigma)\ll_{k}n^{-k-|\sigma|}$
as $n\to\infty$.
\end{prop}

\subsection{Free groups and surface groups\label{subsec:Free-groups-and-surface-groups}}

Let $\F_{2g}\eqdf\langle a_{1},b_{1},\ldots,a_{g},b_{g}\rangle$ be
the free group on $2g$ generators $a_{1},b_{1},\ldots,a_{g},b_{g}$
and $R_{g}\eqdf[a_{1},b_{1}]\cdots[a_{g},b_{g}]\in\F_{2g}$. There
is a quotient map $\F_{2g}\to\Gamma_{g}$ given by reduction modulo
$R_{g}$. We say that $w\in\F_{2g}$ \emph{represents the conjugacy
class of $\gamma\in\Gamma_{g}$ }if the projection of $w$ to $\Gamma_{g}$
is in the conjugacy class of $\gamma$ in $\Gamma_{g}$. 

Given $w\in\F_{2g}$, we view $w$ as a combinatorial word in\linebreak{}
$a_{1},a_{1}^{-1},b_{1},b_{1}^{-1},\ldots,a_{g},a_{g}^{-1},b_{g},b_{g}^{-1}$
by writing it in reduced (shortest) form; i.e., $a_{1}$ does not
follow $a_{1}^{-1}$ etc. We say that $w$ is \emph{cyclically reduced}
if the first letter of its reduced word is not the inverse of the
last letter. The length $|w|$ of $w\in\F_{2g}$ is the length of
its reduced form word. We say $w\in\F_{2g}$ is a \emph{shortest element}
representing the conjugacy class of $\gamma\in\Gamma_{g}$ if it has
minimal length among all elements representing the conjugacy class
of $\gamma$. If $w$ is a shortest element representing some conjugacy
class in $\Gamma_{g}$ then $w$ is cyclically reduced.

For any group $H$, the commutator subgroup $[H,H]\leq H$ is the
subgroup generated by all elements of the form $[h_{1},h_{2}]\eqdf h_{1}h_{2}h_{1}^{-1}h_{2}^{-1}$
with $h_{1},h_{2}\in H$. If $\gamma\in[\Gamma_{g},\Gamma_{g}]$,
and $w$ represents the conjugacy class of $\gamma$, then $w\in[\F_{2g},\F_{2g}]$
(see \cite[\S 2.6]{MageeRURSG1}).

\subsection{Witten zeta functions\label{subsec:Witten-zeta-functions}}

Witten zeta functions appeared first in Witten's work \cite{Witten1991}
and were named by Zagier in \cite{Zagier}. The \emph{Witten zeta
function of $\SU(n)$ }is defined, for $s$ in a half-plane of convergence,
by 
\begin{equation}
\zeta(s;n)\eqdf\sum_{\substack{\text{(\ensuremath{\rho,W)\in\widehat{\SU(n)}}}}
}\frac{1}{(\dim W)^{s}}\label{eq:witten-def}
\end{equation}
where $\widehat{\SU(n)}$ denotes the equivalence classes of irreducible
representations of $\SU(n)$. Indeed, the series (\ref{eq:witten-def})
converges for $\Re(s)>\frac{2}{n}$ by a result of Larsen and Lubotzky
\cite[Thm. 5.1]{LL} (see also \cite[\S 2]{HS}). Also relevant to
this work is a result of Guralnick, Larsen, and Manack \cite[Thm 2., also eq. (7)]{GLM}
that states for fixed $s>0$
\begin{equation}
\lim_{n\to\infty}\zeta(s;n)=1.\label{eq:witten-limit}
\end{equation}

\subsection{Results of the prequel paper\label{subsec:Results-of-the}}

By \cite[Prop. 1.5]{MageeRURSG1}, if $\gamma\notin[\Gamma_{g},\Gamma_{g}]$,
then $\E_{g,n}[\tr_{\gamma}]=0$ for $n\ge n_{0}(\gamma)$. This proves
Theorem \ref{thm:main-convergence} in this case. \emph{Hence in the
rest of the paper we need only consider $\gamma\in[\Gamma_{g},\Gamma_{g}]$
and hence $w\in[\F_{2g},\F_{2g}]$ if $w\in\F_{2g}$ represents the
conjugacy class of $\gamma$.}

For each $w\in\F_{2g}$, we have a \emph{word map $w:\U(n)^{2g}\to\U(n)$
}obtained by substituting matrices for the generators of $\F_{2g}$.
For example, if \linebreak{}
$u_{1},v_{1},\ldots,u_{g},v_{g}\in\U(n)$ then $R_{g}(u_{1},v_{1},\ldots,u_{g},v_{g})=[u_{1},v_{1}]\cdots[u_{g},v_{g}]$.
We begin with the following result from the prequel paper \cite[Cor. 1.8]{MageeRURSG1}.
\begin{prop}
\label{prop:prequel-result}Suppose that $g\geq2$, $\gamma\in\Gamma_{g}$,
and $w\in\F_{2g}$ represents the conjugacy class of $\gamma$. For
any $B\in\N$ we have as $n\to\infty$
\begin{align}
\E_{g,n}[\tr_{\gamma}] & =\zeta(2g-2;n)^{-1}\sum_{\substack{\mu,\nu\,\mathrm{Young}\mathrm{\thinspace diagrams}\text{}\\
\text{\text{\ensuremath{\ell(\mu),\ell(\nu)\leq B,\mu_{1},\nu_{1}\leq B^{2}}}}
}
}D_{\mu,\nu}(n)\,\mathcal{I}_{n}(w,\mu,\nu)\label{eq:first-formula}\\
 & \thinspace\thinspace\thinspace\thinspace\thinspace\thinspace\thinspace\thinspace\thinspace\thinspace+O_{B,w,g}\left(n^{|w|}n^{-2\log B}\right).\nonumber 
\end{align}
where 
\begin{equation}
\mathcal{I}_{n}(w,\mu,\nu)\eqdf\int_{\SU(n)^{2g}}\tr(w(x))\overline{s_{\mu,\nu}(R_{g}(x)))}d\mu(x).\label{eq:I-def}
\end{equation}
\end{prop}

Notice that for $n\geq2B$ the right hand side of (\ref{eq:first-formula})
makes sense, i.e. $D_{\mu,\nu},s_{\mu,\nu}$ are well-defined. We
also have the following proposition that follows from \cite[Prop. 3.1]{MageeRURSG1}
together with $\overline{s_{\mu,\nu}}=s_{\nu,\mu}$.
\begin{prop}
\label{prop:passage-to-unitary}Let $w\in[\F_{2g},\F_{2g}]$. Then
for any fixed $\mu,\nu$ and $n\geq\ell(\mu)+\ell(\nu)$
\[
\mathcal{I}_{n}(w,\mu,\nu)=\J_{n}(w,\nu,\mu)\eqdf\int_{\U(n)^{2g}}\tr(w(x))s_{\nu,\mu}(R_{g}(x))d\mu(x).
\]
 
\end{prop}

This is convenient as it will allow us to use the Weingarten calculus
directly as it is presented in $\S\S$\ref{subsec:The-Weingarten-calculus}
for $\U(n)$ rather than $\SU(n)$. By using Proposition \ref{prop:passage-to-unitary},
taking a representative $w\in\F_{2g}$ of the conjugacy class of $\gamma$
and taking $B$ such that $|w|-2\log B\leq-1$ in Proposition \ref{prop:prequel-result}
we obtain the following result from which we begin the new arguments
of this paper.
\begin{cor}
\label{cor:starting-point}Let $\gamma\in[\Gamma_{g},\Gamma_{g}]$
and $w\in[\F_{2g},\F_{2g}]$ be a representative of the conjugacy
class of $\gamma\in\Gamma$. Then there exists a \uline{finite}
set $\tilde{\Omega}$ of pairs $(\mu,\nu)$ of Young diagrams such
that 
\[
\E_{g,n}[\tr_{\gamma}]=\zeta(2g-2;n)^{-1}\sum_{\substack{\text{\text{(\ensuremath{\mu,\nu)\in}\ensuremath{\tilde{\Omega}}}}}
}D_{\mu,\nu}(n)\mathcal{J}_{n}(w,\nu,\mu)+O_{w,g}\left(\frac{1}{n}\right).
\]
\end{cor}

As we know $\lim_{n\to\infty}\zeta(2g-2,n)=1$ by (\ref{eq:witten-limit}),
we have now reduced the proof of Theorem \ref{thm:main-convergence}
to establishing suitable bounds for the integrals $\mathcal{J}_{n}(w,\mu,\nu)$
where we can view $\mu,\nu$ as \emph{fixed} Young diagrams since
$\ensuremath{\tilde{\Omega}}$ is finite.

\section{Combinatorial Integration\label{sec:The-contribution-from-single-small-dim}}

\subsection{Setup and motivation\label{subsec:Setup-and-motivation}}

The main result of the rest of the paper is the following.
\begin{thm}
\label{thm:main-single-small-dim}Let $\gamma\in\Gamma_{g}$ with
$\gamma\neq\id$. Let $w\in\F_{2g}$ be a shortest element representing
the conjugacy class of $\gamma$. For each $k,\ell\in\N_{0}$ there
is a constant $C(w,k,\ell)>0$ such that for any $\mu\vdash k$, $\nu\vdash\ell$
\[
\left|D_{\mu,\nu}(n)\J_{n}(w,\mu,\nu)\right|\leq C(w,k,\ell)
\]
for all $n\in\N$. 
\end{thm}

Accordingly, since we know the large $n$ behavior of $D_{\mu,\nu}(n)$
from (\ref{eq:dimension-asymp}), in this section we wish to estimate
\[
\J_{n}(w,\mu,\nu)=\int_{\U(n)^{2g}}\tr(w(x))s_{\mu,\nu}(R_{g}(x))d\mu(x)
\]
for fixed $\mu\vdash k,\nu\vdash\ell$. 

\subsubsection*{What doesn't work }

We begin by discussing why the most straightforward approach to this
problem leads to serious complications. It is possible to approach
the problem by writing $s_{\mu,\nu}(h)$ as a fixed finite linear
combinations of functions
\[
p_{\mu'}(h)p_{\nu'}(h^{-1})
\]
where $p_{\mu'}(h)$ (resp. $p_{\nu'}(h^{-1})$) is a power sum symmetric
polynomial of the eigenvalues of $h$ (resp. $h^{-1}$ or $\bar{h}$).
See for example \cite[\S\S 3.3]{MageeRURSG1} for one way to do this.
The coefficients of this expansion are fixed, but not transparent,
since they involve Littlewood--Richardson coefficients. In any case,
this approach leads to writing $\J_{n}(w,\mu,\nu)$ as a finite linear
combination of integrals of the form
\begin{align}
 & \int_{\U(n)^{2g}}\tr(w(x))\tr\left(R_{g}(x)^{k_{1}}\right)\cdots\tr\left(R_{g}(x)^{k_{p}}\right)\label{eq:power-sum-integral-naive}\\
\cdot & \tr\left(R_{g}(x)^{-\ell_{1}}\right)\cdots\tr\left(R_{g}(x)^{-\ell_{q}}\right)d\mu(x)\nonumber 
\end{align}
where $\sum k_{j}=|\mu|$ and $\sum\ell_{j}=|\nu|$. 

The work of the author and Puder in \cite{MPunitary} gives a full
asymptotic expansion for (\ref{eq:power-sum-integral-naive}) as $n\to\infty$.
However these estimates are not sufficient for the current paper and
to motivate the rest of this $\S\ref{sec:The-contribution-from-single-small-dim}$
we explain briefly the issues involved. However, this discussion is
not needed to understand the arguments that we will make to prove
Theorem \ref{thm:main-single-small-dim}.

The main result of \cite{MPunitary} gives a full `genus' expansion
of (\ref{eq:power-sum-integral-naive}) in terms of surfaces and maps
on surfaces dictated by $w\in\F_{2g}$. Roughly speaking, every term
in this expansion comes from a homotopy class of map $f$ from an
orientable surface $\Sigma_{f}$ to $\bigvee_{i=1}^{2g}S^{1}$; to
contribute to (\ref{eq:power-sum-integral-naive}) the surface $\Sigma_{f}$
has one boundary component that maps to $w$ at the level of the fundamental
groups, $p$ boundary components that map respectively to $R_{g}^{k_{1}},\ldots,R_{g}^{k_{p}}$
at the level of fundamental groups, and $q$ boundary components that
map respectively to $R_{g}^{-\ell_{1}},\ldots,R_{g}^{-\ell_{q}}$
at the level of fundamental groups. The contribution of the pair $(f,\Sigma_{f})$
to (\ref{eq:power-sum-integral-naive}) is of the form $c(f,\Sigma_{f})n^{\chi(\Sigma_{f})}$;
the coefficient $c(f,\Sigma_{f})$ is an Euler characteristic of a
symmetry group of $(f,\Sigma_{f})$ and is not easy to calculate in
general. However, one could still hope to get decay of (\ref{eq:power-sum-integral-naive})
by controlling the possible $\chi(\Sigma_{f})$ that could appear. 

There are two issues with this. The first one is that if $w$ is not
the shortest element representing the conjugacy class of $\gamma$
then we get bounds that are not helpful. For a very simple example,
let $w=R_{g}^{\ell}$, $\gamma=\id_{\Gamma_{g}}$, and consider the
potential contribution from $p=0,q=1,\ell_{1}=\ell$. Then for any
$\nu$ with $|\nu|=\ell$ there is contribution to $\J_{n}(w,\emptyset,\nu)$
that is a multiple of 
\[
\int_{\U(n)^{2g}}\tr\left(R_{g}(x)^{\ell}\right)\tr\left(R_{g}(x)^{-\ell}\right)d\mu(x).
\]
Here, in the theory of \cite{MPunitary} there is a $(\Sigma_{f},f)$
that is an annulus, one boundary component corresponding to $w=R_{g}^{\ell}$
and one corresponding to $R_{g}^{-\ell}$, so we can only bound the
corresponding contribution to\linebreak{}
 $D_{\emptyset,\nu}(n)\J_{n}(w,\emptyset,\nu)$ by using \cite{MPunitary}
on the order of $D_{\emptyset,\nu}(n)\asymp n^{\ell}$. On the other
hand, any approach that works to establish Theorem \ref{thm:main-single-small-dim}
(for $\gamma\neq\id$) should extend to show when $\gamma=\id$, $D_{\emptyset,\nu}(n)\J_{n}(w,\emptyset,\nu)\ll n$
as $\E_{g,n}[\tr_{\id}]=n$.

Indeed, this phenomenon extends to words of the form $w_{0}R_{g}^{\ell}$
and more generally to words that are not shortest representatives
of some conjugacy class in $\Gamma_{g}$. It means that even if we
use something similar in spirit to \cite{MPunitary}, to prove Theorem
\ref{thm:main-single-small-dim} we must incorporate the theory of
shortest representative words. This indeed takes place in $\S\S$\ref{subsec:A-topological-result}-$\S\S\ref{subsec:Proof-of-chi-bound}$;
the topological result proved there hinges on this theory.

The second issue is a little more subtle and only appears for `mixed'
representations, i.e., both $\mu,\nu\neq\emptyset$. In this case,
suppose $w$ is a shortest element representing some conjugacy class
in $\Gamma_{g}$ and $w\in[\F_{2g},\F_{2g}]$. This means that there
is a pair $(f_{0},\Sigma_{f_{0}})$ where $\Sigma_{f_{0}}$ has one
boundary component that maps to $w$ at the level of the fundamental
groups. Let us take $\mu,\nu=(k),(k)$, i.e each Young diagram has
one row of $k$ boxes. This means we get a potential contribution
to $D_{\mu,\nu}(n)\J_{n}(w,\mu,\nu)$ that is a constant multiple
of
\begin{equation}
D_{(k),(k)}(n)\int_{\U(n)^{2g}}\tr(w(x))\tr\left(R_{g}(x)^{k}\right)\tr\left(R_{g}(x)^{-k}\right)d\mu(x)\label{eq:k-k-int}
\end{equation}
Now, for every $k\in\N$, there is $(f,\Sigma_{f})$ contributing
to (\ref{eq:k-k-int}) with one component that is $(f_{0},\Sigma_{f_{0}})$
and the other is an annulus with boundary components corresponding
to $R_{g}^{k},R_{g}^{-k}$. Since the annulus has Euler characteristic
$0$, and $D_{(k),(k)}\asymp n^{2k}$, the order of this contribution
to $D_{(k),(k)}(n)\J_{n}(w,(k),(k))$ is potentially $\gg n^{2k}n^{\chi(\Sigma_{f_{0}})}$.
For large enough $k$ the exponent here is arbitrarily large, which
is clearly catastrophic. In reality, this contribution must cancel
with some other contribution but we do not know how to see these cancellations.

\emph{This ends the discussion of the difficulties of the most straightforward
approach to the problems of this paper.}

\subsubsection*{What does work}

To bypass the previous issues we produce a refined version of the
Weingarten calculus that leads to a restricted set of surfaces, for
instance, not including the ones causing the problem above as well
as all generalizations of this issue. 

The basic approach is the following. Instead of trying to deal with
a complicated formula for $s_{\mu,\nu}(R_{2}(x))$ (as above) we instead
use the copy $W_{n}(\theta_{\mu,\nu}^{n})$ of $W_{n}^{\mu,\nu}$
in $\tkld$ that we found in Corollary \ref{cor:subspace-construct}.
In $\S\S\ref{subsec:A-projection-formula}$ we compute the orthogonal
projection $\q_{\theta}$ from $\tkl$ (note; not $\tkld$) onto $W_{n}(\theta_{\mu,\nu}^{n})$
(Proposition \ref{prop:proj-form}). In the formula we obtain, we
give bounds on the coefficients appearing therein (Lemma \ref{lem:coeff-bound}).
In addition, we also remember that $q_{\theta}\in\End(\tkld)$; this
fact is not obvious from our formula but turns out to be vital going
forward. 

The calculation of $\q_{\theta}$ is extra to, but in the same spirit
as, the vanilla Weingarten calculus, which is why we claim to have
refined the Weingarten calculus here.

In the expression for $\J_{n}(w,\mu,\nu)$ we now write 
\[
s_{\mu,\nu}(R_{2}(x))=\Tr_{\tkl}(A\q_{\theta}B\q_{\theta}A^{-1}\q_{\theta}B^{-1}\q_{\theta}C\q_{\theta}D\q_{\theta}C^{-1}\q_{\theta}D^{-1}\q_{\theta})
\]
 where $A,B,C,D$ are the images of the generators of $\Gamma_{2}$
under $x$. Then the entire integral of $\tr(w(x))s_{\mu,\nu}(R_{2}(x))$
is done using the usual Weingarten calculus. The fact that $q_{\theta}\in\End(\tkld)$
intervenes at a critical point to show that certain contributions
from the classical Weingarten calculus cancel and lead to restrictions
on the non-zero contributions. Precisely, the restriction we obtain
is summarized in the \textbf{forbidden matching} property below ($\S\S$\ref{subsec:A-combinatorial-integration})
and property \textbf{P4}\textbf{\emph{ }}($\S\S$\ref{subsec:A-topological-result}).

\subsection{Proof of Theorem \ref{thm:main-single-small-dim} when $k=\ell=0$\label{subsec:Proof-of-k=00003Dl=00003D0}}

Here we give a proof of Theorem \ref{thm:main-single-small-dim} when
$k=\ell=0$. This will allow us to bypass the slightly confusing issue
of using the Weingarten function $\Wg_{n,k+\ell}$ when $k+\ell=0$
in $\S\S$\ref{subsec:A-projection-formula}.

If $k=\ell=0$ then the only possible $\mu\vdash k$, $\nu\vdash\ell$
are empty Young diagrams $\mu=\nu=\emptyset$, and $W_{n}^{\emptyset,\emptyset}$
is the trivial representation of $\U(n)$, so $D_{\emptyset,\emptyset}(n)=1$
for all $n\geq1$ and $s_{\emptyset,\emptyset}(h)=1$ for all $h\in\U(n)$.
We then have 
\begin{equation}
D_{\emptyset,\emptyset}(n)\J_{n}(w,\emptyset,\emptyset)=\J_{n}(w,\emptyset,\emptyset)=\int_{\U(n)^{2g}}\tr(w(x))d\mu(x).\label{eq:empty-empty}
\end{equation}
If $w\in\F_{2g}$ is a cyclically shortest word representing the conjugacy
class of $\gamma\in\Gamma_{g}$ with $\gamma\neq\id$, then $w\neq\id$.
It then follows from (\ref{eq:haar-word-integral}) that $D_{\emptyset,\emptyset}(n)\J_{n}(w,\emptyset,\emptyset)=o_{w}(n)$
as $n\to\infty$, but in fact, (\ref{eq:empty-empty}) is given by
a rational function of $n$ for $n\geq n_{0}(w)$ by a straightforward
application of the Weingarten calculus \cite{MPunitary}. This implies
$D_{\emptyset,\emptyset}(n)\J_{n}(w,\emptyset,\emptyset)=O_{w}(1)$
as $n\to\infty$, as required.

\emph{This proves Theorem \ref{thm:main-single-small-dim} when $k=\ell=0$.
Hence in the rest of this $\S$\ref{sec:The-contribution-from-single-small-dim}
we can assume $k+\ell>0$.}

\subsection{A projection formula\label{subsec:A-projection-formula}}

Here we develop an integral calculus that is more powerful than the
usual Weingarten calculus and allows us to directly tackle $\J_{n}(w,\mu,\nu)$
without writing it in terms of integrals as in (\ref{eq:power-sum-integral-naive}).
The key point is that our method leads to the \textbf{forbidden matchings
}property of $\S\S$\ref{subsec:A-combinatorial-integration} and
property \textbf{P4}\textbf{\emph{ }}of $\S\S$\ref{subsec:A-topological-result}. 

We now view $k,\ell$, $\mu\vdash k,\nu\vdash\ell$ as fixed, assume
$k+\ell>0$, $n\geq\ell(\mu)+\ell(\nu)$, and write $\theta=\theta_{\mu,\nu}^{n}$
as in (\ref{eq:theta-def}), suppressing the dependence on $n$. Let
$W_{n}(\theta)$ be defined as in Corollary \ref{cor:subspace-construct}.
Thus $W_{n}(\theta)$ is an irreducible summand of $\tkld$ isomorphic
to $W_{n}^{\mu,\nu}$ for the group $\U(n)$.

\emph{In the remainder of the paper we drop the dependence of our
notation on $n$ whenever it adds clarity.}

Our first task is to compute the orthogonal projection $\q_{\theta}$
onto $W(\theta)$. Let $P_{\theta}$ denote the orthogonal projection
in $\tkl$ onto $\theta$. We also view $P_{\theta}$ as an element
of $\End(\tkld)$ by restriction. 

Under the canonical isomorphism $\End(\tkld)\cong\tkld\otimes\left(\tkld\right)^{\vee}$
we have $P_{\theta}\cong\frac{\theta\otimes\theta^{\vee}}{\|\theta\|^{2}}$
and also from (\ref{eq:theta-def})
\begin{equation}
P_{\theta}=\frac{1}{\|\theta\|^{2}}\rho^{k}(\p_{\mu})\hat{\rho}^{\ell}(\p_{\nu})[\tilde{\theta}_{\mu,\nu}\otimes\tilde{\theta}_{\mu,\nu}^{\vee}]\rho^{k}(\p_{\mu})\hat{\rho}^{\ell}(\p_{\nu});\label{eq:P-xi-formula}
\end{equation}
here the inner square bracket is interpreted as an element of $\End(\tkld)$.
By Schur's lemma we have
\begin{equation}
\q_{\theta}=D_{\mu,\nu}(n)\int_{h\in\U(n)}\pi(h)P_{\theta}\pi(h^{-1})d\mu(h)\label{eq:qxi-formula1}
\end{equation}
since the right hand side is an element of $\End(W(\theta))\subset\End(\tkl)$
that commutes with $\pi^{k,\ell}(\U(n))$, so it is a multiple of
$\q_{\theta}$, and it has the correct trace.

On the other hand, we can view $\tkl\otimes\left(\tkld\right)^{\vee}\cong\T_{n}^{k+\ell,k+\ell}$
by the canonical isomorphism 
\[
\tkl\otimes\left(\tkld\right)^{\vee}\cong(\C^{n})^{\otimes k}\otimes\text{\ensuremath{\left((\C^{n})^{\otimes\ell}\right)}}^{\vee}\otimes\text{\ensuremath{\left((\C^{n})^{\otimes k}\right)}}^{\vee}\otimes(\C^{n})^{\otimes\ell}
\]
followed by the following fixed isomorphism
\begin{equation}
\varphi:e_{I}^{J}\otimes\check{e}_{I'}^{J'}\mapsto e_{I\sqcup J'}\otimes\check{e}_{I'\sqcup J}.\label{eq:varphi1}
\end{equation}
Finally, there is a canonical isomorphism $\T_{n}^{k+\ell,k+\ell}\cong\End((\C^{n})^{\otimes k+\ell})$.
So combining these we fix isomorphisms
\begin{equation}
\End(\tkl)\cong\tkld\otimes\left(\tkld\right)^{\vee}\xrightarrow[\varphi]{\sim}\T_{n}^{k+\ell,k+\ell}\cong\End((\C^{n})^{\otimes k+\ell}).\label{eq:varphi2}
\end{equation}
We view the outer two isomorphisms as fixed identifications. These
isomorphisms are of unitary representations of $\U(n)$ when everything
is given its natural inner product. Moreover for $\sigma=(\sigma_{1},\sigma_{2})\in S_{k}\times S_{\ell}$
and $\tau=(\tau_{1},\tau_{2})\in S_{k}\times S_{l}$ we have for $A\in\End(\tkl)$
\begin{equation}
\varphi[\rho^{k}(\sigma_{1})\hat{\rho}^{\ell}(\sigma_{2})A\rho^{k}(\tau_{1})\hat{\rho}^{\ell}(\tau_{2})]=\rho^{k+\ell}(\sigma_{1},\tau_{2}^{-1})\varphi[A]\rho^{k+\ell}(\tau_{1},\sigma_{2}^{-1}),\label{eq:sym-twisted-intertwiner}
\end{equation}
recalling that $\rho^{k+\ell}:\C[S_{k+\ell}]\to\End((\C^{n})^{\otimes k+\ell})$
is the representation by permuting coordinates. 

We now return to the calculation of $\q_{\theta}$ in (\ref{eq:qxi-formula1}).
We have
\begin{equation}
\q_{\theta}=D_{\mu,\nu}(n)\varphi^{-1}[P_{n,k+\ell}[\varphi(P_{\theta})]]\label{eq:axi1}
\end{equation}
where $P_{n,k+\ell}$ is the projection onto the $\U(n)$-invariant
vectors (by conjugation) in $\End((\C^{n})^{\otimes k+\ell})$. This
can now be done using the classical Weingarten calculus. By Proposition
\ref{prop:Weingarten} we have
\begin{equation}
P_{n,k+\ell}[\varphi(P_{\theta})]=\rho^{k+\ell}\left(\Phi[\varphi(P_{\theta})]\cdot\Wg_{n,k+\ell}\right)\label{eq:wg-proj-form}
\end{equation}
where 
\begin{align*}
\Phi[\varphi(P_{\theta})] & =\sum_{\sigma\in S_{k+\ell}}\tr(\varphi(P_{\theta})\rho^{k+\ell}(\sigma^{-1}))\sigma.
\end{align*}
By (\ref{eq:sym-twisted-intertwiner}) and (\ref{eq:P-xi-formula}),
and since e.g. $\chi_{\mu}(g)=\chi_{\mu}(g^{-1})$ we obtain
\begin{align*}
\varphi(P_{\theta}) & =\frac{1}{\|\theta\|^{2}}\varphi\left(\rho^{k}(\p_{\mu})\hat{\rho}^{\ell}(\p_{\nu})[\tilde{\theta}_{\mu,\nu}\otimes\tilde{\theta}_{\mu,\nu}^{\vee}]\rho^{k}(\p_{\mu})\hat{\rho}^{\ell}(\p_{\nu})\right)\\
 & =\frac{1}{\|\theta\|^{2}}\rho^{k+\ell}(\p_{\mu\otimes\nu})\varphi\left(\tilde{\theta}_{\mu,\nu}\otimes\tilde{\theta}_{\mu,\nu}^{\vee}\right)\rho^{k+\ell}(\p_{\mu\otimes\nu})
\end{align*}
where 
\[
\p_{\mu\otimes\nu}\eqdf\frac{d_{\mu}d_{\nu}}{k!\ell!}\sum_{\sigma=(\sigma_{1},\sigma_{2})\in S_{k}\times S_{\ell}}\chi_{\mu}(\sigma_{1})\chi_{\nu}(\sigma_{2})\sigma\in\C[S_{k+\ell}].
\]
Now using that $\Phi$ is a $\C[S_{k+\ell}]$-bimodule morphism \cite[Prop. 2.3 (1)]{CS}
we obtain
\begin{align*}
\Phi[\varphi(P_{\theta})] & =\frac{1}{\|\theta\|^{2}}\p_{\mu\otimes\nu}\Phi\left[\varphi\left(\tilde{\theta}_{\mu,\nu}\otimes\tilde{\theta}_{\mu,\nu}^{\vee}\right)\right]\p_{\mu\otimes\nu}\\
 & =\frac{1}{\|\theta\|^{2}}\p_{\mu\otimes\nu}\left(\sum_{\sigma\in S_{k+\ell}}\tr\left(\varphi\left(\tilde{\theta}_{\mu,\nu}\otimes\tilde{\theta}_{\mu,\nu}^{\vee}\right)\rho^{k+\ell}\left(\sigma^{-1}\right)\right)\sigma\right)\p_{\mu\otimes\nu}.
\end{align*}
Now, $\tr(\varphi\left(\tilde{\theta}_{\mu,\nu}\otimes\tilde{\theta}_{\mu,\nu}^{\vee}\right)\rho^{k+\ell}(\sigma^{-1}))$
is equal to $1$ if and only if $\sigma$ is in $S_{\mu}\times S_{\nu}\leq S_{k}\times S_{\ell}$,
and is $0$ otherwise. So we obtain
\begin{align*}
\Phi[\varphi(P_{\theta})] & =\frac{1}{\|\theta\|^{2}}\p_{\mu\otimes\nu}\left(\sum_{\sigma\in S_{\mu}\times S_{\nu}}\sigma\right)\p_{\mu\otimes\nu},
\end{align*}
hence from (\ref{eq:wg-proj-form})
\[
P_{n,k+\ell}[\varphi(P_{\theta})]=\rho^{k+\ell}\left(z_{\theta}\right)
\]
where
\begin{equation}
z_{\theta}\eqdf\sum_{\tau\in S_{k+\ell}}z_{\theta}(\tau)\tau\eqdf\frac{1}{\|\theta\|^{2}}\p_{\mu\otimes\nu}\left(\sum_{\sigma\in S_{\mu}\times S_{\nu}}\sigma\right)\p_{\mu\otimes\nu}\Wg_{n,k+\ell}\in\C[S_{k+\ell}].\label{eq:z-theta-def}
\end{equation}
Therefore we obtain the following proposition.
\begin{prop}
\label{prop:proj-form}We have 
\[
\q_{\theta}=D_{\mu,\nu}(n)\varphi^{-1}[\rho^{k+\ell}(z_{\theta})].
\]
\end{prop}

We can use the bound for the coefficients of $\Wg_{n,k+\ell}$ from
Proposition \ref{prop:wg-bound} to infer a bound on the coefficients
$z_{\theta}(\tau)$ . For $\sigma\in S_{k+\ell}$, let $\|\sigma\|_{k,\ell}$
denote the minimum $m$ for which 
\[
\sigma=\sigma_{0}t_{1}t_{2}\cdots t_{m}
\]
 where $\sigma_{0}\in S_{k}\times S_{\ell}$ and $t_{1},\ldots,t_{m}$
are transpositions in $S_{k+\ell}$.
\begin{lem}
\label{lem:coeff-bound}For all $\tau\in S_{k+\ell}$ and $\theta=\theta_{\mu,\nu}$
as above,\linebreak{}
$z_{\theta}(\tau)=O_{k,\ell}(n^{-k-\ell-\|\tau\|_{k,\ell}})$ as $n\to\infty$.
\end{lem}

\begin{proof}
Referring to (\ref{eq:z-theta-def}), as $n\to\infty$, $\|\theta\|^{-2}=O_{k,\ell}(1)$
by Lemma \ref{lem:theta-norm} and the coefficients of\\
$\p_{\mu\otimes\nu}\left(\sum_{\sigma\in S_{\mu}\times S_{\nu}}\sigma\right)\p_{\mu\otimes\nu}$
are clearly $O_{k,\ell}(1)$, so $z_{\theta}$ has the form
\[
\left(\sum_{\sigma\in S_{k}\times S_{\ell}}A(\sigma)\sigma\right)\Wg_{n,k+\ell}
\]
 where each $A(\sigma)$ is $O_{k,\ell}(1)$. This means 
\[
z_{\theta}(\tau)=\sum_{\sigma\in S_{k}\times S_{\ell},\,\sigma'\in S_{k+\ell}\,:\,\sigma\sigma'=\tau}A(\sigma)\Wg_{n,k+\ell}(\sigma').
\]
The order of any of the finitely many summands above is $n^{-k-\ell-|\sigma'|}$
by Proposition \ref{prop:wg-bound}, and the minimum possible value
of $|\sigma'|$ is $\|\tau\|_{k,\ell}$.
\end{proof}
Before moving on, it is useful to explain the operator $\varphi^{-1}[\rho^{k+\ell}(\pi)]$
for $\pi\in S_{k+\ell}$. For $I=(i_{1},\ldots,i_{k+\ell})$ let $I'(I;\pi)\eqdf i_{\pi(1)},\ldots,i_{\pi(k)}$
and $J'(I;\pi)\eqdf i_{\pi(k+1)},\ldots,i_{\pi(k+\ell)}$. As an element
of $(\C^{n})^{\otimes k+\ell}\otimes\left(\left(\C^{n}\right)^{\vee}\right){}^{\otimes k+\ell}$,
$\rho^{k+\ell}(\pi)$ is given by
\[
\sum_{I=(i_{1},\ldots,i_{k}),J=(j_{k+1},\ldots,j_{k+\ell})}e_{I'(I\sqcup J;\pi)\sqcup J'(I\sqcup J;\pi)}\otimes\check{e}_{I\sqcup J},
\]
so from (\ref{eq:varphi1}) 
\begin{align}
\varphi^{-1}[\rho^{k+\ell}(\pi)] & =\sum_{I=(i_{1},\ldots,i_{k}),J=(j_{k+1},\ldots,j_{k+\ell})}e_{I'(I\sqcup J;\pi)}^{J}\otimes\check{e}_{I}^{J'(I\sqcup J;\pi)}.\label{eq:matching-end}
\end{align}

\subsection{A combinatorial integration formula\label{subsec:A-combinatorial-integration}}

In this rest of this $\S\ref{sec:The-contribution-from-single-small-dim}$
we assume $g=2$. All proofs extend to $g\geq3$. We write $\{a,b,c,d\}$
for the generators of $\F_{4}$ and $R\eqdf[a,b][c,d]$. Assume both
$\gamma$ and $w$ are not the identity and $w\in[\F_{4},\F_{4}]$
according to the remarks at the beginning of $\S\S$\ref{subsec:Results-of-the}.
We write $w$ in reduced form:
\begin{equation}
w=f_{1}^{\epsilon_{1}}f_{2}^{\epsilon_{2}}\ldots f_{|w|}^{\epsilon_{|w|}},\quad\epsilon_{u}\in\{\pm1\},\,f_{u}\in\{a,b,c,d\},\label{eq:combinatorial-word}
\end{equation}
where if $f_{u}=f_{u+1}$, then $\epsilon_{u}=\epsilon_{u+1}$. For
$f\in\{a,b,c,d\}$ let $p_{f}$ denote the number of occurrences of
$f^{+1}$ in (\ref{eq:combinatorial-word}). 
The expression (\ref{eq:combinatorial-word}) implies that for $h\eqdf(h_{a},h_{b},h_{c},h_{d})\in\U(n)^{4}$,
\begin{equation}
\tr(w(h))=\sum_{i_{j}\in[n]}(h_{f_{1}}^{\epsilon_{1}})_{i_{1}i_{2}}(h_{f_{2}}^{\epsilon_{2}})_{i_{2}i_{3}}\cdots(h_{f_{|w|}}^{\epsilon_{|w|}})_{i_{|w|}i_{1}}.\label{eq:word-trace}
\end{equation}

Working with this expression will be cumbersome so we explain a diagrammatic
way to think about (\ref{eq:word-trace}). This will be the starting
point for how we eventually understand $\J_{n}(w,\mu,\nu)$ in terms
of decorated surfaces. We begin with a collection of intervals as
follows.

\emph{\uline{\mbox{$w$}-intervals and the \mbox{$w$}-loop}}

Firstly, for every $j\in[w]$, with $f_{j}=f$ as in (\ref{eq:combinatorial-word})
and $\epsilon_{j}=1$ we take a copy of $[0,1]$ and direct it from
$0$ to $1$.

In our constructions, every interval will have two directions: the
\emph{intrinsic direction} (which is the direction from $0$ to $1$)
and the \emph{assigned direction. }In the case just discussed, these
agree, but in general they will not.

We write $[0,1]_{f,j,w}$ for such an interval and $\mathfrak{I}_{f,w}^{+}$
for the collection of these intervals.

For every $j\in[w]$, with $f_{j}=f$ as in (\ref{eq:combinatorial-word})
and $\epsilon_{j}=-1$ we take a copy of $[0,1]$ and direct this
interval from $1$ to $0$. We write $[0,1]_{f^{-1},j,w}$ for such
an interval and $\mathfrak{I}_{f,w}^{-}$ for the collection of these
intervals.

All the intervals described above are called \emph{$w$-intervals}.
There are $|w|$ of these intervals in total.

\emph{\uline{w-intermediate-intervals.}}

Between each $[0,1]_{f_{j}^{\epsilon_{j}},j,w}$ and $[0,1]_{f_{j+1}^{\epsilon_{j+1}},j+1,w}$
we add a new interval connecting $1_{f_{j}^{\epsilon_{j}},j,w}$ to
$0_{f_{j+1}^{\epsilon_{j+1}},j+1,w}$, where the indices $j$ run
mod $|w|$. These intervals added are called $w$\emph{-intermediate-intervals.
}Note that these intervals together with the $w$-intervals now form
a closed cycle that is paved by $2|w|$ intervals alternating between
$w$-intervals and $w$-intermediate-intervals. Starting at $[0,1]_{f_{1}^{\epsilon_{1}},1,w}$,
reading the directions and $f$-labels of the $w$-intervals so that
every $w$-interval is traversed from $0$ to $1$ spells out the
word $w$. The resulting circle is called the $w$-loop and the previously
defined orientation of this loop is now fixed. See Figure \ref{fig:w-loops}
for an illustration of the $w$-loop in a particular example.

\begin{figure}

\begin{centering}
\includegraphics[scale=0.6]{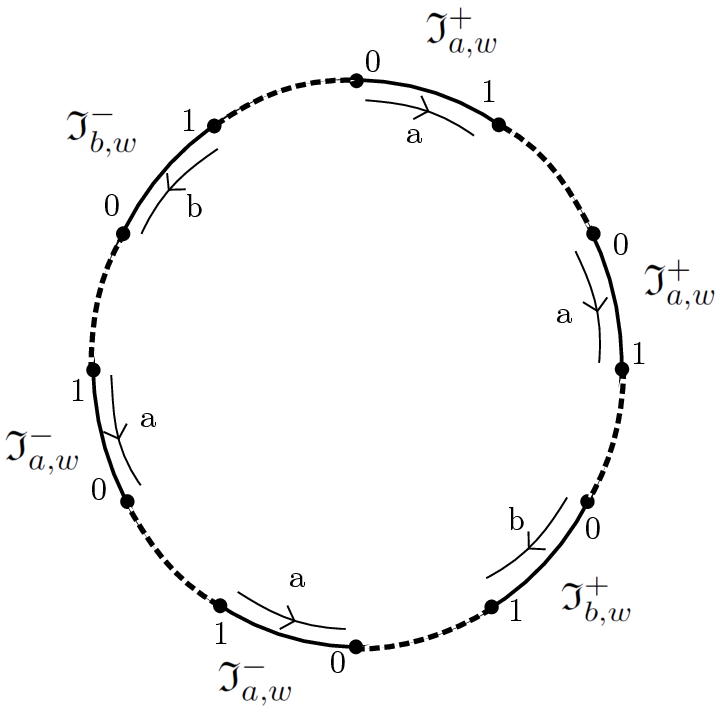}\caption{\label{fig:w-loops}Illustration of the $w$-loop for $w=a^{2}ba^{-2}b^{-1}$.
The solid intervals are $w$-intervals and the dashed intervals are
$w$-intermediate-intervals. We also label each interval by the set
e.g. $\protect\II_{a,w}^{+}$ to which they belong.}
\par\end{centering}
\end{figure}

We now view the indices $i_{j}$ as an assignment
\begin{align*}
\a & :\{\text{end-points of \ensuremath{w}-intervals}\}\to[n],\\
\a(0_{f,j,w})\eqdf i_{j},\, & \a(1_{f,j,w})=i_{j+1},\,\a(0_{f^{-1},j,w})=i_{j},\,\a(1_{f^{-1},j,w})=i_{j+1}.
\end{align*}
The condition that $\a$ comes from a single collection of $i_{j}$
is precisely that \emph{if two end points of $w$-intervals are connected
by a $w$-intermediate-interval, they are assigned the same value
by $\a$. }Let $\A(w)$ denote the collection of such $\a$. If $I$
is any copy of $[0,1]$ we write $0_{I}$ for the copy of $0$ and
$1_{I}$ for the copy of $1$ in $I$. We can now write
\[
\tr(w(h))=\sum_{\a\in\A(w)}\prod_{f\in\{a,b,c,d\}}\text{\ensuremath{\left(\prod_{\i\in\II_{f,w}^{+}}h_{\a(0_{\i})\a(1_{\i})}\right)\left(\prod_{\j\in\II_{f,w}^{-}}\overline{h_{\a(1_{\j})\a(0_{\j})}}\right)}}.
\]

Now let $v_{p}$ be an orthonormal basis for $W_{n}(\theta)$. We
have 
\begin{align*}
s_{\mu,\nu}(R_{g}(h_{a},h_{b},h_{c},h_{d}))= & \sum_{p_{i}}\langle h_{a}v_{p_{2}},v_{p_{1}}\rangle\langle h_{b}v_{p_{3}},v_{p_{2}}\rangle\langle h_{a}^{-1}v_{p_{4}},v_{p_{3}}\rangle\langle h_{b}^{-1}v_{p_{5}},v_{p_{4}}\rangle\\
 & \langle h_{c}v_{p_{6}},v_{p_{5}}\rangle\langle h_{d}v_{p_{7}},v_{p_{6}}\rangle\langle h_{c}^{-1}v_{p_{8}},v_{p_{7}}\rangle\langle h_{d}^{-1}v_{p_{1}},v_{p_{8}}\rangle.
\end{align*}
Here we have written e.g. $h_{a}v_{p_{2}}$ for $\pi_{n}^{k,\ell}(h_{a})v_{p_{2}}$
to make things easier to read. Next we write each $v_{p}=\sum_{I,J}\beta_{pI}^{J}e_{I}^{J}$,
where $\beta_{pI}^{J}\eqdf\langle v_{p},e_{I}^{J}\rangle.$ We then
have 
\begin{align}
 & \langle h_{a}v_{p_{2}},v_{p_{1}}\rangle\langle h_{b}v_{p_{3}},v_{p_{2}}\rangle\langle h_{a}^{-1}v_{p_{4}},v_{p_{3}}\rangle\langle h_{b}^{-1}v_{p_{5}},v_{p_{4}}\rangle\nonumber \\
 & \times\langle h_{c}v_{p_{6}},v_{p_{5}}\rangle\langle h_{d}v_{p_{7}},v_{p_{6}}\rangle\langle h_{c}^{-1}v_{p_{8}},v_{p_{7}}\rangle\langle h_{d}^{-1}v_{p_{1}},v_{p_{8}}\rangle\nonumber \\
= & \sum_{\r_{f},\RR_{f},\V_{f},\v_{f},\UU_{f},\u_{f},\s_{f},\SS_{f}}\beta_{p_{2}\s_{a}}^{\V_{a}}\bar{\beta}_{p_{1}\r_{a}}^{\UU_{a}}\beta_{p_{3}\s_{b}}^{\V_{b}}\bar{\beta}_{p_{2}\r_{b}}^{\UU_{b}}\beta_{p_{4}\RR_{a}}^{\u_{a}}\bar{\beta}_{p_{3}\SS_{a}}^{\v_{a}}\beta_{p_{5}\RR_{b}}^{\u_{b}}\bar{\beta}_{p_{4}\SS_{b}}^{\v_{b}}\nonumber \\
 & \beta_{p_{6}\s_{c}}^{\V_{c}}\bar{\beta}_{p_{5}\r_{c}}^{\UU_{c}}\beta_{p_{7}\s_{d}}^{\V_{d}}\bar{\beta}_{p_{6}\r_{d}}^{\UU_{d}}\beta_{p_{8}\RR_{c}}^{\u_{c}}\bar{\beta}_{p_{7}\SS_{c}}^{\v_{c}}\beta_{p_{1}\RR_{d}}^{\u_{d}}\bar{\beta}_{p_{8}\SS_{d}}^{\v_{d}}\nonumber \\
 & \langle h_{a}e_{\s_{a}}^{\V_{a}},e_{\r_{a}}^{\UU_{a}}\rangle\langle h_{b}e_{\s_{b}}^{\V_{b}},e_{\r_{b}}^{\UU_{b}}\rangle\langle h_{a}^{-1}e_{\RR_{a}}^{\u_{a}},e_{\SS_{a}}^{\v_{a}}\rangle\langle h_{b}^{-1}e_{\RR_{b}}^{\u_{b}},e_{\SS_{b}}^{\v_{b}}\rangle\nonumber \\
 & \langle h_{c}e_{\s_{c}}^{\V_{c}},e_{\r_{c}}^{\UU_{c}}\rangle\langle h_{d}e_{\s_{d}}^{\V_{d}},e_{\r_{d}}^{\UU_{d}}\rangle\langle h_{c}^{-1}e_{\RR_{c}}^{\u_{c}},e_{\SS_{c}}^{\v_{c}}\rangle\langle h_{d}^{-1}e_{\RR_{d}}^{\u_{d}},e_{\SS_{d}}^{\v_{d}}\rangle.\label{eq:pre-diagrammatic}
\end{align}
We calculate
\begin{align}
 & \langle h_{f}e_{\s_{f}}^{\V_{f}},e_{\r_{f}}^{\UU_{f}}\rangle\langle h_{f}^{-1}e_{\RR_{f}}^{\u_{f}},e_{\SS_{f}}^{\v_{f}}\rangle\nonumber \\
= & \langle h_{f}e_{\s_{f}},e_{\r_{f}}\rangle\overline{\langle h_{f}e_{\V_{f}},e_{\UU_{f}}\rangle}\overline{\langle h_{f}e_{\SS_{f}},e_{\RR_{f}}\rangle}\langle h_{f}e_{\v_{f}},e_{\u_{f}}\rangle\nonumber \\
= & \langle h_{f}e_{\s_{f}\sqcup\v_{f}},e_{\r_{f}\sqcup\u_{f}}\rangle\overline{\langle h_{f}e_{\SS_{f}\sqcup\V_{f}},e_{\RR_{f}\sqcup\UU_{f}}\rangle}.\label{eq:plain-form}
\end{align}
We now want a diagrammatic interpretation of (\ref{eq:pre-diagrammatic})
similarly to before. We make the following constructions.

\emph{\uline{\mbox{$R$}-intervals.}}

For each $j\in[k]$ and $f\in\{a,b,c,d\}$, we make a copy of $[0,1]$,
direct it from $0$ to $1$, label it by $f$, and also number it
by $j$. We write $\II_{f,R}^{+}$ for the collection of these intervals.
These correspond to occurrences of $f$ in $R$.

For each $j\in[k]$ and $f\in\{a,b,c,d\}$, we make a copy of $[0,1]$,
direct it from $1$ to $0$, label it by $f$, and also number it
by $j$. We write $\II_{f,R}^{-}$ for the collection of these intervals.
These correspond to occurrences of $f^{-1}$ in $R$.

(These two constructions of $k$ intervals correspond to the presence
of $f$ and $f^{-1}$ each exactly once in $R$.)

These intervals are called \emph{$R$-intervals. }There are $8k$
$R$-intervals in total (for general $g$, there are $4gk$ of these
intervals).

\uline{\mbox{$R^{-1}$}}\emph{\uline{-intervals.}}

For each $j\in[k+1,k+\ell]$ and $f\in\{a,b,c,d\}$, we make a copy
of $[0,1]$, direct it from $0$ to $1$, label it by $f$, and also
number it by $j$. We write $\II_{f,R^{-1}}^{+}$ for the collection
of these intervals. These correspond to occurrences of $f$ in $R^{-1}$.

For each $j\in[k+1,k+\ell${]} and $f\in\{a,b,c,d\}$, we make a copy
of $[0,1]$, direct it from $1$ to $0$, label it by $f$, and also
number it by $j$. We write $\II_{f,R^{-1}}^{-}$ for the collection
of these intervals. These correspond to occurrences of $f^{-1}$ in
$R^{-1}$. 

These intervals are called \emph{$R^{-1}$-intervals. }There are $8\ell$
$R^{-1}$ intervals in total (for general $g$, there are $4g\ell$
of these intervals). See Figure \ref{fig:R-loops} for an illustration
of the $R$ and $R^{-1}$ intervals.

\begin{figure}

\begin{centering}
\includegraphics[scale=0.5]{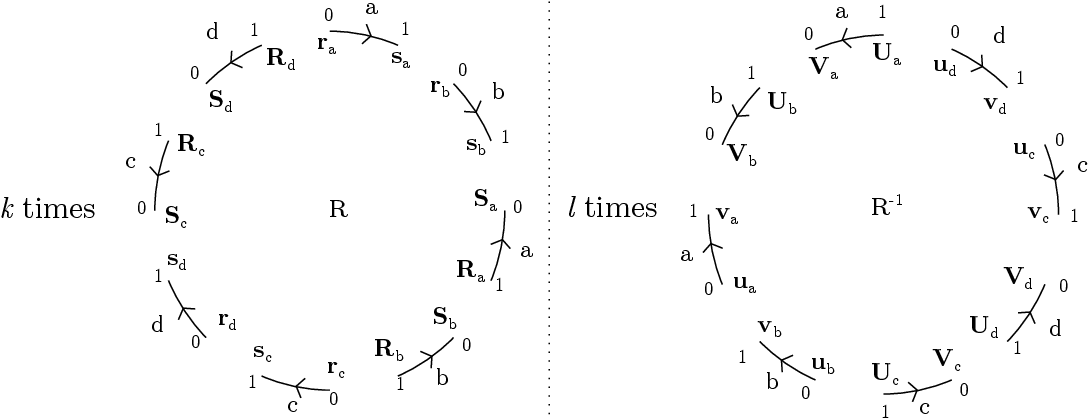}\caption{\label{fig:R-loops}Here is shown the $R$-intervals (left) and the
$R^{-1}$-intervals (right). We have indicated their assigned direction
and label (which $f$ they correspond to). We have also, for each
endpoint of an interval, indicated which index function, e.g. $\protect\r_{a}$,
has this endpoint in its domain.}
\par\end{centering}
\end{figure}

We now view (by identifying endpoints of intervals with the given
numbers of intervals in $[k+\ell]$)
\begin{align*}
\r_{f} & :\{0_{\i}:\i\in\II_{f,R}^{+}\}\to[n],\quad\RR_{f}:\{1_{\i}:\i\in\II_{f,R}^{-}\}\to[n],\\
\s_{f} & :\{1_{\i}:\i\in\II_{f,R}^{+}\}\to[n],\quad\SS_{f}:\{0_{\i}:\i\in\II_{f,R}^{-}\}\to[n],\\
\UU_{f} & :\{1_{\i}:\i\in\II_{f,R^{-1}}^{-}\}\to[n],\quad\u_{f}:\{0_{\i}:\i\in\II_{f,R^{-1}}^{+}\}\to[n],\\
\V_{f} & :\{0_{\i}:\i\in\II_{f,R^{-1}}^{-}\}\to[n],\quad\v_{f}:\{1_{\i}:\i\in\II_{f,R^{-1}}^{+}\}\to[n].
\end{align*}

We obtain from (\ref{eq:plain-form})
\begin{align*}
 & \langle h_{a}e_{\s_{a}}^{\V_{a}},e_{\r_{a}}^{\UU_{a}}\rangle\langle h_{b}e_{\s_{b}}^{\V_{b}},e_{\r_{b}}^{\UU_{b}}\rangle\langle h_{a}^{-1}e_{\RR_{a}}^{\u_{a}},e_{\SS_{a}}^{\v_{a}}\rangle\langle h_{b}^{-1}e_{\RR_{b}}^{\u_{b}},e_{\SS_{b}}^{\v_{b}}\rangle\\
 & \langle h_{c}e_{\s_{c}}^{\V_{c}},e_{\r_{c}}^{\UU_{c}}\rangle\langle h_{d}e_{\s_{d}}^{\V_{d}},e_{\r_{d}}^{\UU_{d}}\rangle\langle h_{c}^{-1}e_{\RR_{c}}^{\u_{c}},e_{\SS_{c}}^{\v_{c}}\rangle\langle h_{d}^{-1}e_{\RR_{d}}^{\u_{d}},e_{\SS_{d}}^{\v_{d}}\rangle\\
=\prod_{f} & \prod_{\i^{+}\in\II_{f,R}^{+}}\prod_{\i^{-}\in\II_{f,R}^{-}}\prod_{j^{+}\in\II_{f,R^{-1}}^{+}}\prod_{j^{-}\in\II_{f,R^{-1}}^{-}}\\
 & h_{\r_{f}(0_{\i^{+}})\s_{f}(1_{\i^{+}})}h_{\u_{f}(0_{\j^{+}})\v_{f}(1_{\j^{+}})}\bar{h}_{\RR_{f}(1_{\i^{-}})\SS_{f}(0_{\i^{-}})}\bar{h}_{\UU_{f}(1_{\j^{-}})\V_{f}(0_{\j^{-}})}.
\end{align*}

With this formalism we obtain
\begin{align}
 & \J_{n}(w,\mu,\nu)\nonumber \\
= & \sum_{p_{i}}\sum_{\r_{f},\RR_{f},\V_{f},\v_{f},\UU_{f},\u_{f},\s_{f},\SS_{f}}\sum_{\a\in\A(w)}\nonumber \\
 & \beta_{p_{2}\s_{a}}^{\V_{a}}\bar{\beta}_{p_{1}\r_{a}}^{\UU_{a}}\beta_{p_{3}\s_{b}}^{\V_{b}}\bar{\beta}_{p_{2}\r_{b}}^{\UU_{b}}\beta_{p_{4}\RR_{a}}^{\u_{a}}\bar{\beta}_{p_{3}\SS_{a}}^{\v_{a}}\beta_{p_{5}\RR_{b}}^{\u_{b}}\bar{\beta}_{p_{4}\SS_{b}}^{\v_{b}}\nonumber \\
\times & \beta_{p_{6}\s_{c}}^{\V_{c}}\bar{\beta}_{p_{5}\r_{c}}^{\UU_{c}}\beta_{p_{7}\s_{d}}^{\V_{d}}\bar{\beta}_{p_{6}\r_{d}}^{\UU_{d}}\beta_{p_{8}\RR_{c}}^{\u_{c}}\bar{\beta}_{p_{7}\SS_{c}}^{\v_{c}}\beta_{p_{1}\RR_{d}}^{\u_{d}}\bar{\beta}_{p_{8}\SS_{d}}^{\v_{d}}\nonumber \\
 & \prod_{f\in\{a,b,c,d\}}\nonumber \\
 & \int_{h\in\U(n)}\prod_{\i\in\II_{f,w}^{+}}\prod_{\j\in\II_{f,w}^{-}}\prod_{\i^{+}\in\II_{f,R}^{+}}\prod_{\i^{-}\in\II_{f,R}^{-}}\prod_{j^{+}\in\II_{f,R^{-1}}^{+}}\prod_{j^{-}\in\II_{f,R^{-1}}^{-}}\label{eq:Jn-expanded}\\
 & h_{\r_{f}(0_{\i^{+}})\s_{f}(1_{\i^{+}})}h_{\u_{f}(0_{\j^{+}})\v_{f}(1_{\j^{+}})}\bar{h}_{\RR_{f}(1_{\i^{-}})\SS_{f}(0_{\i^{-}})}\bar{h}_{\UU_{f}(1_{\j^{-}})\V_{f}(0_{\j^{-}})}dh.\nonumber 
\end{align}

For each $f$, the integral in (\ref{eq:Jn-expanded}) can be done
using the Weingarten calculus (Theorem \ref{thm:Wg-vanilla}). To
do this, fix bijections for each $f\in\{a,b,c,d\}$
\[
\II_{f}^{+}\eqdf\II_{f,R}^{+}\cup\II_{f,R^{-1}}^{+}\cup\II_{f,w}^{+}\cong[k+\ell+p_{f}]
\]
\[
\II_{f}^{-}\eqdf\II_{f,R}^{-}\cup\II_{f,R^{-1}}^{-}\cup\II_{f,w}^{-}\cong[k+\ell+p_{f}]
\]
such that 
\[
\II_{f,w}^{+}\cong[k+\ell+1,k+\ell+p_{f}],\,\II_{f,w}^{-}\cong[k+\ell+1,k+\ell+p_{f}]
\]
 and
\begin{equation}
\II_{f,R}^{+}\cong[k],\,\II_{f,R}^{-}\cong[k],\,\II_{f,R^{-1}}^{+}\cong[k+1,k+\ell],\,\II_{f,R^{-1}}^{-}\cong[k+1,k+\ell]\label{eq:bijs}
\end{equation}
correspond to the original numberings of $\II_{f,R}^{+}$, $\II_{f,R}^{-}$,
$\II_{f,R^{-1}}^{+}$, $\II_{f,R^{-1}}^{-}$.

Hence if $\sigma_{f},\tau_{f}\in S_{k+\ell+p_{f}}$ we view $\sigma_{f},\tau_{f}:\II_{f}^{+}\to\II_{f}^{-}$
by the above fixed bijections. For each $f\in\{a,b,c,d\}$ we say
$(\a,\r_{f},\u_{f},\R_{f},\UU_{f})\to\sigma_{f}$ if for all $\i\in\II_{f}^{+}$,
$\i'\in\II_{f}^{-}$ with $\sigma_{f}(\i)=\i'$, we have 
\[
[\r_{f}\sqcup\u_{f}\sqcup\a](0_{\i})=[\R_{f}\sqcup\UU_{f}\sqcup\a](1_{\i'});
\]
here we wrote e.g. $[\r_{f}\sqcup\u_{f}\sqcup\a]$ for the function
that $\a,\r_{f},\u_{f}$ induce on $\{0_{\i}\,:\,\i\in\II_{f}^{+}\}$.
Similarly we say $(\a,\s_{f},\v_{f},\SS_{f},\V_{f})\to\tau_{f}$ if
for all $\i\in\II_{f}^{+},\i'\in\II_{f}^{-}$ with $\tau_{f}(\i)=\i'$
we have 
\[
[\s_{f}\sqcup\v_{f}\sqcup\a](1_{\i})=[\SS_{f}\sqcup\V_{f}\sqcup\a](0_{\i'}).
\]
Theorem \ref{thm:Wg-vanilla} translates to
\begin{align*}
 & \int_{h\in\U(n)}\prod_{\i\in\II_{f,w}^{+}}\prod_{\j\in\II_{f,w}^{-}}\prod_{\i^{+}\in\II_{f,R}^{+}}\prod_{\i^{-}\in\II_{f,R}^{-}}\prod_{j^{+}\in\II_{f,R^{-1}}^{+}}\prod_{j^{-}\in\II_{f,R^{-1}}^{-}}\\
 & h_{\r_{f}(0_{\i^{+}})\s_{f}(1_{\i^{+}})}h_{\u_{f}(0_{\j^{+}})\v_{f}(1_{\j^{+}})}\bar{h}_{\RR_{f}(1_{\i^{-}})\SS_{f}(0_{\i^{-}})}\bar{h}_{\UU_{f}(1_{\j^{-}})\V_{f}(0_{\j^{-}})}dh\\
= & \sum_{\sigma_{f},\tau_{f}\in S_{k+\ell+p_{f}}}\Wg_{n,k+\ell+p_{f}}(\sigma_{f}\tau_{f}^{-1})\\
 & \times\mathbf{1}\{(\a,\r_{f},\u_{f},\R_{f},\UU_{f})\to\sigma_{f},(\a,\s_{f},\v_{f},\SS_{f},\V_{f})\to\tau_{f}\},
\end{align*}
so putting this into (\ref{eq:Jn-expanded}) gives
\begin{align*}
\J_{n}(w,\mu,\nu)= & \sum_{\sigma_{f},\tau_{f}\in S_{k+\ell+p_{f}}}\left(\prod_{f\in\{a,b,c,d\}}\Wg_{n,k+\ell+p_{f}}(\sigma_{f}\tau_{f}^{-1})\right)\\
 & \sum_{p_{i}}\sum_{\substack{\a\in\A(w),\r_{f},\RR_{f},\V_{f},\v_{f},\UU_{f},\u_{f},\s_{f},\SS_{f}\\
(\a,\r_{f},\u_{f},\R_{f},\UU_{f})\to\sigma_{f}\\
(\a,\s_{f},\v_{f},\SS_{f},\V_{f})\to\tau_{f}
}
}\\
 & \beta_{p_{2}\s_{a}}^{\V_{a}}\bar{\beta}_{p_{1}\r_{a}}^{\UU_{a}}\beta_{p_{3}\s_{b}}^{\V_{b}}\bar{\beta}_{p_{2}\r_{b}}^{\UU_{b}}\beta_{p_{4}\RR_{a}}^{\u_{a}}\bar{\beta}_{p_{3}\SS_{a}}^{\v_{a}}\beta_{p_{5}\RR_{b}}^{\u_{b}}\bar{\beta}_{p_{4}\SS_{b}}^{\v_{b}}\\
 & \times\beta_{p_{6}\s_{c}}^{\V_{c}}\bar{\beta}_{p_{5}\r_{c}}^{\UU_{c}}\beta_{p_{7}\s_{d}}^{\V_{d}}\bar{\beta}_{p_{6}\r_{d}}^{\UU_{d}}\beta_{p_{8}\RR_{c}}^{\u_{c}}\bar{\beta}_{p_{7}\SS_{c}}^{\v_{c}}\beta_{p_{1}\RR_{d}}^{\u_{d}}\bar{\beta}_{p_{8}\SS_{d}}^{\v_{d}}.
\end{align*}
Here we make our main improvement over the classical Weingarten calculus.
We introduce the following beneficial property that the $\sigma_{f},\tau_{f}$
possibly have.
\begin{description}
\item [{Forbidden~matchings~property}] For every $f\in\{a,b,c,d\}$ the
following hold: neither $\sigma_{f}$ nor $\tau_{f}$ map any element
of $\II_{f,R}^{+}$ to an element of $\II_{f,R^{-1}}^{-}$, or map
an element of $\II_{f,R^{-1}}^{+}$ to an element of $\II_{f,R}^{-}$. 
\end{description}
We have the following key lemma.
\begin{lem}
\label{lem:forbidden-matchings}If for some $f\in\{a,b,c,d\},$ $\sigma_{f}$
and $\tau_{f}$ do \uline{not} have the \textbf{forbidden matchings}
property, then for any choice of $p_{1},\ldots,p_{8}$
\begin{align}
 & \sum_{\substack{\a\in\A(w),\r_{f},\RR_{f},\V_{f},\v_{f},\UU_{f},\u_{f},\s_{f},\SS_{f}\\
(\a,\r_{f},\u_{f},\R_{f},\UU_{f})\to\sigma_{f}\\
(\a,\s_{f},\v_{f},\SS_{f},\V_{f})\to\tau_{f}
}
}\label{eq:variables}\\
 & \beta_{p_{2}\s_{a}}^{\V_{a}}\bar{\beta}_{p_{1}\r_{a}}^{\UU_{a}}\beta_{p_{3}\s_{b}}^{\V_{b}}\bar{\beta}_{p_{2}\r_{b}}^{\UU_{b}}\beta_{p_{4}\RR_{a}}^{\u_{a}}\bar{\beta}_{p_{3}\SS_{a}}^{\v_{a}}\beta_{p_{5}\RR_{b}}^{\u_{b}}\bar{\beta}_{p_{4}\SS_{b}}^{\v_{b}}\nonumber \\
 & \times\beta_{p_{6}\s_{c}}^{\V_{c}}\bar{\beta}_{p_{5}\r_{c}}^{\UU_{c}}\beta_{p_{7}\s_{d}}^{\V_{d}}\bar{\beta}_{p_{6}\r_{d}}^{\UU_{d}}\beta_{p_{8}\RR_{c}}^{\u_{c}}\bar{\beta}_{p_{7}\SS_{c}}^{\v_{c}}\beta_{p_{1}\RR_{d}}^{\u_{d}}\bar{\beta}_{p_{8}\SS_{d}}^{\v_{d}}\nonumber \\
= & \,0.\nonumber 
\end{align}
\end{lem}

\begin{proof}
Indeed suppose $\sigma_{a}$ matches an element $\i\in\II_{a,R}^{+}$
with $\j\in\II_{a,R^{-1}}^{-}$; $\sigma_{a}(\i)=\j$. With our given
fixed bijections (\ref{eq:bijs}), $\i$ corresponds to an element
of $[k]$ and $\j$ corresponds to an element of $[k+1,k+\ell].$
Without loss of generality in the argument suppose that $0_{\i}$
corresponds to $1$ and $0_{\j}$ corresponds to $k+1$. The condition
$\sigma_{a}(\i)=\j$ and $(\a,\r_{a},\u_{a},\R_{a},\UU_{a})\to\sigma_{f}$
means that as functions on $[k]$ and $[k+1,k+\ell${]}, $\r_{a}(1)=\UU_{a}(k+1)$.
There are no other constraints on these values.

Then for all variables in (\ref{eq:variables}) fixed apart from $\r_{a}$
and $\UU_{a}$, and all values of $\r_{a},\UU_{a}$ fixed other than
$\r_{a}(1)$ and $\UU_{a}(k+1$) the ensuing sum over $\r_{a},\UU_{a}$
is 
\[
\sum_{\substack{\r_{a}(1)=\UU_{a}(k+1)}
}\beta_{p_{2}\r_{a}}^{\UU_{a}}.
\]
But recalling the contraction operators from (\ref{eq:contraction-def}),
this sum is the coordinate of $e_{\r_{a}(2)}\otimes\cdots\cdots e_{\r_{a}(k)}\otimes\check{e}_{\UU_{a}(k+2)}\otimes\cdots\otimes\check{e}_{\UU_{a}(k+\ell)}$
in $c_{1,1}(v_{p_{2}})$. But $c_{1,1}(v_{p_{2}})=0$ because $v_{p_{2}}\in\tkld$. 
\end{proof}
\emph{We henceforth write $\sum_{\sigma_{f},\tau_{f}}^{*}$ to mean
the sum is restricted to $\sigma_{f},\tau_{f}$ satisfying the }\textbf{\emph{forbidden
matchings }}\emph{property. }Lemma \ref{lem:forbidden-matchings}
now implies
\begin{align}
\J_{n}(w,\mu,\nu)= & \sum_{\sigma_{f},\tau_{f}\in S_{k+\ell+p_{f}}}^{*}\left(\prod_{f\in\{a,b,c,d\}}\Wg_{n,k+\ell+p_{f}}(\sigma_{f}\tau_{f}^{-1})\right)\nonumber \\
 & \sum_{p_{i}}\sum_{\substack{\a\in\A(w),\r_{f},\RR_{f},\V_{f},\v_{f},\UU_{f},\u_{f},\s_{f},\SS_{f}\\
(\a,\r_{f},\u_{f},\R_{f},\UU_{f})\to\sigma_{f}\\
(\a,\s_{f},\v_{f},\SS_{f},\V_{f})\to\tau_{f}
}
}\nonumber \\
 & \beta_{p_{2}\s_{a}}^{\V_{a}}\bar{\beta}_{p_{1}\r_{a}}^{\UU_{a}}\beta_{p_{3}\s_{b}}^{\V_{b}}\bar{\beta}_{p_{2}\r_{b}}^{\UU_{b}}\beta_{p_{4}\RR_{a}}^{\u_{a}}\bar{\beta}_{p_{3}\SS_{a}}^{\v_{a}}\beta_{p_{5}\RR_{b}}^{\u_{b}}\bar{\beta}_{p_{4}\SS_{b}}^{\v_{b}}\nonumber \\
 & \times\beta_{p_{6}\s_{c}}^{\V_{c}}\bar{\beta}_{p_{5}\r_{c}}^{\UU_{c}}\beta_{p_{7}\s_{d}}^{\V_{d}}\bar{\beta}_{p_{6}\r_{d}}^{\UU_{d}}\beta_{p_{8}\RR_{c}}^{\u_{c}}\bar{\beta}_{p_{7}\SS_{c}}^{\v_{c}}\beta_{p_{1}\RR_{d}}^{\u_{d}}\bar{\beta}_{p_{8}\SS_{d}}^{\v_{d}}.\label{eq:intergal-coarse-form}
\end{align}
Moreover, we can significantly tidy up (\ref{eq:intergal-coarse-form}).
For everything in (\ref{eq:intergal-coarse-form}) fixed except for
e.g. $p_{2}$, the ensuing sum over $p_{2}$ is 
\[
\sum_{p_{2}}\beta_{p_{2}\s_{a}}^{\V_{a}}\bar{\beta}_{p_{2}\r_{b}}^{\UU_{b}}=\sum_{p_{2}}\langle e_{\r_{b}}^{\UU_{b}},v_{p_{2}}\rangle\langle v_{p_{2}},e_{\s_{a}}^{\V_{a}}\rangle=\langle\q_{\theta}e_{\r_{b}}^{\UU_{b}},e_{\s_{a}}^{\V_{a}}\rangle.
\]

Therefore executing the sums over $p_{i}$ in (\ref{eq:intergal-coarse-form})
we replace the sum over $p_{i}$ and the product over $\beta$-terms
by
\begin{align}
\langle\q_{\theta}e_{\r_{b}}^{\UU_{b}},e_{\s_{a}}^{\V_{a}}\rangle\langle\q_{\theta}e_{\SS_{a}}^{\v_{a}},e_{\s_{b}}^{\V_{b}}\rangle\langle\q_{\theta}e_{\SS_{b}}^{\v_{b}},e_{\R_{a}}^{\u_{a}}\rangle\langle\q_{\theta}e_{\r_{c}}^{\UU_{c}},e_{\RR_{b}}^{\u_{b}}\rangle & \times\label{eq:pi-1}\\
\langle\q_{\theta}e_{\r_{d}}^{\UU_{d}},e_{\s_{c}}^{\V_{c}}\rangle\langle\q_{\theta}e_{\SS_{c}}^{\v_{c}},e_{\s_{d}}^{\V_{d}}\rangle\langle\q_{\theta}e_{\SS_{d}}^{\v_{d}},e_{\RR_{c}}^{\u_{c}}\rangle\langle\q_{\theta}e_{\r_{a}}^{\UU_{a}},e_{\RR_{d}}^{\u_{d}}\rangle.\nonumber 
\end{align}
By Proposition \ref{prop:proj-form} we have e.g. 
\[
\langle\q_{\theta}e_{\r_{b}}^{\UU_{b}},e_{\s_{a}}^{\V_{a}}\rangle=D_{\mu,\nu}(n)\sum_{\pi\in S_{k+\ell}}z_{\theta}(\pi)\langle\varphi^{-1}[\rho_{n}^{k,\ell}(\pi)]e_{\r_{b}}^{\UU_{b}},e_{\s_{a}}^{\V_{a}}\rangle
\]
Now recall from (\ref{eq:matching-end}) that
\[
\varphi^{-1}[\rho_{n}^{k+\ell}(\pi)]=\sum_{I=(i_{1},\ldots,i_{k}),J=(j_{k+1},\ldots,j_{k+\ell})}e_{I'(I\sqcup J;\pi)}^{J}\otimes\check{e}_{I}^{J'(I\sqcup J;\pi)}.
\]
This means that $\langle\varphi^{-1}[\rho_{n}^{k+\ell}(\pi)]e_{\r_{b}}^{\UU_{b}},e_{\s_{a}}^{\V_{a}}\rangle$
is either equal to 0 or 1 and\linebreak{}
 $\langle\varphi^{-1}[\rho_{n}^{k+\ell}(\pi)]e_{\r_{b}}^{\UU_{b}},e_{\s_{a}}^{\V_{a}}\rangle=1$
if and only if, letting (\ref{eq:bijs}) induce identifications
\begin{align*}
\{1_{\i}:\i\in\II_{a,R}^{+}\}\cong[k],\{1_{\i}:\i\in\II_{b,R^{-1}}^{-}\}\cong[k+1,k+\ell], & \,\\
\{0_{\i}:\i\in\II_{b,R}^{+}\}\cong[k],\{0_{\i}:\i\in\II_{a,R^{-1}}^{-}\}\cong[k+1,k+\ell],
\end{align*}
via their given indexing of intervals, we have $[\s_{a}\sqcup\UU_{b}]\circ\pi=[\r_{b}\sqcup\V_{a}]$,
where e.g. $\s_{a}\sqcup\UU_{b}$ is the function either on endpoints
of intervals or on $[k+\ell]$ induced by the union of $\s_{a}$ and
$\UU_{b}$. Hence, repeating this argument,
\begin{align*}
\eqref{eq:pi-1} & =D_{\mu,\nu}(n)^{8}\sum_{\pi_{1},\ldots,\pi_{8}\in S_{k+\ell}}\left(\prod_{i=1}^{8}z_{\theta}(\pi_{i})\right)\\
 & \mathbf{1}\{[\s_{a}\sqcup\UU_{b}]\circ\pi_{1}=[\r_{b}\sqcup\V_{a}],[\s_{b}\sqcup\v_{a}]\circ\pi_{2}=[\SS_{a}\sqcup\V_{b}],\\
 & [\RR_{a}\sqcup\v_{b}]\circ\pi_{3}=[\SS_{b}\sqcup\u_{a}],[\R_{b}\sqcup\UU_{c}]\circ\pi_{4}=[\r_{c}\sqcup\u_{b}],\\
 & [\s_{c}\sqcup\UU_{d}]\circ\pi_{5}=[\r_{d}\sqcup\V_{c}],[\s_{d}\sqcup\v_{c}]\circ\pi_{6}=[\SS_{c}\sqcup\V_{d}],\\
 & [\RR_{c}\sqcup\v_{d}]\circ\pi_{7}=[\SS_{d}\sqcup\u_{c}],[\RR_{d}\sqcup\UU_{a}]\circ\pi_{8}=[\r_{a}\sqcup\u_{d}]\}.
\end{align*}
Putting all these arguments together gives
\begin{align*}
 & \J_{n}(w,\mu,\nu)\\
= & D_{\mu,\nu}(n)^{8}\sum_{\sigma_{f},\tau_{f}\in S_{p_{f}+k+\ell}}^{*}\sum_{\pi_{1},\ldots,\pi_{8}\in S_{k+\ell}}\left(\prod_{f\in\{a,b,c,d\}}\Wg_{n,k+\ell+p_{f}}(\sigma_{f}\tau_{f}^{-1})\right)\left(\prod_{i=1}^{8}z_{\theta}(\pi_{i})\right)\\
 & \sum_{p_{i}}\sum_{\substack{\a\in\A(w),\r_{f},\RR_{f},\V_{f},\v_{f},\UU_{f},\u_{f},\s_{f},\SS_{f}\\
(\a,\r_{f},\u_{f},\R_{f},\UU_{f})\to\sigma_{f}\\
(\a,\s_{f},\v_{f},\SS_{f},\V_{f})\to\tau_{f}
}
}\\
 & \mathbf{1}\{[\s_{a}\sqcup\UU_{b}]\circ\pi_{1}=[\r_{b}\sqcup\V_{a}],[\s_{b}\sqcup\v_{a}]\circ\pi_{2}=[\SS_{a}\sqcup\V_{b}],\\
 & [\RR_{a}\sqcup\v_{b}]\circ\pi_{3}=[\SS_{b}\sqcup\u_{a}],[\R_{b}\sqcup\UU_{c}]\circ\pi_{4}=[\r_{c}\sqcup\u_{b}],\\
 & [\s_{c}\sqcup\UU_{d}]\circ\pi_{5}=[\r_{d}\sqcup\V_{c}],[\s_{d}\sqcup\v_{c}]\circ\pi_{6}=[\SS_{c}\sqcup\V_{d}],\\
 & [\RR_{c}\sqcup\v_{d}]\circ\pi_{7}=[\SS_{d}\sqcup\u_{c}],[\RR_{d}\sqcup\UU_{a}]\circ\pi_{8}=[\r_{a}\sqcup\u_{d}]\}.
\end{align*}

This formula says that we can calculate $\J_{n}(w,\mu,\nu)$ by summing
over some combinatorial data of matchings (the $\sigma_{f},\tau_{f},\pi_{i}$)
a quantity that we can understand well times a count of the number
of indices that satisfy the prescribed matchings. To formalize this
point of view we make the following definition.
\begin{defn}
A \emph{matching datum} of the triple $(w,k,\ell)$ is a pair $(\sigma_{f},\tau_{f})\in S_{k+\ell+p_{f}}\times S_{k+\ell+p_{f}}$
as above, satisfying the \textbf{forbidden matchings} property for
each $f\in\{a,b,c,d\}$, together with $(\pi_{1},\ldots,\pi_{8})\in(S_{k+\ell})^{8}$.
We write 
\[
\match(w,k,\ell)
\]
 for the finite collection of all matching data for $(w,k,\ell)$.

Given a matching datum $\{\sigma_{f},\tau_{f},\pi_{i}\}$, we write
$\NN(\{\sigma_{f},\tau_{f},\pi_{i}\})$ for the number of choices
of $\a\in\A(w),\r_{f},\RR_{f},\V_{f},\v_{f},\UU_{f},\u_{f},\s_{f},\SS_{f}$
such that
\begin{align}
(\a,\r_{f},\u_{f},\R_{f},\UU_{f})\to\sigma_{f} & ,(\a,\s_{f},\v_{f},\SS_{f},\V_{f})\to\tau_{f},\nonumber \\{}
[\s_{a}\sqcup\UU_{b}]\circ\pi_{1}=[\r_{b}\sqcup\V_{a}], & [\s_{b}\sqcup\v_{a}]\circ\pi_{2}=[\SS_{a}\sqcup\V_{b}],\nonumber \\{}
[\RR_{a}\sqcup\v_{b}]\circ\pi_{3}=[\SS_{b}\sqcup\u_{a}], & [\R_{b}\sqcup\UU_{c}]\circ\pi_{4}=[\r_{c}\sqcup\u_{b}],\nonumber \\{}
[\s_{c}\sqcup\UU_{d}]\circ\pi_{5}=[\r_{d}\sqcup\V_{c}], & [\s_{d}\sqcup\v_{c}]\circ\pi_{6}=[\SS_{c}\sqcup\V_{d}],\nonumber \\{}
[\RR_{c}\sqcup\v_{d}]\circ\pi_{7}=[\SS_{d}\sqcup\u_{c}], & [\RR_{d}\sqcup\UU_{a}]\circ\pi_{8}=[\r_{a}\sqcup\u_{d}].\label{eq:index-conditions}
\end{align}
\end{defn}

With this notation, we have proved the following theorem.
\begin{thm}
\label{thm:first-integration-formula}For $k+\ell>0$, $\mu\vdash k$
and $\nu\vdash\ell$, $w\in[\F_{4},\F_{4}]$, we have
\begin{align}
\J_{n}(w,\mu,\nu)= & D_{\mu,\nu}(n)^{8}\sum_{\{\sigma_{f},\tau_{f},\pi_{i}\}\in\match(w,k,\ell)}\left(\prod_{i=1}^{8}z_{\theta}(\pi_{i})\right)\nonumber \\
 & \left(\prod_{f\in\{a,b,c,d\}}\Wg_{n,k+\ell+p_{f}}(\sigma_{f}\tau_{f}^{-1})\right)\NN(\{\sigma_{f},\tau_{f},\pi_{i}\}).\label{eq:comb-int-formula}
\end{align}
\end{thm}

We conclude this section by bounding the terms $z_{\theta}(\pi_{i})$
and\linebreak{}
$\Wg_{n,k+\ell+p_{f}}(\sigma_{f}\tau_{f}^{-1})$ using Proposition
\ref{prop:wg-bound} and Lemma \ref{lem:coeff-bound}, recalling also
(\ref{eq:dimension-asymp}). Note that $\sum_{f\in\{a,b,c,d\}}p_{f}=\frac{|w|}{2}$.
This yields
\begin{cor}
\label{cor:J_n-bound-comb}For $k+\ell>0$, $\mu\vdash k$ and $\nu\vdash\ell$,
$w\in[\F_{4},\F_{4}]$, we have
\begin{align}
\J_{n}(w,\mu,\nu)\ll_{k,\ell,w} & n^{-4k-4\ell-\frac{|w|}{2}}\sum_{\{\sigma_{f},\tau_{f},\pi_{i}\}\in\match(w,k,\ell)}\label{eq:J_n-bound-combinatorial}\\
 & n^{-\sum_{f}|\sigma_{f}\tau_{f}^{-1}|-\sum_{i=1}^{8}\|\pi_{i}\|_{k,\ell}}\NN(\{\sigma_{f},\tau_{f},\pi_{i}\}).\nonumber 
\end{align}
\end{cor}

We will proceed in the next section to understand all the quantities
in (\ref{eq:J_n-bound-combinatorial}) in topological terms by constructing
a surface from each $\{\sigma_{f},\tau_{f},\pi_{i}\}$.

\section{Topology}

\subsection{Construction of surfaces from matching data\label{subsec:Construction-of-surfaces}}

We now show how a datum in $\match(w,k,\ell)$ can be used to construct
a surface such that the terms appearing in (\ref{eq:comb-int-formula})
can be bounded by topological features of the surface. This construction
is similar to the constructions of \cite{MPunitary,MP}, but with
the presence of additional $\pi_{i}$ adding a new aspect. We continue
to assume $g=2$ for simplicity. We can still assume that $\gamma\in[\Gamma_{2},\Gamma_{2}]$
and hence $w\in[\F_{4},\F_{4}]$.

\subsubsection*{Construction of the 1-skeleton}

\uline{\mbox{$\pi$}}\emph{\uline{-intervals.}} The identifications
of the previous section mean that we view
\begin{align}
\pi_{1} & :\{\,0_{\i}\,:\,\i\in\II_{b,R}^{+}\cup\II_{a,R^{-1}}^{-}\}\to\{\,1_{\i'}\,:\,\i'\in\II_{a,R}^{+}\cup\II_{b,R^{-1}}^{-}\,\},\nonumber \\
\pi_{2} & :\{\,0_{\i}\,:\,\i\in\II_{a,R}^{-}\cup\II_{b,R^{-1}}^{-}\}\to\{\,1_{\i'}\,:\,\i'\in\II_{b,R}^{+}\cup\II_{a,R^{-1}}^{+}\,\},\nonumber \\
\pi_{3} & :\{\,0_{\i}\,:\,\i\in\II_{b,R}^{-}\cup\II_{a,R^{-1}}^{+}\}\to\{\,1_{\i'}\,:\,\i'\in\II_{a,R}^{-}\cup\II_{b,R^{-1}}^{+}\,\},\nonumber \\
\pi_{4} & :\{\,0_{\i}\,:\,\i\in\II_{c,R}^{+}\cup\II_{b,R^{-1}}^{+}\}\to\{\,1_{\i'}\,:\,\i'\in\II_{b,R}^{-}\cup\II_{c,R^{-1}}^{-}\,\},\nonumber \\
\pi_{5} & :\{\,0_{\i}\,:\,\i\in\II_{d,R}^{+}\cup\II_{c,R^{-1}}^{-}\}\to\{\,1_{\i'}\,:\,\i'\in\II_{c,R}^{+}\cup\II_{d,R^{-1}}^{-}\,\},\nonumber \\
\pi_{6} & :\{\,0_{\i}\,:\,\i\in\II_{c,R}^{-}\cup\II_{d,R^{-1}}^{-}\}\to\{\,1_{\i'}\,:\,\i'\in\II_{d,R}^{+}\cup\II_{c,R^{-1}}^{+}\,\},\nonumber \\
\pi_{7} & :\{\,0_{\i}\,:\,\i\in\II_{d,R}^{-}\cup\II_{c,R^{-1}}^{+}\}\to\{\,1_{\i'}\,:\,\i'\in\II_{c,R}^{-}\cup\II_{d,R^{-1}}^{+}\,\},\nonumber \\
\pi_{8} & :\{\,0_{\i}\,:\,\i\in\II_{a,R}^{+}\cup\II_{d,R^{-1}}^{+}\}\to\{\,1_{\i'}\,:\,\i'\in\II_{d,R}^{-}\cup\II_{a,R^{-1}}^{-}\,\}.\label{eq:pi-maps}
\end{align}
We add an arc between any two interval endpoints that are mapped to
one another by some $\pi_{i}$. All the intervals added here are called
\emph{$\pi$-intervals}. The purpose of this construction is that
the conditions concerning $\pi_{i}$ in (\ref{eq:index-conditions})
correspond to the fact that \emph{two end-points of intervals connected
by a $\pi$-interval are assigned the same value in $[n]$ by the
relevant functions out of $\r_{f},\RR_{f},\V_{f},\v_{f},\UU_{f},\u_{f},\s_{f},\SS_{f}$
(at most one of these functions has any given interval endpoint in
its domain).}

The $\pi$-intervals together with the $R$-intervals and $R^{-1}$
intervals form a collection of loops that we call $R^{\pm}$-$\pi$-loops. 

\emph{\uline{\mbox{$\sigma$}-arcs and \mbox{$\tau$}-arcs.}}\emph{
}Recall from the previous sections that we view 
\[
\sigma_{f},\tau_{f}:\II_{f}^{+}\to\II_{f}^{-}.
\]
 We add an arc between each $0_{\i}$ and $1_{\i'}$ with $\sigma_{f}(\i)=\i'$
and between each $1_{\i}$ and $0_{\i'}$ with $\tau_{f}(\i)=\i'$.
These arcs are called $\sigma_{f}$-arcs and $\tau_{f}$-arcs respectively.
Any $\sigma_{f}$-arc (resp. $\tau_{f}$-arc) is also called a $\sigma$-arc
(resp. $\tau$-arc). Notice even though an arc is formally the same
as an interval, we distinguish these types of objects. The only arcs
that exist are $\sigma$-arcs and $\tau$-arcs. The purpose of this
construction is that the conditions pertaining to $\sigma_{f},\tau_{f}$
in (\ref{eq:index-conditions}) are equivalent to the fact that \emph{two
end-points of intervals connected by a $\sigma$-arc or $\tau$-arc
are assigned the same value in $[n]$ by the relevant functions out
of $\a,\r_{f},\RR_{f},\V_{f},\v_{f},\UU_{f},\u_{f},\s_{f},\SS_{f}$.}

After adding these arcs, every endpoint of an interval has exactly
one arc emanating from it. We have therefore now constructed a trivalent
graph 
\[
G(\{\sigma_{f},\tau_{f},\pi_{i}\}).
\]
Each vertex of the graph is an endpoint of two intervals and one arc.
The number of vertices of this graph is twice the total number of
$w$-intervals, $R$-intervals, and $R^{-1}$-intervals which is $2(|w|+8(k+\ell))$.
Therefore we have 
\begin{equation}
\chi(G(\{\sigma_{f},\tau_{f},\pi_{i}\}))=-(|w|+8(k+\ell)).\label{eq:chi-of-graph}
\end{equation}
(For general $g$, we have $\chi(G(\{\sigma_{f},\tau_{f},\pi_{i}\}))=-(|w|+4g(k+\ell))$.)
Moreover, \emph{the conditions in (\ref{eq:index-conditions}) are
now interpreted purely in terms of the combinatorics of this graph.}

\subsubsection*{Gluing in discs}

There are two types of cycles in $G(\{\sigma_{f},\tau_{f},\pi_{i}\})$
that we wish to consider:
\begin{itemize}
\item Cycles that alternate between following either a $w$-intermediate
interval or a $\pi$-interval and then either a $\sigma$-arc or a
$\tau$-arc. These cycles are disjoint from one another, and every
$\sigma$ or $\tau$-arc is contained in exactly one such cycle. We
call these cycles \emph{type-I cycles. }For every type-I cycle, we
glue a disc to $G(\{\sigma_{f},\tau_{f},\pi_{i}\})$ along its boundary,
following the cycle. These discs will be called \emph{type-I discs.
}(These are analogous to the $o$-discs of \cite{MPunitary}.) 
\item Cycles that alternate between following either a $w$-interval, an
$R$-interval, or an $R^{-1}$-interval and then either a $\sigma$-arc
or a $\tau$-arc. Again, these cycles are disjoint, and every $\sigma$
or $\tau$-arc is contained in exactly one such cycle. We call these
cycles \emph{type-II cycles. }For every type-II cycle, we glue a disc
to $G(\{\sigma_{f},\tau_{f},\pi_{i}\})$ identifying the boundary
of the disc with the cycle. These discs will be called \emph{type-II
discs.} (These are similar to the $z$-discs of \cite{MPunitary}.)
\end{itemize}
Because every interior of an interval meets exactly one of the glued-in
discs, and every arc has two boundary segments of discs glued to it,
the object resulting from gluing in these discs is a decorated topological
surface that we denote by 
\[
\Sigma(\{\sigma_{f},\tau_{f},\pi_{i}\}).
\]

An example of this construction is depicted in Figure \ref{fig:example-CW}.

\begin{figure}
\begin{centering}
\includegraphics[scale=0.35]{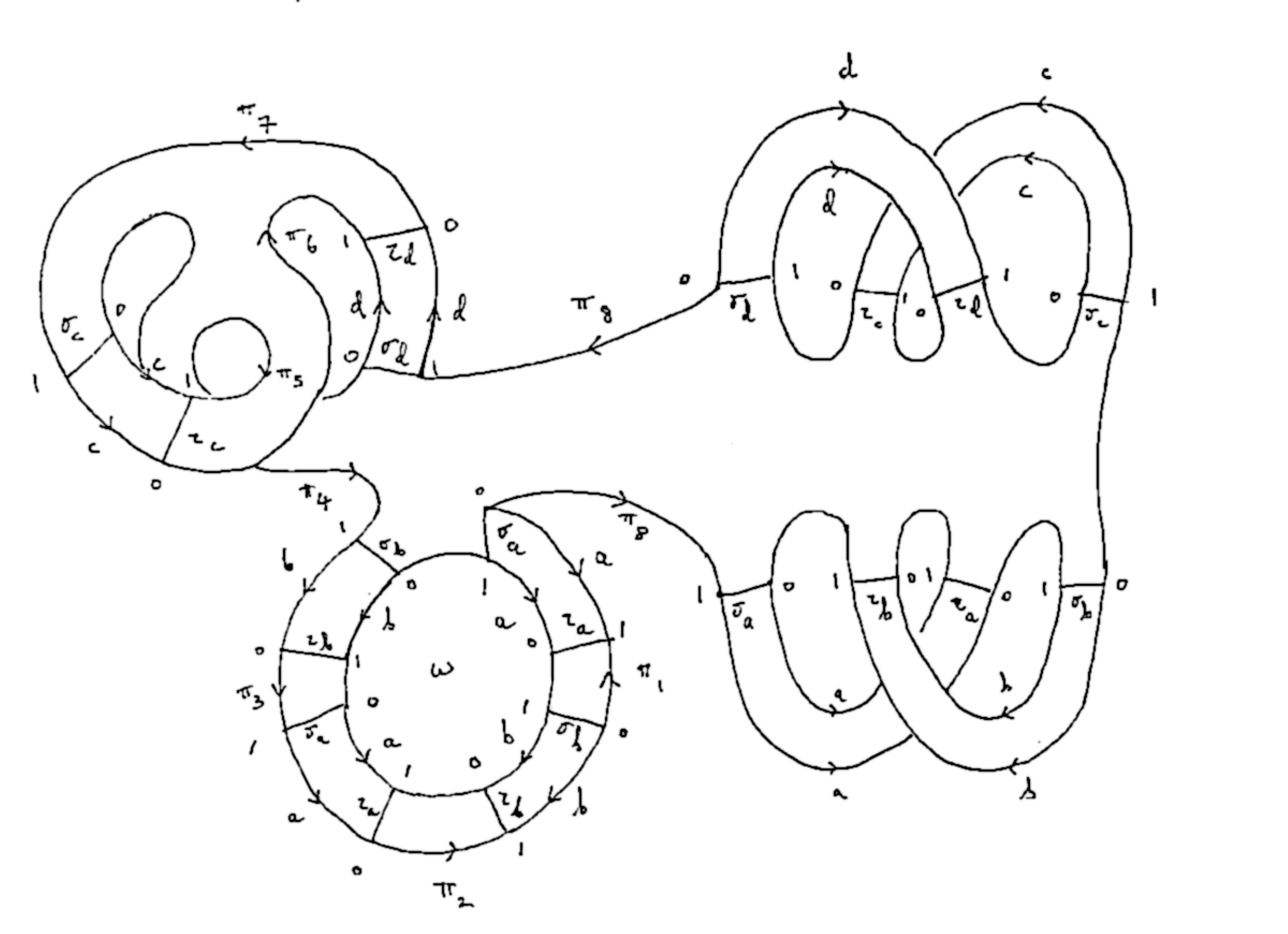}
\par\end{centering}
\caption{\label{fig:example-CW}This is a depiction of an example $\Sigma(\{\sigma_{f},\tau_{f},\pi_{i}\})$
for $w=ab^{-1}a^{-1}b$. The $\sigma$, $\tau$, and some of the $\pi_{i}$-arcs
are labeled along with the numbers ($0$ or $1$) of the points being
matched in the $w$-intervals. Each $w$-interval is also labeled
with its corresponding letter. Here $k=\ell=1$; $\pi_{8}$ is a transposition
and all other $\pi_{i}$ are the identity. There is one resulting
$R^{\pm}$-$\pi$-loop. In this example, for each $f\in\{a,b,c,d\}$,
$\sigma_{f}=\tau_{f}$. This means that all type II discs are rectangles.}

\end{figure}
The boundary components of $\Sigma(\{\sigma_{f},\tau_{f},\pi_{i}\})$
consist of the $w$-loop and the $R^{\pm}$-$\pi$-loops. It is not
hard to check that $\Sigma(\{\sigma_{f},\tau_{f},\pi_{i}\})$ is orientable
with an orientation compatible with the fixed orientations of the
boundary loops corresponding to traversing every $w$-interval or
$R^{\pm1}$-interval from $0$ to $1$.

We view the given CW-complex structure, and the assigned labelings
and directions of the intervals that now pave $\partial\Sigma$ as
part of the data of $\Sigma(\{\sigma_{f},\tau_{f},\pi_{i}\})$. The
number of discs of $\Sigma(\{\sigma_{f},\tau_{f},\pi_{i}\})$ is connected
to the quantities appearing in Proposition \ref{thm:first-integration-formula}
as follows.
\begin{lem}
\label{lem:I-discs}$\NN(\{\sigma_{f},\tau_{f},\pi_{i}\})=n^{\#\{\text{type-I discs of \ensuremath{\Sigma(\{\sigma_{f},\tau_{f},\pi_{i}\})}\}}}$.
\end{lem}

\begin{proof}
The constraints on the functions $\a,\r_{f},\RR_{f},\V_{f},\v_{f},\UU_{f},\u_{f},\s_{f},\SS_{f}$
in (\ref{eq:index-conditions}) now correspond to the fact that altogether,
they assign the same value in $[n]$ to every interval end-point in
the same type-I-cycle, and there are no other constraints between
them. 
\end{proof}
The quantities $|\sigma_{f}\tau_{f}^{-1}|$ in (\ref{eq:J_n-bound-combinatorial})
can also be related to $\Sigma(\{\sigma_{f},\tau_{f},\pi_{i}\})$
as follows.
\begin{lem}
\label{lem:II-discs}We have
\[
\prod_{f\in\{a,b,c,d\}}n^{-|\sigma_{f}\tau_{f}^{-1}|}=n^{-4(k+\ell)-\frac{|w|}{2}}n^{\#\{\text{type-II discs of \ensuremath{\Sigma(\{\sigma_{f},\tau_{f},\pi_{i}\})}\}}}.
\]
\end{lem}

\begin{proof}
Recalling the definition of $|\sigma_{f}\tau_{f}^{-1}|$ from Proposition
\ref{prop:wg-bound}, we can also write
\[
|\sigma_{f}\tau_{f}^{-1}|=k+\ell+p_{f}-\#\{\text{cycles of \ensuremath{\sigma_{f}\tau_{f}^{-1}\}}.}
\]
The cycles of $\{\sigma_{f}\tau_{f}^{-1}:f\in\{a,b,c,d\}\}$ are in
1:1 correspondence with the type-II cycles of $\Sigma(\{\sigma_{f},\tau_{f},\pi_{i}\})$
and hence also the type-II discs. Therefore 
\begin{align*}
\prod_{f\in\{a,b,c,d\}}n^{-|\sigma_{f}\tau_{f}^{-1}|} & =n^{-4(k+\ell)}n^{\sum_{f\in\{a,b,c,d\}}(-p_{f}+\#\{\text{cycles of \ensuremath{\sigma_{f}\tau_{f}^{-1}\})}}}\\
 & =n^{-4(k+\ell)-\frac{|w|}{2}}n^{\#\{\text{type-II discs of \ensuremath{\Sigma(\{\sigma_{f},\tau_{f},\pi_{i}\})}\}}}.
\end{align*}
\end{proof}
We are now able to prove the following.
\begin{thm}
\label{thm:genus-exp-bound}For $k+\ell>0$, $\mu\vdash k$ and $\nu\vdash\ell$,
$w\in[\F_{4},\F_{4}]$, we have
\begin{align*}
\J_{n}(w,\mu,\nu)\ll_{w,k,\ell} & \sum_{\{\sigma_{f},\tau_{f},\pi_{i}\}\in\match(w,k,\ell)}n^{-\sum_{i=1}^{8}\|\pi_{i}\|_{k,\ell}}n^{\chi(\Sigma(\{\sigma_{f},\tau_{f},\pi_{i}\}))}.
\end{align*}
\end{thm}

\begin{proof}
Combining Lemmas \ref{lem:I-discs} and \ref{lem:II-discs} with Corollary
\ref{cor:J_n-bound-comb} gives
\begin{align*}
\J_{n}(w,\mu,\nu)\ll_{w,k,\ell} & n^{-8k-8\ell-|w|}\\
\times & \sum_{\{\sigma_{f},\tau_{f},\pi_{i}\}\in\match(w,k,\ell)}n^{-\sum_{i=1}^{8}\|\pi_{i}\|_{k,\ell}}n^{\#\{\text{discs of \ensuremath{\Sigma}(\{\ensuremath{\sigma_{f}},\ensuremath{\tau_{f}},\ensuremath{\pi_{i}}\})}\}}.
\end{align*}
Then from (\ref{eq:chi-of-graph}) we obtain
\begin{align*}
\J_{n}(w,\mu,\nu)\ll_{w,k,\ell} & \sum_{\{\sigma_{f},\tau_{f},\pi_{i}\}\in\match(w,k,\ell)}\\
 & n^{-\sum_{i=1}^{8}\|\pi_{i}\|_{k,\ell}}n^{\chi(G(\{\sigma_{f},\tau_{f},\pi_{i}\}))+\#\{\text{discs of \ensuremath{\Sigma}(\{\ensuremath{\sigma_{f}},\ensuremath{\tau_{f}},\ensuremath{\pi_{i}}\})}\}}\\
= & \sum_{\{\sigma_{f},\tau_{f},\pi_{i}\}\in\match(w,k,\ell)}n^{-\sum_{i=1}^{8}\|\pi_{i}\|_{k,\ell}}n^{\chi(\Sigma(\{\sigma_{f},\tau_{f},\pi_{i}\}))}.
\end{align*}
\end{proof}

\subsection{Two simplifying surgeries\label{subsec:Two-simplifying-surgeries}}

Theorem \ref{thm:genus-exp-bound} suggests that we now bound 
\[
\chi(\Sigma(\{\sigma_{f},\tau_{f},\pi_{i}\}))-\sum_{i=1}^{8}\|\pi_{i}\|_{k,\ell}
\]
 for all $\{\sigma_{f},\tau_{f},\pi_{i}\}\in\match(w,k,\ell)$. To
do this, we make some observations that simplify the task. If $C$
is a simple closed curve in a surface $S$, then \emph{compressing
$S$ along $C$} means that we cut $S$ along $C$ and then glue discs
to cap off any new boundary components created by the cut.

Suppose that we are given $\{\sigma_{f},\tau_{f},\pi_{i}\}\in\match(w,k,\ell)$.
Then $\{\sigma_{f},\sigma_{f},\pi_{i}\}$ is also in $\match(w,k,\ell)$
(the \textbf{forbidden matching}\textbf{\emph{ }}property continues
to hold). It is not hard to see that
\[
\chi(\Sigma(\{\sigma_{f},\sigma_{f},\pi_{i}\}))\geq\chi(\Sigma(\{\sigma_{f},\tau_{f},\pi_{i}\})).
\]
 Indeed, the $\tau_{f}$ arcs can be replaced by $\sigma_{f}$-parallel
arcs inside the type-II discs of $\Sigma(\{\sigma_{f},\tau_{f},\pi_{i}\})$.
The resulting surface's arcs may not cut the surface into discs, but
this can be fixed by (possibly repeatedly) compressing the surface
along simple closed curves disjoint from the arcs, leaving the combinatorial
data of the arcs unchanged but only potentially increasing the Euler
characteristic. 

It remains to deal with the sum $\sum_{i=1}^{8}\|\pi_{i}\|_{k,\ell}$.

Suppose again that an arbitrary $\{\sigma_{f},\tau_{f},\pi_{i}\}\in\match(w,k,\ell)$
is given. For each $i\in[8]$ write
\[
\pi_{i}=\pi_{i}^{*}\sigma_{i}
\]
where $\pi_{i}^{*}\in S_{k}\times S_{\ell}$, $\sigma_{i}=(\pi_{i}^{*})^{-1}\pi_{i}\in S_{k+\ell}$,
and $|\sigma_{i}|=\|\pi_{i}\|_{k,\ell}$. Let $X_{0}\eqdf\Sigma(\{\sigma_{f},\tau_{f},\pi_{i}\})$. 

Take $\Sigma(\{\sigma_{f},\tau_{f},\pi_{i}\})$ and add to it all
the $\pi_{i}^{*}$-intervals that would have been added if $\pi_{i}$
was replaced by $\pi_{i}^{*}$ for each $i\in[8]$ in its construction.
The resulting object $X_{1}$ is the decorated surface $X_{0}$ together
with a collection of $\pi_{i}^{*}$-intervals with endpoints in the
boundary of $X_{0}$, and interiors disjoint from $X_{0}$. This adds
$8(k+\ell)$ edges to $X_{0}$ and hence 
\[
\chi(X_{1})=\chi(\Sigma(\{\sigma_{f},\tau_{f},\pi_{i}\}))-8(k+\ell).
\]
 Now we consider all cycles that for any fixed $i\in[8]$, alternate
between $\pi_{i}$-intervals and $\pi_{i}^{*}$-intervals. The number
of these cycles is the total number of cycles of the permutations
$\{\,(\pi_{i}^{*})^{-1}\pi_{i}\,:\,i\in[8]\,\}$. On the other hand,
the number of cycles of $(\pi_{i}^{*})^{-1}\pi_{i}$ is 
\[
k+\ell-|(\pi_{i}^{*})^{-1}\pi_{i}|=k+\ell-|\sigma_{i}|=k+\ell-\|\pi_{i}\|_{k,\ell}
\]
So in total there are $8(k+\ell)-\sum_{i}\|\pi_{i}\|_{k,\ell}$ of
these cycles. For every such cycle, we glue a disc along its boundary
to the cycle. The resulting object is denoted $X_{2}$. Now, $X_{2}$
is a topological surface, and we added $8(k+\ell)-\sum_{i}\|\pi_{i}\|_{k,\ell}$
discs to $X_{1}$ to form $X_{2}$, so
\[
\chi(X_{2})=\chi(X_{1})+8(k+\ell)-\sum_{i}\|\pi_{i}\|_{k,\ell}=\chi(\Sigma(\{\sigma_{f},\tau_{f},\pi_{i}\}))-\sum_{i}\|\pi_{i}\|_{k,\ell}.
\]
Now `forget' all the original $\pi_{i}$-intervals from $X_{2}$ to
form $X_{3}$. The surface $X_{3}$ is a decorated surface in the
same sense as $X_{0}$, except the connected components of $X_{3}-\{\text{arcs}\}$
may not be discs. Similarly to before, by sequentially compressing
$X_{3}$ along non-nullhomotopic simple closed curves disjoint from
arcs, if they exist, we obtain a new decorated surface $X_{4}$. See
Figure \ref{fig:surgery} for an illustration of this surgery taking
place. Moreover, and this is the main point, $X_{4}$ is the same
as $\Sigma(\{\sigma_{f},\tau_{f},\pi_{i}^{*}\})$ in the sense that
they are related by a decoration-respecting cellular homeomorphism.
Compression can only increase the Euler characteristic, so we obtain
\[
\chi(\Sigma(\{\sigma_{f},\tau_{f},\pi_{i}^{*}\}))\geq\chi(X_{3})=\chi(X_{2})=\chi(\Sigma(\{\sigma_{f},\tau_{f},\pi_{i}\})-\sum_{i}\|\pi_{i}\|_{k,\ell}.
\]

\begin{figure}

\begin{centering}
\includegraphics[scale=0.3]{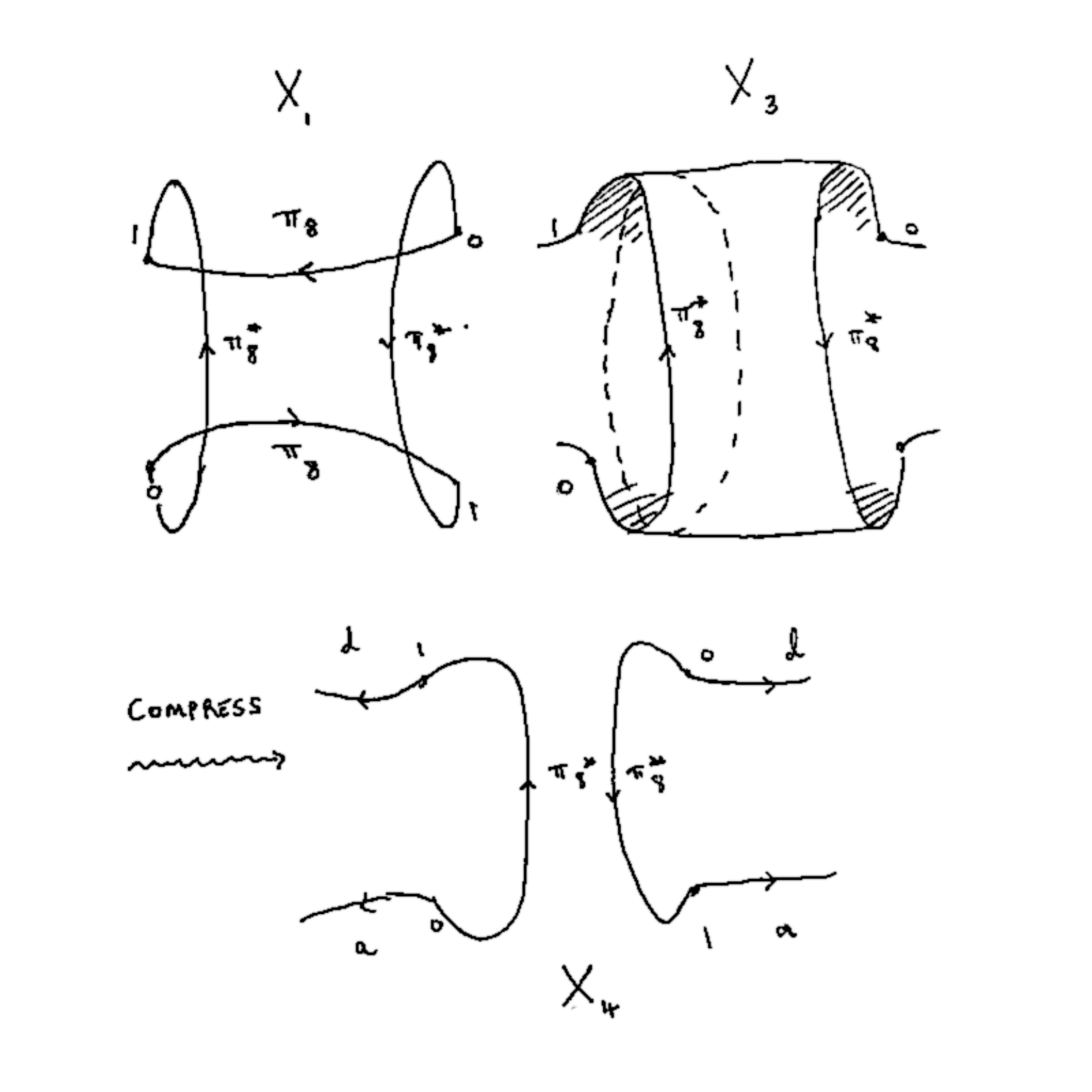}\caption{\label{fig:surgery}This is a local illustration of the second type
of simplifying surgery, precisely in the context of Figure \ref{fig:example-CW}.
The dashed simple closed curve in $X_{3}$ is disjoint from any arcs,
and cutting along this curve and gluing in two discs yields $X_{4}$.
Going back to Figure \ref{fig:example-CW} again, the net effect of
this surgery is to cut the left half from the right half. }
\par\end{centering}
\end{figure}

Combining these two arguments proves the following proposition.
\begin{prop}
For any given $\{\sigma_{f},\tau_{f},\pi_{i}\}$, there exist $\pi_{i}^{*}\in S_{k}\times S_{\ell}$
for $i\in[8]$ such that 
\begin{align*}
\chi(\Sigma(\{\sigma_{f},\sigma_{f},\pi_{i}^{*}\}))-\sum_{i=1}^{8}\|\pi_{i}^{*}\|_{k,\ell} & =\chi(\Sigma(\{\sigma_{f},\sigma_{f},\pi_{i}^{*}\}))\\
 & \geq\chi(\Sigma(\{\sigma_{f},\tau_{f},\pi_{i}\}))-\sum_{i=1}^{8}\|\pi_{i}\|_{k,\ell}.
\end{align*}
\end{prop}

This has the following immediate corollary when combined with Theorem
\ref{thm:genus-exp-bound}. Let 
\begin{align*}
\match^{*}(w,k,\ell)
\end{align*}
 denote the subset of $\match(w,k,\ell)$ consisting of $\{\sigma_{f},\sigma_{f},\pi_{i}\}$
(i.e. $\sigma_{f}=\tau_{f}$ for each $f\in\{a,b,c,d\}$) with $\pi_{i}\in S_{k}\times S_{\ell}$
for each $i\in[8]$.
\begin{cor}
\label{cor:goodmatchings}For $k+\ell>0$, $\mu\vdash k$ and $\nu\vdash\ell$,
$w\in[\F_{4},\F_{4}]$, we have
\begin{align*}
\J_{n}(w,\mu,\nu)\ll_{w,k,\ell}n^{\max_{\{\sigma_{f},\sigma_{f},\pi_{i}\}\in\match^{*}(w,k,\ell)}\chi(\Sigma(\{\sigma_{f},\sigma_{f},\pi_{i}\}))} & .
\end{align*}
\end{cor}

The benefit to having $\pi_{i}\in S_{k}\times S_{\ell}$ for $i\in[8]$
is the following. Suppose now that $\{\sigma_{f},\sigma_{f},\pi_{i}\}\in\match^{*}(w,k,\ell)$.
Recall that the boundary loops of $\Sigma(\{\sigma_{f},\sigma_{f},\pi_{i}\})$
consist of one $w$-loop and some number of $R^{\pm}$-$\pi$-loops.
The condition that each $\pi_{i}\in S_{k}\times S_{\ell}$ means that
no $\pi$-interval ever connects an endpoint of a $R$-interval with
an endpoint of an $R^{-1}$-interval. So every boundary component
of $\Sigma(\{\sigma_{f},\sigma_{f},\pi_{i}\})$ that is not the $w$-loop
contains either only $R$-intervals or only $R^{-1}$-intervals, and
in fact, when following the boundary component and reading the directions
and labels of the intervals according to traversing each from $0$
to $1$, reads out a positive power of $R$ (in the former case of
only $R$-intervals) or a negative power of $R^{-1}$ (in the latter
case of only $R^{-1}$-intervals). The sum of the positive powers
of $R$ in boundary loops is $k$, and the sum of the negative powers
of $R$ is $-\ell$. Knowing this boundary structure is extremely
important for the arguments in the next sections.

\subsection{A topological result that proves Theorem \ref{thm:main-single-small-dim}\label{subsec:A-topological-result}}

Here, in the spirit of Culler \cite{CULLER}, we explain another way
to think about the surfaces \linebreak{}
$\Sigma(\{\sigma_{f},\sigma_{f},\pi_{i}\})$ for $\{\sigma_{f},\sigma_{f},\pi_{i}\}\in\match^{*}(w,k,\ell)$
that is easier to work with than the construction we gave. At this
point we also show how things work for general $g\geq2$. An \emph{arc
}in a surface $\Sigma$ is a properly embedded interval in $\Sigma$
with endpoints in the boundary $\partial\Sigma$.
\begin{defn}
\label{def:surfaces}For $w\in\F_{2g}$, we define $\surfaces(w,k,\ell)$
to be the set of all decorated surfaces $\Sigma^{*}$ as follows.
A decorated surface $\Sigma^{*}\in\surfaces(w,k,\ell)$ is an oriented
surface with boundary, with compatibly oriented boundary components,
together with a collection of disjoint embedded arcs that cut $\Sigma^{*}$
into topological discs. One boundary component is assigned to be a
$w$-loop, and every other boundary component is assigned to be either
a $R$-loop or an $R^{-1}$-loop. Each arc is assigned a transverse
direction and a label in $\{a_{1},b_{1},\ldots,a_{g},b_{g}\}$. Every
arc-endpoint in $\partial\Sigma^{*}$ inherits a transverse direction
and label from the assigned direction and label of its arc. We require
that $\Sigma^{*}$ satisfy the following properties. 
\begin{description}
\item [{P1}] When one follows the $w$-loop according to its assigned orientation,
and reads $f$ when an $f$-labeled arc-endpoint is traversed in its
given direction, and $f^{-1}$ when an $f$-labeled arc-endpoint is
traversed counter to its given direction, one reads a cyclic rotation
of $w$ in reduced form, depending on where one begins to read.
\item [{P2}] When one follows any $R$-loop according to its assigned orientation
in the same way as before, one reads (a cyclic rotation) of some positive
power of $R_{g}$ in reduced form. The sum of these positive powers
over all $R$-loops is $k$.
\item [{P3}] When one follows any $R^{-1}$-loop according to its assigned
orientation in the same way as before, one reads (a cyclic rotation)
of some negative power of $R_{g}$ in reduced form. The sum of these
negative powers over all $R^{-1}$-loops is $-\ell$.
\item [{P4}] No arc connects an $R$-loop to an $R^{-1}$-loop.
\end{description}
\end{defn}

Given a surface $\Sigma(\{\sigma_{f},\sigma_{f},\pi_{i}\})$ with
$\{\sigma_{f},\sigma_{f},\pi_{i}\}\in\match^{*}(w,k,\ell)$, all the
type-II discs of the surface are rectangles. Hence, by collapsing
each $w$-interval, $R$-interval, and $R^{-1}$-interval to a point,
and collapsing every type-II rectangle to an arc, we obtain a CW-complex
that is a surface with boundary, cut into discs by arcs. Every arc
inherits a transverse direction and label from the compatible assigned
directions and labels of the intervals in the boundary of its originating
type-II rectangle. We call this modified surface $\Sigma^{*}=\Sigma^{*}(\{\sigma_{f},\pi_{i}\})$.
It clearly satisfies \textbf{P1-P3 }and \textbf{P4} follows from the
\textbf{forbidden matchings }property. (Of course, when $g=2$, we
identify $\{a,b,c,d\}$ with $\{a_{1},b_{1},a_{2},b_{2}\}$.) We also
have $\chi(\Sigma(\{\sigma_{f},\sigma_{f},\pi_{i}\}))=\chi(\Sigma^{*}(\{\sigma_{f},\pi_{i}\}))$.
With Definition \ref{def:surfaces} and the remarks proceeding it,
we can now state a further consequence of Corollary \ref{cor:goodmatchings}
as it extends to general $g\geq2$. 
\begin{cor}
\label{cor:max-chi-bound}For $k+\ell>0$, $\mu\vdash k$, $\nu\vdash\ell$,
$w\in[\F_{2g},\F_{2g}]$, as $n\to\infty$
\begin{align*}
\J_{n}(w,\mu,\nu)\ll_{w,k,\ell} & \,n^{\max\{\,\chi(\Sigma^{*})\,:\,\Sigma^{*}\in\surfaces(w,k,\ell)\,\}}.
\end{align*}
\end{cor}

In order for Corollary \ref{cor:max-chi-bound} to give us strong
enough results it needs to be combined with the following non-trivial
topological bound.
\begin{prop}
\label{prop:chi-bound}If $w\in[\F_{2g},\F_{2g}]$ is a shortest element
representing the conjugacy class of $\gamma\in\Gamma_{g}$, $w\neq\id$,
and $\Sigma^{*}\in\surfaces(w,k,\ell)$ then $\chi(\Sigma^{*})\leq-(k+\ell)$.
\end{prop}

\begin{rem}
Proposition \ref{prop:chi-bound} is by no means a trivial statement
and one has to use that $w$ is a shortest element representing the
conjugacy class of some element of $\Gamma_{g}$. For example, if
$w=R_{g}$, then $w$ represents the conjugacy class of $\id_{\Gamma_{g}}$,
but for $k=0$ and $\ell=1$ there is an `obvious' annulus in $\surfaces(w,0,1)$.
This has $\chi=0>-(k+\ell)=-1$. Proposition \ref{prop:chi-bound}
also requires $w\neq\id$; if $w=\id$ then for $k=0$ and $\ell=1$
one can take a disc with no arcs as a valid element of $\surfaces(\id,0,0)$.
This has $\chi=1>-(k+\ell)=0$. In fact this disc is ultimately responsible
for $\E_{g,n}[\tr_{\id}]=n$.
\end{rem}

The proof of Proposition \ref{prop:chi-bound} is self-contained and
given in $\S\S$\ref{subsec:Proof-of-chi-bound}. Before doing this,
we prove Theorem \ref{thm:main-single-small-dim}.
\begin{proof}[Proof of Theorem \ref{thm:main-single-small-dim} given Proposition
\ref{prop:chi-bound}]
Since Theorem \ref{thm:main-single-small-dim} was proved when $k=\ell=0$
in $\S\S$\ref{subsec:Proof-of-k=00003Dl=00003D0}, we can assume
$k+\ell>0$. Then combining Corollary \ref{cor:max-chi-bound} and
Proposition \ref{prop:chi-bound} gives
\[
\J_{n}(w,\mu,\nu)\ll_{w,k,\ell}n^{-(k+\ell)}.
\]
On the other hand, $D_{\mu,\nu}(n)=O\left(n^{k+\ell}\right)$ from
(\ref{eq:dimension-asymp}). Therefore\linebreak{}
$D_{\mu,\nu}(n)\J_{n}(w,\mu,\nu)\ll_{w,k,\ell}1$.
\end{proof}

\subsection{Work of Dehn and Birman-Series\label{subsec:Work-of-Dehn}}

As we mentioned in $\S\S$\ref{subsec:Setup-and-motivation}, to prove
Proposition \ref{prop:chi-bound} we have to use the fact that $w\in[\F_{2g},\F_{2g}]$
is a shortest element representing the conjugacy class of $\gamma\in\Gamma_{g}$.
We use a combinatorial characterization of such words that stems from
Dehn's algorithm \cite{Dehn} for solving the problem of whether a
given word represents the identity in $\Gamma_{g}$. The ideas of
Dehn's algorithm were refined by Birman and Series in \cite{BirmanSeries}.
In \cite{MPasympcover}, the author and Puder used Birman and Series'
results (alongside other methods) to obtain the analog of Theorem
\ref{thm:main-convergence} when the family of groups $\SU(n)$ is
replaced by the family of symmetric groups $S_{n}$. Similar consequences
of the work of Dehn, Birman, and Series that we used in \emph{(ibid.)}
will be used here.

We now follow the language of \cite{MPasympcover} to state the results
we need in this paper. These results are simple and direct consequences
of the work of Birman and Series.

We view the universal cover of $\Sigma_{g}$ as a disc tiled by $4g$-gons
that we call $U$. We assume every edge of this tiling is directed
and labeled by some element of $\{a_{1},b_{1},\ldots,a_{g},b_{g}\}$
such that when we read counter-clockwise along the boundary of any
octagon we read the reduced cyclic word $[a_{1},b_{1}]\cdots[a_{g},b_{g}]$.
By fixing a basepoint $u\in U$ we obtain a free cellular action of
$\Gamma_{g}$ on $U$ that respects the labels and directions of edges
and identifies the quotient $\Gamma_{g}\backslash U$ with $\Sigma_{g}$;
this gives a description of $\Sigma_{g}$ as a $4g$-gon with glued
sides as is typical.

Now suppose that $\gamma\in\Gamma$ is not the identity. The quotient
$A_{\gamma}\eqdf\langle\gamma\rangle\backslash U$ of $U$ by the
cyclic group generated by $\gamma$ is an open annulus tiled by infinitely
many $4g$-gons. The edges of $A_{\gamma}$ inherit directions and
labels from those of the edges of $U$. The point $u\in U$ maps to
some point denoted by $x_{0}\in A_{\gamma}$. 

Now let $w\in\F_{2g}$ be an element that represents $\gamma$, and
identify $w$ with a combinatorial word by writing $w$ in reduced
form. Beginning at $x_{0}$, and following the path spelled out by
$w$ beginning at $x_{0}$, we obtain an oriented closed loop $L_{w}$
in the one-skeleton of $A_{\gamma}$. If $w$ is a shortest element
representing the conjugacy class of $\gamma$, then this loop $L_{w}$
must not have self-intersections. In this case, that we from now assume,
$L_{w}$ is therefore a topologically embedded circle in the annulus
$A_{\gamma}$ that is non-nullhomotopic and cuts $A_{\gamma}$ into
two annuli $A_{\gamma}^{\pm}$.

Every vertex of $A_{\gamma}$ has $4g$ incident half-edges each of
which has an orientation and direction given by the edge they are
in. Going clockwise, the cyclic order of the half-edges incident at
any vertex is 

`$a_{1}$-outgoing, $b_{1}$-incoming, \textbf{$a_{1}$}-incoming,
$b_{1}$-outgoing, ... , $a_{g}$-outgoing, $b_{g}$-incoming, $a_{g}$-incoming,
$b_{g}$-outgoing'.

We define $\hat{L}_{w}$ to be the loop $L_{w}$ with all incident
half edges in $A_{\gamma}$ attached. We call the new half-edges added
\emph{hanging half-edges.}

Moreover, we thicken up $\hat{L}_{w}$ by viewing each edge of $L_{w}$
as a rectangle, each hanging half-edge as a half-rectangle, and each
vertex replaced by a disc. In other words, we take a small neighborhood
of $\hat{L}_{w}$ in $A_{\gamma}$. We now think of $\hat{L}_{w}$
as the thickened version. This is a topological annulus, where the
hanging half-edges have become stubs hanging off. A \emph{piece of
$\hat{L}_{w}$ }is a contiguous collection of hanging half-rectangles
and rectangle sides following edges of $L_{w}$ in the boundary of
$\hat{L}_{w}$. Such a piece is in either $A_{\gamma}^{+}$ or $A_{\gamma}^{-}$.
Given a piece $P$ of $\hat{L}_{w}$ we write $\e(P)$ for the number
of rectangle sides following edges of $L_{w}$, and $\he(P)$ for
the number of hanging-half edges in $P$. We say that a piece $P$
has Euler characteristic $\chi(P)=0$ if it follows an entire boundary
component of $\hat{L}_{w}$, and $\chi(P)=1$ otherwise as we view
it as an interval running along the rectangle sides and around the
sides of the hanging half-rectangles. See Figure \ref{fig:Illustration-of-a-piece}
for an illustration of a piece of $\hat{L}_{w}$.

\begin{figure}

\begin{centering}
\includegraphics[scale=1]{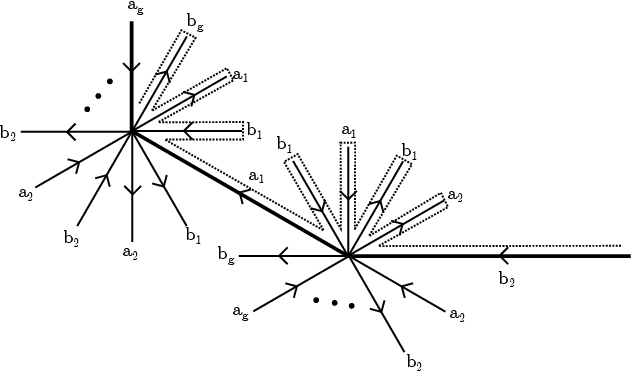}\caption{\label{fig:Illustration-of-a-piece}Illustration of a piece $P$ of
$\hat{L}_{w}$ in the case when the reduced form of $w$ contains
$a_{g}a_{1}^{-1}b_{2}^{-1}$ as a subword. The edges of $L_{w}$ are
in bold. The piece is indicated by the dotted lines. This piece $P$
has $\protect\e(P)=2$, $\protect\he(P)=7$, and $\chi(P)=1$. Note
that a piece may also run along the other side of $L_{w}$.}
\par\end{centering}
\end{figure}

Birman and Series prove in \cite[Thm. 2.12(a)]{BirmanSeries} that
if $w$ is a shortest element representing the conjugacy class of
$\gamma\in\Gamma_{g}$ then there are strong restrictions on the pieces
of $\hat{L}_{w}$ that can appear. This has the following consequence
which is given by\footnote{We stress that Lemma \ref{lem:inequality-for-pieces} is a straightforward
consequence of Birman and Series' work, so even though we cite \cite{MPasympcover},
this paper does not depend on \cite{MPasympcover} in any significant
way.} \cite[Proof of Lem. 5.18]{MPasympcover}.
\begin{lem}
\label{lem:inequality-for-pieces}If $w$ is a shortest element representing
the conjugacy class of $\gamma\in\Gamma_{g}$, and both $\gamma$
and hence $w$ are non-identity, then for any piece $P$ of $\hat{L}_{w}$,
we have 
\[
\e(P)\leq(2g-1)\he(P)+2g\chi(P).
\]
\end{lem}

\begin{proof}
Since $w$ is a shortest element representing some non-identity conjugacy
class in $\Gamma_{g}$, in the language of \cite{MPasympcover}, $L_{w}$
is a boundary reduced tiled surface. Then the proof of \cite[Lem. 5.18]{MPasympcover}
contains the result stated in the lemma. The basic idea of the proof
is not complicated and goes back to Dehn \cite{Dehn}: if there are
too many edges (i.e. $\e(P)$ is large) then one can find a string
of letters in the reduced word of $w$ (e.g. $aba^{-1}b^{-1}c$) that
can be shortened using the relator $R$ (e.g. $aba^{-1}b^{-1}c=dcd^{-1}$).
\end{proof}
This inequality plays a crucial role in the next section. 

\subsection{Proof of Proposition \ref{prop:chi-bound}\label{subsec:Proof-of-chi-bound}}

Suppose that $g\geq2$ and $w\in[\F_{2g},\F_{2g}]$ is a non-identity
shortest element representing the conjugacy class of $\gamma\in\Gamma_{g}$.
In particular, $w$ is cyclically reduced. We let $R=R_{g}.$ Now
fix $k,\ell\in\N_{0}$ and suppose $\Sigma^{*}\in\surfaces(w,k,\ell)$.
The arcs of $\Sigma^{*}$ are of three different types:
\begin{description}
\item [{WR}] An arc with one endpoint in the $w$-loop and one endpoint
in an $R$ or $R^{-1}$-loop. 
\item [{RR}] An arc with both endpoints in $R$ or $R^{-1}$ loops. By
property \textbf{P4,}\textbf{\emph{ }}the endpoints of such an arc
are both in $R$-loops or both in $R^{-1}$-loops.
\item [{WW}] An arc with both endpoints in the $w$-loop.
\end{description}
The boundary of any disc of $\Sigma^{*}$ alternates between segments
of $\partial\Sigma^{*}$ and arcs. A disc is a \emph{pre-piece disc}
if its boundary contains exactly one segment of the $w$-loop. A disc
is called a \emph{junction disc} if it is not a pre-piece disc. We
say that a junction disc is \emph{piece-adjacent }if it meets a WR-arc-side.

To be precise, we view all discs as open discs, and hence not containing
any arcs. A disc meets certain arc-sides along its boundary; it is
possible for a disc to meet both sides of the same arc and we view
this scenario as the disc meeting two separate arc-sides. We say an
arc-side has the same type WR/RR/WW as its corresponding arc.

Note that any pre-piece disc cannot meet any WW-arc-side: if it did,
the disc could only meet this one arc-side together with one segment
of the $w$-loop and this would contradict the fact that $w$ is cyclically
reduced since the arc matches a letter $f$ with a cyclically adjacent
letter $f^{-1}$ of $w$. It is also clear that any pre-piece disc
meets exactly 2 WR-arc-sides: the ones that emanate from the sole
segment of the $w$-loop. So in light of \textbf{P4} a pre-piece disc
takes one of the forms shown in Figure \ref{fig:pre-piece-disc}.

\begin{figure}
\begin{centering}
\includegraphics{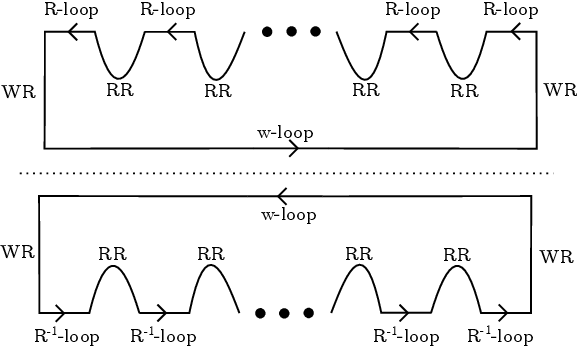}
\par\end{centering}
\caption{\label{fig:pre-piece-disc}Possible forms of pre-piece discs. The
number of $R$-loop segments or $R^{-1}$-loop segments is at least
$1$ and bounded given $k$ and $\ell$. The arrows denote the orientations
of the boundary loops.}

\end{figure}

We define a \emph{piece of} $\Sigma^{*}$ to be a connected component
of 
\[
\{\text{pre-piece discs\}}\cup\{\text{WR-arcs\}}.
\]
A piece of $\Sigma^{*}$ is therefore either a contiguous collection
of pre-piece discs that meet only along WR-arcs, or a single WR-arc.
If $P$ is a piece of $\Sigma^{*}$, either $\chi(P)=1$, or $\chi(P)=0$,
in which case $P$ meets the entire $w$-loop and is the unique piece.

We now have \textbf{two} definitions of pieces; pieces of $\hat{L}_{w}$
and pieces of $\Sigma^{*}$. These are, as the names suggest, closely
related, and this is the key observation in the proof of Proposition
\ref{prop:chi-bound}. Indeed, the reader should carefully consider
Figure \ref{fig:pp-hhe} that leads to the following lemma. In analogy
to pieces of $\hat{L}_{w}$, if $P$ is any piece of $\Sigma^{*}$,
we write $\e(P)$ for the number of WR-arcs in $P$, and $\he(P)$
for the number of RR-arc sides that meet $P$ (this is zero if $P$
is a single WR-arc).

\begin{figure}
\begin{centering}
\includegraphics[scale=0.9]{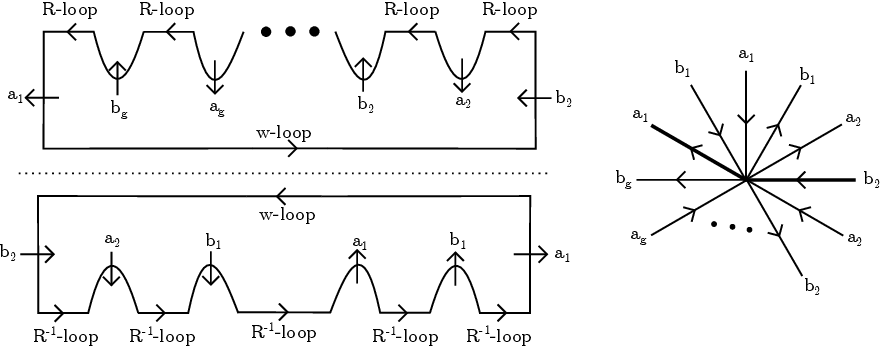}\caption{\label{fig:pp-hhe}Given a segment of the $w$-loop corresponding
to a juncture between letters $a_{1}^{-1}b_{2}^{-1}$ in $w$, if
this segment is part of a pre-piece disc then some possible forms
of that disc are shown above. This juncture between letters of $w$
corresponds to a vertex in $L_{w}$. The right hand illustration shows
the neighborhood of this vertex in the annulus $A_{\gamma}$, where
the bold lines correspond to half-edges of $L_{w}$. The right hand
picture actually almost determines the left hand pictures. Indeed,
given the $a_{1}$ arc on the top-left, the next arc has to be a $b_{g}$
arc with the given direction, since \uline{only} $b_{g}^{-1}$
cyclically precedes $a_{1}$ in $R_{g}$ or any power of $R_{g}$.
Then the next arc $a_{g}$ with its direction is determined since
\uline{only} $a_{g}$ cyclically precedes $b_{g}$ in $R_{g}$.
This continues until an arc labeled by $b_{2}$ and with an incoming
direction is reached, as in the right arc of the top-left picture.
At this point, the boundary of the disc may close up. (This is analogous
to what happens in the bottom picture, where an analogous pattern
occurs.) The only indeterminacy is that after reaching a $b_{2}$
arc with an incoming direction for the first time, the \uline{entire
pattern} shown in the right hand picture may repeat any number of
times, as long as $k$ and $\ell$ allow it. The upshot of this is
that any pre-piece disc has at least as many incident RR-arc-sides
as there are hanging half-edges on the corresponding side of $L_{w}$,
at the corresponding vertex.}
\par\end{centering}
\end{figure}

\begin{lem}
\label{lem:sigma-piece-inequality}If $w$ is a shortest element representing
the conjugacy class of $\gamma\in\Gamma_{g}$, $k,\ell\in\N_{0}$,
and $\Sigma^{*}\in\surfaces(w,k,\ell)$ then for any piece $P$ of
$\Sigma^{*}$, we have 
\[
\e(P)\leq(2g-1)\he(P)+2g\chi(P).
\]
\end{lem}

\begin{proof}
Given any piece $P$ of $\Sigma^{*}$, it contains a consecutive (possibly
cyclic) series of WR-arcs that correspond to a contiguous collection
of edges in the loop $L_{w}$. The discs of $P$ correspond to certain
vertices of $L_{w}$; each of these vertices has two emanating half-edges
belonging to the edges defined by WR-arcs of $P$. The piece $P$
can either meet only $R$-loops or meet only $R^{-1}$-loops. 

We define a piece $P'$ of $\hat{L}_{w}$ corresponding to $P$ as
follows. If $P$ meets $R$-loops, then $P'$ consists of rectangle
sides along the edges of $L_{w}$ corresponding to the WR-arcs of
$P$ together with all hanging half-edges at vertices corresponding
to discs of $P$ that are on the \uline{left} of $L_{w}$ as it
is traversed in its assigned orientation (corresponding to reading
$w$ along $L_{w}$). If $P'$ meets $R^{-1}$-loops, then $P'$ is
defined similarly with the modification that we include instead hanging
half-edges on the \uline{right} of $L_{w}$. Figure \ref{fig:pp-hhe}
together with its captioned discussion now shows that 
\[
\he(P')\leq\he(P),
\]
and $\e(P)=\e(P')$ by construction. We also have $\chi(P')=\chi(P)$.
Therefore Lemma \ref{lem:inequality-for-pieces} applied to $P'$
implies
\[
\e(P)=\e(P')\leq(2g-1)\he(P')+2g\chi(P')\leq(2g-1)\he(P)+2g\chi(P).
\]
\end{proof}
Let $N_{RR}$ be the number of RR-arcs, $N_{WR}$ the number of WR-arcs,
and $N_{WW}$ the number of WW-arcs in $\Sigma^{*}$. In the following
we refer to discs of $\Sigma^{*}$ simply as discs. Since there are
$4g(k+\ell)$ incidences between arcs and $R$-loops or $R^{-1}$
loops we have 
\begin{align}
2N_{RR}+N_{WR} & =4g(k+\ell).\label{eq:arcs-to-R-length}
\end{align}
Let $\Sigma_{1}$ be the surface formed by cutting $\Sigma^{*}$ along
all RR-arcs. We have
\[
\chi(\Sigma_{1})=\sum_{\text{discs}\,D}\left(1-\frac{d'(D)}{2}\right)
\]
where $d'(D)$ is the number of arc-sides meeting $D$ that are \uline{not}
of type RR. This formula holds because $d'(D)$ is the degree of the
disc $D$ in the dual graph $G_{1}$ of $\Sigma_{1}$, the right hand
side is easily seen to be $\chi(G_{1})=V(G_{1})-E(G_{1})$, and since
$\Sigma_{1}$ deformation retracts to an obvious embedded copy of
$G_{1}$, $\chi(G_{1})=\chi(\Sigma_{1})$. We partition the sum above
according to
\begin{align*}
\chi(\Sigma_{1})= & S_{0}+S_{1}+S_{2},\\
S_{0}\eqdf & \sum_{\text{pre-piece discs \ensuremath{D}}}\left(1-\frac{d'(D)}{2}\right),\\
S_{1}\eqdf & \sum_{\text{piece-adjacent junction discs}\,D}\left(1-\frac{d'(D)}{2}\right),\\
S_{2}\eqdf & \sum_{\text{not piece-adjacent junction discs}\,D}\left(1-\frac{d'(D)}{2}\right).
\end{align*}
Note first that a pre-piece disc has $d'(D)=2$ (cf. Fig \ref{fig:pre-piece-disc}).
Hence $S_{0}=0$. We deal with $S_{1}$ next. For a disc $D$ of $\Sigma^{*}$,
let $d_{WR}(D)$ denote the number of WR-arc-sides meeting $D$. Note
that a piece-adjacent junction disc $D$ has $d_{WR}(D)>0$ by definition.
We rewrite $S_{1}$ as

\begin{align}
S_{1}= & \sum_{\text{piece-adjacent junction discs}\,D}\left(1-\frac{d'(D)}{2}\right)\frac{1}{d_{WR}(D)}\nonumber \\
 & \sum_{\text{incidences between \ensuremath{D} and WR-arc-sides}}1\nonumber \\
= & \sum_{\text{pieces \ensuremath{P}}}\sum_{\text{incidences between \ensuremath{P} and some junction disc \ensuremath{D} along WR-arc}}Q(D)\label{eq:sum-to-sum-ofquotients}
\end{align}
where for a piece-adjacent junction disc $D$
\[
Q(D)\eqdf\frac{1}{d_{WR}(D)}\left(1-\frac{d'(D)}{2}\right).
\]

Suppose that $D$ is a piece-adjacent junction disc. By parity considerations,
$d_{WR}(D)$ is even. We estimate $Q(D)$ by splitting into two cases.
If $d_{WR}(D)=2$ then $d'(D)\geq3$ since otherwise, $D$ would meet
only 2 WR arc-sides and other RR arc-sides, hence be a pre-piece disc
and not be a junction disc. In this case
\[
Q(D)=\frac{1}{2}\left(1-\frac{d'(D)}{2}\right)\leq\frac{1}{2}\left(1-\frac{3}{2}\right)=-\frac{1}{4}.
\]
Otherwise, $d_{WR}(D)\geq4$ and since $d'(D)\geq d_{WR}(D)$, we
have 
\[
Q(D)\leq\frac{1}{d_{WR}(D)}\left(1-\frac{d_{WR}(D)}{2}\right)=\frac{1}{d_{WR}(D)}-\frac{1}{2}\leq\frac{1}{4}-\frac{1}{2}=-\frac{1}{4}.
\]
So we have proved that for all piece-adjacent junction discs $D$,
$Q(D)\leq-\frac{1}{4}$. Putting this into (\ref{eq:sum-to-sum-ofquotients})
gives
\begin{align}
S_{1} & \leq-\frac{1}{4}\sum_{\text{pieces \ensuremath{P}}}\sum_{\text{incidences between \ensuremath{P} and some junction disc \ensuremath{D} along WR-arc}}1\nonumber \\
 & =-\frac{1}{4}\sum_{\text{pieces \ensuremath{P}}}2\chi(P)=-\frac{1}{2}\sum_{\text{pieces \ensuremath{P}}}\chi(P).\label{eq:S1bound}
\end{align}

We now turn to $S_{2}$. \emph{Here is the key moment where $w\neq\id$
is used}\footnote{Although technically, $w\neq\id$ was used to define $L_{w}$ and
pieces etc, if $w$ is the identity the proof of Proposition \ref{prop:chi-bound}
could, a priori, circumvent these definitions.}. Since $w\neq\id$, any disc must meet an arc. Indeed, the only other
possibility is that the boundary of the disc is an entire boundary
loop that has no emanating arcs. This hypothetical boundary loop cannot
be an $R$ or $R^{-1}$-loop, so it has to be the $w$-loop. But this
would entail $w=\id$.

Hence any disc contributing to $S_{2}$ meets no WR-arc-side, but
meets some arc-side. Therefore it meets only WW-arcs or only RR-arcs.
Every disc $D$ contributing to $S_{2}$ meeting only WW-arcs gives
a non-positive contribution since $w$ is cyclically reduced hence
$d'(D)\geq2$. Every disc $D$ contributing to $S_{2}$ meeting only
RR-arcs, which we will call an \emph{RR-disc}, has $d'(D)=0$ and
hence contributes $1$ to $S_{2}$. 

This shows
\begin{equation}
S_{2}\leq\#\{\text{RR-discs}\}.\label{eq:S2-bound}
\end{equation}
In total combining $S_{0}=0$ with (\ref{eq:S1bound}) and (\ref{eq:S2-bound})
we get
\[
\chi(\Sigma_{1})\leq\#\{\text{RR-discs}\}-\frac{1}{2}\sum_{\text{pieces \ensuremath{P} of \ensuremath{\Sigma^{*}}}}\chi(P).
\]
To obtain $\Sigma^{*}$ from $\Sigma_{1}$ we have to glue all cut
RR-arcs, of which there are $N_{RR}$. Each gluing decreases $\chi$
by 1 so
\[
\chi(\Sigma^{*})\leq\#\{\text{RR-discs}\}-N_{RR}-\frac{1}{2}\sum_{\text{pieces \ensuremath{P} of \ensuremath{\Sigma^{*}}}}\chi(P).
\]
Using Lemma \ref{lem:sigma-piece-inequality} with the above gives
\begin{align}
\chi(\Sigma^{*})\leq & \#\{\text{RR-discs}\}-N_{RR}-\frac{1}{2}\sum_{\text{pieces \ensuremath{P} of \ensuremath{\Sigma^{*}}}}\chi(P)\nonumber \\
\leq & \#\{\text{RR-discs}\}-N_{RR}-\frac{1}{4g}\sum_{\text{pieces \ensuremath{P} of \ensuremath{\Sigma^{*}}}}\e(P)\nonumber \\
 & +\frac{(2g-1)}{4g}\sum_{\text{pieces \ensuremath{P} of \ensuremath{\Sigma^{*}}}}\he(P)\\
= & \#\{\text{RR-discs}\}-N_{RR}-\frac{N_{WR}}{4g}+\frac{(2g-1)}{4g}\sum_{\text{pieces \ensuremath{P} of \ensuremath{\Sigma^{*}}}}\he(P).\label{eq:chisigmastar1}
\end{align}

Let $\he'(\Sigma^{*})$ denote the total number of RR-arc-sides meeting
RR-discs. Every RR-disc has to meet at least $4g$ arc-sides; this
observation is similar to the reasoning in Figure \ref{fig:pp-hhe}.
Therefore
\begin{equation}
\he'(\Sigma^{*})\geq4g\#\{\text{RR-discs}\}.\label{eq:hedash1}
\end{equation}
 Every RR-arc-side either meets a piece $P$ and contributes to $\he(P)$
or a disc meeting only RR-arc-sides and contributes to $\he'(\Sigma^{*})$.
Hence
\begin{equation}
\he'(\Sigma^{*})+\sum_{\text{pieces \ensuremath{P} of \ensuremath{\Sigma^{*}}}}\he(P)=2N_{RR}.\label{eq:hedash2}
\end{equation}
Combining (\ref{eq:arcs-to-R-length}), (\ref{eq:hedash1}), and (\ref{eq:hedash2})
with (\ref{eq:chisigmastar1}) gives
\begin{align*}
\chi(\Sigma^{*}) & \stackrel{\eqref{eq:hedash1}}{\leq}\frac{\he'(\Sigma^{*})}{4g}-N_{RR}-\frac{N_{WR}}{4g}+\frac{(2g-1)}{4g}\sum_{\text{pieces \ensuremath{P} of \ensuremath{\Sigma^{*}}}}\he(P)\\
 & \stackrel{\eqref{eq:hedash2}}{=}\frac{\he'(\Sigma^{*})}{4g}-N_{RR}-\frac{N_{WR}}{4g}+\frac{(2g-1)N_{RR}}{2g}-\frac{(2g-1)}{4g}\he'(\Sigma^{*})\\
 & =-\frac{1}{4g}\left(2N_{RR}+N_{WR}\right)-\frac{2g-2}{4g}\he'(\Sigma^{*})\\
 & \leq-\frac{1}{4g}\left(2N_{RR}+N_{WR}\right)\stackrel{\eqref{eq:arcs-to-R-length}}{=}-\frac{4g(k+\ell)}{4g}=-(k+\ell).
\end{align*}
This completes the proof of Proposition \ref{prop:chi-bound}. $\square$

\section{Proof of main theorem\label{sec:Proof-of-main}}

\subsection{Proof of Theorem \ref{thm:main-convergence}\label{subsec:Proof-of-Theorem}}
\begin{proof}[Proof of Theorem \ref{thm:main-convergence}]
Assume $\gamma\in[\Gamma_{g},\Gamma_{g}]$ is not the identity and
that $w\in[\F_{2g},\F_{2g}]$ is a shortest element representing the
conjugacy class of $\gamma$, hence also not the identity. By Corollary
\ref{cor:starting-point} we have
\[
\E_{g,n}[\tr_{\gamma}]=\zeta(2g-2;n)^{-1}\sum_{\substack{\text{\text{(\ensuremath{\mu,\nu)\in}\ensuremath{\tilde{\Omega}}}}}
}D_{\mu,\nu}(n)\mathcal{J}_{n}(w,\nu,\mu)+O_{w,g}\left(\frac{1}{n}\right),
\]
where $\tilde{\Omega}$ is a finite collection of pairs of Young diagrams.
We know \linebreak{}
$\lim_{n\to\infty}\zeta(2g-2;n)=1$ from (\ref{eq:witten-limit})
and for each fixed $(\mu,\nu)$, \linebreak{}
$D_{\mu,\nu}(n)\J_{n}(w,\nu,\mu)=D_{\nu,\mu}(n)\J_{n}(w,\nu,\mu)=O_{w,\mu,\nu}(1)$
by Theorem \ref{thm:main-single-small-dim}. Hence $\E_{g,n}[\tr_{\gamma}]=O_{\gamma}(1)$
as $n\to\infty$ as required.
\end{proof}
\bibliographystyle{gtart}

\end{document}